\documentclass[11pt,thmsa]{amsart}

\usepackage{amssymb,latexsym}
\usepackage{amsmath,amscd}
\usepackage[pdftex,all]{xy}
\usepackage{color}
\usepackage{hyperref}
\usepackage{esvect}
\usepackage{framed}
\usepackage{comment}
\usepackage{shuffle}
\usepackage{amsmath}
\usepackage{amsthm}
\usepackage{amsfonts}
\usepackage{amssymb}
\usepackage[pdftex]{graphicx}
\usepackage{wrapfig}
\usepackage{layout}
\usepackage{verbatim}
\usepackage{alltt}
\usepackage{ mathrsfs }
\usepackage{tabularx}

\newtheorem{theorem}{Theorem}[section]
\newtheorem{lemma}[theorem]{Lemma}
\newtheorem{proposition}[theorem]{Proposition}

\theoremstyle{definition}
\newtheorem{definition}[theorem]{Definition}

\newtheorem{remark}[theorem]{Remark}

\newtheorem*{acknowledgements}{Acknowledgements}

\newcommand{\id}{\mathrm{id}}

\newcommand{\Coder}{\mathrm{Coder}}
\newcommand{\End}{\mathrm{End}}

\newcommand{\K}{\mathbb{K}}

\def\aaltbin#1#2{\ensuremath{\left(\kern-.35em\left(\genfrac{}{}{0pt}{}{#1}{#2}\right)\kern-.35em\right)} }

\newcommand{\Hom}{\operatorname{Hom}}

\begin{document}
	
	\title{Cumulants, Koszul brackets and homological perturbation theory for commutative $BV_\infty$ and $IBL_\infty$ algebras}
\title[Cumulants, Koszul brackets and homological perturbation theory ...]
{Cumulants, Koszul brackets and homological perturbation theory for commutative $BV_\infty$ and $IBL_\infty$ algebras}
\author{Ruggero Bandiera}
\address{Universit\`{a} degli studi di Roma ``La Sapienza''}
\email{bandiera@mat.uniroma1.it}
	\maketitle
\begin{abstract}
	We explore the relationship between the classical constructions of cumulants and Koszul brackets, showing that the former are an expontial version of the latter. Moreover, under some additional technical assumptions, we prove that both constructions are compatible with standard homological perturbation theory in an appropriate sense. As an application of these results, we provide new proofs for the homotopy transfer Theorem for $L_\infty$ and $IBL_\infty$ algebras based on the symmetrized tensor trick and the standard perturbation Lemma, as in the usual approach for $A_\infty$ algebras. Moreover, we prove a homotopy transfer Theorem for commutative $BV_\infty$ algebras in the sense of Kravchenko which appears to be new. Along the way, we introduce a new definition of morphism between commutative $BV_\infty$ algebras. \end{abstract}
\tableofcontents	
	\renewcommand{\1}{\mathbf{1}}
	\section*{Introduction}
	
Given a graded commutative algebra $A$ and an endomorphism $\delta:A\to A$, the \emph{Koszul brackets} $\mathcal{K}(\delta)_n:A^{\odot n}\to A$ are graded symmetric maps measuring the deviation of $\delta$ from being a derivation. These were introduced by J. L. Koszul \cite{Kos} in the algebraic study of differential operators, and have been applied to a variety of situations, from the study of the BV formalism in mathematical physics \cite{Alf,BDA,Bat0,Bat1,Bat2,Ber} to Poisson geometry \cite{Kos,KosSchw,FMKos,BMcois}. Their algebraic properties have been extensively studied, see for instance \cite{Markl1,Markl2,MKos1,MKos2}, and several generalizations to non-commutative settings have been investigated: see \cite{Ber,Bjor,Markl1,Bder,MKos2} for generalizations in the associative setting and \cite{BKap} for a generalization in the pre-Lie setting. See also \cite{Vor1,Vor2} for a closely related construction of higher (derived) brackets.

Given graded commutative algebras $A, B$ and a degree zero map $f:A\to B$, the (classical) \emph{cumulants} $\kappa(f)_n:A^{\odot n}\to B$ are graded symmetric maps measuring the deviation of $f$ from being a morphism of graded algebras. Certain non-commutative generalizations have been considered as well, see \cite{Lehn,Park1,Park2,Ranade1}. These play an important role in probability theory, as well as some of its non-classical variations (such as non-commutative probability theory \cite{Lehn} or homotopy probability theory \cite{Park1,Park2,ParkHP}). In their algebraic formulation considered here, they have also been also applied in the study of homotopy algebras and homotopy morphisms \cite{Ranade1,Ranade2}, as they give obstructions for a linear map between $A_\infty$ or $C_\infty$ algebras to be the linear part of an $\infty$-morphism.
 
In Section \ref{sec:cumulantsandkoszul} we explore the relation between (classical) cumulants and Koszul brackets. Given a graded commutative algebra $A$, we shall denote by $S(A)$ the symmetric coalgebra over $A$, and by $E:S(A)\to S(A)$ the coalgebra morphism which corestricts to $pE:S(A)\to A:a_1\odot\cdots\odot a_n\to a_1\cdots a_n$. Since $E$ has linear part the identity, it is an automorphism of $S(A)$, whose inverse we denote by $L$. It has been observed in several places that given $\delta:A\to A$, its Koszul brackets $\mathcal{K}(\delta)_n:A^{\odot n}\to A$ might be characterized as the Taylor coefficients of the coderivation $\mathcal{K}(\delta):=L\circ \widetilde{\delta}\circ E:S(A)\to S(A)$, where $\widetilde{\delta}:S(A)\to S(A)$ is the linear coderivation extending $\delta$, see for instance \cite{Markl1,Markl2,SulRan,BKap}. Moreover, a similar description holds for the cumulants of a map $f:A\to B$: they are the Taylor coefficients of the coalgebra morphism  $\kappa(f):=L\circ S(f)\circ E:S(A)\to S(B)$, as was for instance observed in \cite{SulRan}. This common description is the starting point of our analysis. From it, one easily derives the basic algebraic properties of cumulants and Koszul brackets, as well as the compatibility between the two constructions, see for instance Propositions \ref{prop:cumulantsvsKoszuldg} and \ref{prop:cumulantsvsKoszulexp}: in particular, the latter shows that cumulants should be considered as an exponential version of the Koszul brackets, shedding further light on the relation between the two constructions.

In Section \ref{subsec:comp}, we investigate the behavior of cumulants and Koszul brackets with respect to homological perturbation theory. Given a differential $d_A$ on $A$ (not necessarily an algebra derivation), it is well known that the associated Koszul brackets $\mathcal{K}(d_A)_n$ satisfy higher Jacobi relations, thus inducing a structure of $L_\infty[1]$ algebra on $A$ (or in other words, a structure of strong homotopy Lie algebra on $A[-1]$, in the sense of \cite{LaSt}). Given a second graded algebra $B$ with a differential (not necessarily an algebra derivation) $d_B$, together with a contraction $\sigma:A\to B$, $\tau:B\to A$, $h:A\to A[-1]$ of $A$ onto $B$, by the usual homotopy transfer Theorem for $L_\infty[1]$ algebras (this will be reviewd in \ref{th:transfer}, and then again in Section \ref{subsec:Lootransfer}), there is an induced $L_\infty[1]$ algebra structure on $B$ and an induced $L_\infty[1]$ morphism $B\to A$. When the contraction satisfies certain additional technical assumptions (we say that it is a \emph{semifull algebra contraction}, see Definition \ref{def:semifullalgebracontraction}), we show in Proposition \ref{prop:transfer} that these are precisely the $L_\infty[1]$ algebra structure on $B$ associated with the Koszul brackets $\mathcal{K}(d_B)_n$ and the $L_\infty[1]$ morphism $B\to A$ associated with the cumulants $\kappa(\tau)_n$. This result will be one of our main technical tools in the proof of the homotopy transfer Theorems \ref{th:transferL}, \ref{th:transferBV} and \ref{th:transferIBL}. Semifull algebra contractions, as well as the dual class of semifull coalgebra contractions,  were introduced by Real \cite{Real}: for our purposes, one of their more useful features will be the fact that they are closed under the Standard Perturbation Lemma \ref{lem:SPL}, see Lemmas \ref{lem:semifullstable} and \ref{lem:semifullcostable}. 

In Section \ref{sec:coetcet} we introduce the dual notions of \emph{cocumulants} and \emph{Koszul cobrackets}. Namely, given a map $f:C\to D$ between graded cocommutative coalgebras, its cocumulants are maps $\kappa^{co}(f)_n:C\to D^{\odot n}$ measuring the deviation of $f$ from being a coalgebra morphism. Given a map $\delta:C\to C$, the Koszul cobrackets $\mathcal{K}(\delta)_n:C\to C^{\odot n}$ are maps 
measuring the deviation of $\delta$ from being a coderivation. All the results from Sections \ref{sec:cumulantsandkoszul} and \ref{subsec:comp} admit analogues in this dual setting.

Finally, in Section \ref{subsec:Lootransfer} we revisit the homotopy transfer Theorem for $L_\infty[1]$ algebras, providing a new proof in light of the results from the previous sections. Our approach is based on the symmetrized tensor trick and the standard perturbation Lemma, in close analogy with the usual argument for $A_\infty[1]$ algebras, see for instance \cite{Hueb2}. 
We point out that if we were only interested in homotopy transfer for $L_\infty[1]$ algebras a similar but simpler argument might be given  \cite{prep}. For our purposes, this section serves more as a preparation for the similar proofs of the homotopy transfer Theorems \ref{th:transferBV}, \ref{th:transfercoBV} and \ref{th:transferIBL} for $BV_\infty$ (co)algebras and $IBL_\infty$ algebras. Moreover, the technical lemmas we establish along the way seem interesting in their own right. In Lemma \ref{lem:coalgebramorphisms} we give necessary and sufficient conditions for a map $F:S(U)\to S(V)$ to be a coalgebra morphism in terms of the associated cumulants $\kappa(F)_n$, which is curious since by construction the latter measure the deviation of $F$ from being a morphism of algebras, rather than coalgebras. Similarly, in Lemma \ref{lem:transkos} we give necessary and sufficient conditions for a map $Q:S(U)\to S(U)$ to be a coderivation in terms of the associated Koszul brackets $\mathcal{K}(Q)_n$ (an analogous characterization of codifferential operators shall be given later, in Lemma \ref{lem:IBLtransKos}). Finally, in Lemma \ref{lem:symmetrizedcontraction} we make the key observation that the contraction of $S(U)$ onto $S(V)$ induced by a contraction of $U$ onto $V$ via the symmetrized tensor trick is a semifull contraction with respect to both the algebraic and the coalgebraic structures.

\emph{Commutative $BV_\infty$ algebras} were introduced by O. Kravchenko \cite{Krav} and have been applied in several contexts, such as deformation quantization, quantum field theory and Poisson geometry, just to name a few, see for instance \cite{Hueb1,CL,DSV1,DSV2,BrLaz,Vit,BashVor,Campos,ParkQFT,LRS,AVor} (in the references \cite{ParkQFT,LRS} they are called \emph{binary QFT algebras}, but in fact the two notions seem to be equivalent).

More precisely, a commutative $BV_\infty$ algebra is a graded algebra $A$ together with a $\K[[t]]$ linear differential $\Delta=\sum_{n\ge0}t^n\Delta_n:A[[t]]\to A[[t]]$ on the algebra $A[[t]]$ of formal power series: moreover, one requires that $\Delta_n$ is a differential operator of order $\le n+1$, $\forall \,n\ge0$. The last condition might be compactly rephrased in terms of the associated Koszul brackets $\mathcal{K}(\Delta)_n$: it is equivalent to requiring $\mathcal{K}(\Delta_i)_n=0$ for all $i<n-1$, or in other words, it is equivalent to 
\begin{equation}\label{intro1}  \mathcal{K}(\Delta)_n\equiv0\pmod{t^{n-1}},\qquad\forall n\ge2. 
\end{equation}

We should remark that commutative $BV_\infty$ algebras in the above sense are not homotopy $BV$ algebras in the full operadic sense \cite{GCTV}, as we maintain strict associativity of the commutative product on $A$, and we are only relaxing the conditions on the $BV$ operator up to coherent homotopies. For this reason, the name $BV_\infty$ algebra might be misleading, as the usual machinery for  $\infty$-algebras (see for instance \cite[\S 10]{LV}), such as the homotopy transfer Theorem, doesn't apply to this kind of structures without further hypotheses. A similar problem is shared for instance by $P_\infty$ algebras in the sense of Cattaneo and Felder \cite{CF}, which are a type of homotopy Poisson algebras where we relax only the Lie algebra structure up to homotopy, while maintaining strict associativity of the commutative product. To avoid this ambiguity, in the paper \cite{BCSX} we proposed the name \emph{derived Poisson algebras} for this kind of algebras: in line with this choice, in the body of the paper we shall call \emph{derived BV algebras} the commutative $BV_\infty$ algebras in the sense of Kravchenko (in the rest of the introduction, however, we shall stick to the more usual terminology). 

In Section \ref{sec:BVmor} we introduce morphisms between commutative $BV_\infty$ algebras. Namely, given commutative $BV_\infty$ algebras $(A,\Delta)$ and $(B,\Delta')$, a morphism of commutative $BV_\infty$ algebras from $A$ to $B$ is a degree zero $\K[[t]]$-linear map $f=\sum_{n\ge0} t^nf_n:A[[t]]\to B[[t]]$ such that $f\circ\Delta=\Delta'\circ f$, and moreover the cumulants satisfy
\begin{equation}\label{intro2} \kappa(f)_n\equiv0\pmod{t^{n-1}},\qquad\forall n\ge2.
\end{equation}
We provide some justification for this definition, and in particular we show that with the above morphisms commutative $BV_\infty$ algebras form indeed a category. More convincingly, we show that the above Equation \eqref{intro2} is in fact an exponential version of \eqref{intro1}. Moreover, we extend the usual correspondence from commutative $BV_\infty$ algebras to $L_\infty$ algebras (more precisely, to derived Poisson algebras as in \cite{BCSX}) to a full-fledged functor. After a first draft of this paper was ready, we learned that the same notion was considered by J.-S. Park \cite{ParkQFT}. Another definition of morphism was introduced by Cieliebak and Latschev \cite{CL} (and further investigated in \cite{BashVor}), but only in the special case when the source algebra is free. We conclude the section by showing that our definition is essentially equivalent with the one from \cite{CL}, in the situations when the latter applies. More precisely, in Proposition \ref{prop:morphismscomparison} we show that morphisms in our sense and in the sense of \cite{CL} are in bijective correspondence with each other via the exponential and the logarithm in a certain convolution algebra. 

In Section \ref{sec:BVtrans} we prove that commutative $BV_\infty$ algebra structures can be transferred along semifull algebra contractions via homological perturbation theory. More precisely, given a contraction $(\sigma,\tau,h)$ of $A$ onto $B$ and a  commutative $BV_\infty$ algebra structure $\Delta=\sum_{n\ge0}t^n\Delta_n$ on $A$, we can apply the standard perturbation Lemma to the perturbation $\Delta_+=\sum_{n\ge1}t^n\Delta_n$ in order to get a perturbed differential $\Delta'$ on $B[[t]]$ and a perturbed contraction $(\breve{\sigma},\breve{\tau},\breve{h})$ of $(A[[t]],\Delta)$ onto $(B[[t]],\Delta')$. In the above situation, we prove in Theorem \ref{th:transferBV} that if $(\sigma,\tau,h)$ is a semifull algebra contraction, then $\Delta'$ is a commutative $BV_\infty$ algebra structure on $B$ and $\breve{\tau}:B[[t]]\to A[[t]]$ is a morphism of commutative $BV_\infty$ algebras.

In Section \ref{sec:BVco} we consider the dual category of $BV_\infty$ cocommutative coalgebras, and extend the previous results to this dual setting.

\emph{$IBL_\infty$ algebras} (short for Involutive Lie Bialgebras up to coherent homotopies) were introduced in the paper \cite{Fuk}, with applications to string topology, symplectic field theory and Lagrangian Floer theory, and have been further investigated in several other papers since then, for instance \cite{Campos,MarklVor,DJP,Haj1,Merk,HLV,NWill,CV,Haj}

We shall work with a definition of $IBL_\infty$ algebra slightly different from, and in a certain sense dual to, the one usually appearing in the literature. More precisely, whereas $IBL_\infty$ algebras are usually regarded as commutative $BV_\infty$ algebras whose underlying algebras are free (see for instance \cite{CL,Campos,MarklVor}), we shall regard them dually as cocommutative $BV_\infty$ coalgebras whose underlying coalgebras are cofree. In Remark \ref{rem:IBL} we shall compare the two definitions, showing  that they are essentially equivalent, aside from some minor differences regarding the degrees of the structure maps. For our purposes, our definition presents some technical advantages: we shall call $IBL_\infty[1]$ algebras the $IBL_\infty$ algebras in our sense. We also establish necessary and sufficient conditions for a map $\delta:S(U)[[t]]\to S(U)[[t]]$ to define an $IBL_\infty[1]$ algebra structure on $U$ in terms of the associated Koszul brackets, see Lemma \ref{lem:IBLtransKos}.

In Section \ref{sec:IBLtrans}, as an application of the results from the previous sections, we shall present a new proof of the homotopy transfer Theorem for $IBL_\infty[1]$ algebras based on homological perturbation theory (different proofs can be found in \cite{Fuk,HLV}). More precisely, given a contraction $(g,f,h)$ of $(U,d_U)$ onto $(V,d_V)$ there is, via the symmetrized tensor trick, an induced contraction of $(S(U)[[t]],d_U)$ onto $(S(V)[[t]],d_V)$. Given $\delta:S(U)[[t]]\to S(U)[[t]]$ an $IBL_\infty[1]$ algebra structure on $U$, we can apply the standard perturbation Lemma to the perturbation $\delta_+:=\delta-d_U$ in order to get a perturbed differential $\delta'$ on $S(V)[[t]]$ and a perturbed contraction $(G,F,H)$ of $(S(U)[[t]],\delta)$ onto $(S(V)[[t]],\delta')$. In the above situation, in Theorem \ref{th:transferIBL} we shall prove that $\delta'$ is an $IBL_\infty[1]$ algebra structure on $V$, and that $G:S(U)[[t]]\to S(V)[[t]]$ and $F:S(V)[[t]]\to S(U)[[t]]$ are $IBL_\infty[1]$ morphisms. 

Finally, we shall consider the behavior of Maurer-Cartan sets under homotopy transfer. In the $L_\infty$ case, this is the content of a formal analog of Kuranishi's Theorem due to Fukaya \cite{Fuk00} and Getzler \cite{GetzlerLie,Getzlerpert}, which will be reviewed in Theorem \ref{th:kuranishi}. With a little more work, we obtain analogue results in the $BV_\infty$ setting (Theorem \ref{th:BVKur}) and $IBL_\infty$ setting (Theorem \ref{th:IBLKur}).

\begin{acknowledgements} It is a pleasure to thank Marco Manetti, Niels Kowalzig, Hsuan-Yi Liao and Luca Vitagliano for useful discussions on the subject of this paper.
\end{acknowledgements}

	
	
	\section{(Co)cumulants and Koszul (co)brackets}
	
	\newcommand{\G}{\widehat{\mathbf{G}}}		
	\newcommand{\DG}{\widehat{\mathbf{DG}}}	
	\newcommand{\GAlg}{\widehat{\mathbf{GAlg}}}

	\subsection{Review of $L_\infty[1]$ algebras}

We shall work over a fixed field $\mathbb{K}$ of characteristic zero.

Given $V=\oplus_{i\in\mathbb{Z}} V^i$ a graded space, we denote by $V^{\odot n}$ the symmetric powers of $V$ ($V^{\odot0}:=\mathbb{K}$), and by $S(V) = \bigoplus_{n\ge0} V^{\odot n}$ the symmetric coalgebra over $V$. The coproduct $\Delta$ on $S(V)$ is explicitly given by
\[ \Delta(v_1\odot\cdots\odot v_n) =\sum_{i=0}^n\sum_{\sigma\in S(i,n-i)}\pm_K (v_{\sigma(1)}\odot\cdots\odot v_{\sigma(i)})\otimes(v_{\sigma(i+1)}\odot\cdots\odot v_{\sigma(n)}),  \]
where we denote by $S(i,n-i)$ the set of $(i,n-i)$-unshuffles and by $\pm_K$ the appropriate Koszul sign. As well known, $S(V)$ is the cofree object cogenerated by $V$ in an appropriate category of coalgebras. In particular, given graded spaces $V$ and $W$, a coalgebra morphism $F:S(V)\to S(W)$ is completely determined by its corestriction \[f=pF=(f_1,\ldots,f_n,\ldots):S(V)\to W,\qquad f_n:V^{\odot n}\to W,\] ($p:S(W)\to S(W)$ being the natural projection) via the formula 
\begin{equation}\label{morfromtaylor}
 F(v_1\odot\cdots\odot v_n) = \sum_{\stackrel{k,i_1,\ldots,i_k\ge1}{i_1+\cdots+i_k=n}}\frac{1}{k!}\sum_{\sigma\in S(i_1,\ldots,i_k)} \pm_K f_{i_1}(v_{\sigma(1)},\ldots)\odot\cdots\odot f_{i_k}(\ldots, v_{\sigma(n)}),
\end{equation}
where $\pm_K$ is the appropriate Koszul sign
(in particular, for $n=0$ this becomes $F(\mathbf{1}_{S(V)})=\mathbf{1}_{S(W)}$). The maps $f_n:V^{\odot n}\to W$, $n\ge1$, are called the \emph{Taylor coefficients} of $F$. 

In a similar manner, every coderivation $Q:S(V)\to S(V)$ is completely determined by its corestriction
\[q=pQ=(q_0,q_1,\ldots,q_n,\ldots):S(V)\to V,\qquad q_n:V^{\odot n}\to V,\] via the formula 
\begin{equation}\label{coderfromtaylor}
	 Q(v_1\odot\cdots\odot v_n) = \sum_{i=0}^n\sum_{\sigma\in S(i,n-i)} \pm_K q_{i}(v_{\sigma(1)},\ldots, v_{\sigma(i)})\odot v_{\sigma(i+i)}\odot\cdots\odot v_{\sigma(n)}, \end{equation}
where $\pm_K$ is the appropriate Koszul sign
(in particular, for $n=0$ this becomes $Q(\mathbf{1}_{S(V)})=q_0(1)$). The maps $q_n:V^{\odot n}\to W$, $n\ge0$, are called once again the \emph{Taylor coefficients} of $Q$. 

We shall denote by $\Coder(S(V))$ the graded Lie algebra of coderivations of $S(V)$, with the commutator bracket. In Taylor coefficients, given coderivations $q=(q_0,q_1,\ldots,q_n,\ldots)$ and $r=(r_0,r_1,\ldots,r_n,\ldots)$, their bracket is $[q,r]=([q,r]_0,[q,r]_1,\ldots,[q,r]_n,\ldots)$
\begin{multline*} [q,r]_n(v_1,\ldots,v_n) = \sum_{i=0}^n\sum_{\sigma\in S(i,n-i)} \pm_K \Big( q_{n-i+1}\left(r_{i}\left(v_{\sigma(1)},\ldots, v_{\sigma(i)}\right), v_{\sigma(i+1)},\ldots, v_{\sigma(n)}\right)\\-(-1)^{|q||r|}r_{n-i+1}\left(q_{i}\left(v_{\sigma(1)},\ldots, v_{\sigma(i)}\right), v_{\sigma(i+1)},\ldots, v_{\sigma(n)}\right) \Big).
\end{multline*}  

\begin{definition} An $L_\infty[1]$ algebra structure on a graded space $V$ is a degree one coderivation $Q\in\Coder^1(S(V))$ such that $Q(\mathbf{1}_{S(V)})=\mathbf{1}_{S(V)}$ (i.e., $q_0=0$) and $Q^2=\frac{1}{2}[Q,Q]=0$. Given $L_\infty[1]$ algebras $(V,q_1,\ldots,q_n,\ldots)$ and $(W,r_1,\ldots,r_n,\ldots)$, an $L_\infty[1]$ morphism between them is a morphism of DG coalgebras $F:(S(V),Q)\to (S(W),R)$.
\end{definition}

\begin{remark} Given an $L_\infty[1]$ algebra $(V,q_1,\ldots,q_n,\ldots)$, the relation $[Q,Q]=0$ translates into a hierarchy of relations on the Taylor coefficients $q_1,\ldots,q_n,\ldots$. The first few relations say that $q_1$ squares to zero and satisfies the Leibniz identity with respect to the bracket $[x,y]:=(-1)^{|x|}q_2(x, y)$. Furthermore, this bracket satisfies the Jacobi identity up to a homotopy given by $q_3$. The higher Taylor coefficients $q_n$ give coherent homotopies for higher Jacobi relations. In particular, the shifted cohomology $H(V,q_1)[-1]$ is a graded Lie algebra via the induced bracket (in fact, it carries additional algebraic structure - a minimal $L_\infty$ algebra structure - corresponding topologically to higher Whitehead brackets). Conversely, any DG Lie algebra $(L,d_L,[-,-])$ gives rise to an $L_\infty[1]$ algebra $(L[1],Q)$ with $q_1(x)=-d_L(x)$, $q_2(x,y)=(-1)^{|x|}[x,y]$ and $q_n=0$ for $n>2$. More in general, $L_\infty[1]$ algebra structures on $V$ correspond to strong homotopy Lie algebra structures (as in \cite{LaSt}) on the suspension $V[-1]$.

In a similar way, given an $L_\infty[1]$ morphism $F:(V,q_1,\ldots,q_n,\ldots)\to(W,r_1,\ldots,r_n,\ldots)$, the identity $FQ=RF$ translates into a hierarchy of relations on the Taylor coefficients of $F,Q,R$. The first few relations say that $f_1:(V,q_1)\to (W,r_1)$ is a morphism of complexes, compatible with the brackets (induced by) $q_2$ and $r_2$ up to the homotopy $f_2$. 
	
\end{remark}

$L_\infty[1]$ algebras are homotopy invariant algebraic structures, meaning that they can be transferred along quasi-isomorphisms. In case the quasi-isomorphism fits into a contraction (see Definition \ref{def:contraction}), the transfer can be made explicit via symmetrized tensor trick and homological perturbation theory, which is the content of the following homotopy transfer Theorem. 

\begin{theorem}\label{th:transfer} Given an $L_\infty[1]$ algebra $(V,q_1,\ldots,q_n,\ldots)$ and a contraction $(g_1,f_1,h)$ of $(V,q_1)$ onto a complex $(W,r_1)$, there is an induced $L_\infty[1]$ algebra structure $R$ on $W$ with linear part $r_1$, together with $L_\infty$ morphisms $F:(W,R)\to(V,Q)$, $G:(V,Q)\to (W,R)$ with linear parts $f_1$, $g_1$ respectively. Denoting by $F^k_i$ the composition $W^{\odot i}\hookrightarrow S(W)\xrightarrow{F} S(V)\twoheadrightarrow V^{\odot k}$, $F$ and $R$ are detemined recursively by (notice that by formula \eqref{morfromtaylor} $F^k_i$ only depends on $f_1,\ldots, f_{i-k+1}$)
\[ f_i = \sum_{k=2}^i hq_k F^k_i \qquad \mbox{for $i\geq2$}, \]	
\[ r_i = \sum_{k=2}^i g_1q_k F^k_i \qquad \mbox{for $i\geq2$}. \]	
It is possible to establish recursive formulas for $G$ as well, but these are a bit more complicated. We denote by $\widehat{h}:V^{\odot i}\to V^{\odot i}$ the contracting homotopy
\[ \widehat{h}(v_1\odot\cdots\odot v_i) = \frac{1}{i!} \sum_{\sigma\in S_i}\sum_{j=1}^{i}\pm_K f_1g_1(v_{\sigma(1)})\odot\cdots\odot f_1g_1(v_{\sigma(j-1)})\odot h(v_{\sigma(j)})\odot v_{\sigma(j+1)}\odot\cdots\odot v_{\sigma(i)}, \]
(where $\pm_K$ is the appropriate Koszul sign, taking into account that $|h|=-1$) given by the symmetrized tensor trick. Denoting by $Q^k_i$ the composition $V^{\odot i}\hookrightarrow S(V)\xrightarrow{Q} S(V)\twoheadrightarrow V^{\odot k}$, the $L_\infty$ morphism $G$ is determined recursively by
	\[ g_i = \sum_{k=1}^{i-1} g_k Q^k_i \widehat{h}\qquad\mbox{for $i\geq2$}.  \]
\end{theorem}

\begin{remark}\label{rem:transfer} The statements about $F$ and $R$ are well known: for a proof we refer to \cite{HuebSt,Fuk00,Hueb4,Mcontr}. The statement about $G$ is less standard, and to the best of the author's knowledge was first showed by Berglund \cite{BerglundHPT} (as a particular case of a much more general result): we shall present a different proof in subsection \ref{subsec:Lootransfer}, where we will revisit the previous theorem in light of the results from the following subsections.
\end{remark}
\begin{remark}\label{rem:complete} Sometimes we shall need to work in the complete setting, especially when considering Maurer-Cartan functors. A \emph{complete space} is a graded space $V$ equipped with a complete descending filtration, i.e., 
$V=F^1V\supset F^2V\supset\cdots\supset F^pV\supset\cdots$
and the natural map $V\to\underleftarrow{\lim}_p V/F^pV$ is an isomorphism. A morphism of complete spaces is continuous if it respects the filtrations. A \emph{complete $L_\infty[1]$ algebra} is an $L_\infty[1]$ algebra $(V,q_1,\ldots,q_n,\ldots)$ equipped with a complete compatible filtration $F^\bullet V$, i.e., $q_i(F^{p_1}V,\ldots,F^{p_i}V)\subset F^{p_1+\cdots+p_i}V$. A morphism of complete $L_\infty[1]$ algebras $F:(V,q_1,\ldots,q_n,\ldots)\to(W,r_1,\ldots,r_n,\ldots)$ is continuous if so are its Taylor coefficients. If the contraction data is continuous, the previous homotopy transfer Theorem can be extended to the complete setting: we refer to \cite{BDel} for a precise statement.
\end{remark}

\begin{definition} The \emph{Maurer-Cartan functor} from complete $L_\infty[1]$ algebras to sets sends a complete $L_\infty[1]$ algebra $(V,q_1,\ldots,q_n,\ldots)$ to its \emph{Maurer-Cartan set} 
	\[ \operatorname{MC}(V):=\{x\in V^0\,\,\mbox{s.t.}\,\,\sum_{n\ge1}\frac{1}{n!}q_n(x^{\odot n})=0 \},\]
	and a continuous morphism $F:(V,q_1,\ldots,q_n,\ldots)\to(W,r_1,\ldots,r_n,\ldots)$ to the push-forward
	\[ \operatorname{MC}(F):\operatorname{MC}(V)\to\operatorname{MC}(W):x\to\sum_{n\ge1}\frac{1}{n!}f_n(x^{\odot n}). \]
\end{definition}

The following formal analog of Kuranishi Theorem, due to Fukaya \cite{Fuk00} and Getzler \cite{GetzlerLie,Getzlerpert}, explains how Maure-Cartan sets behave under homotopy transfer. For a proof of the result as stated here, we refer to \cite[Theorem 1.13]{BDel}.

\begin{theorem}\label{th:kuranishi} In the same hypotheses as in Theorem \ref{th:transfer}, if we are furthermore in a complete setting, then the correspondence
	\[\rho:\operatorname{MC}(V)\to\operatorname{MC}(W)\times h(V^0):x\to(\operatorname{MC}(G)(x),h(x))\]
	is bijective. The inverse $\rho^{-1}$ admits the following recursive construction: given $y\in\operatorname{MC}(W)$ and $h(v)\in h(V^0)$, we define a succession of elements $x_n\in V^0$, $n\geq 0$, by $x_{0}=0$ and 
	\begin{equation}\label{eq:getzlerrecursion} x_{n+1}=f_1(y)-q_1h(v)+\sum_{i\geq2}\frac{1}{i!}\left(hq_i-f_1g_i\right)\left(x_n^{\odot i}\right).    \end{equation}
	This succession converges (with respect to the complete topology induced by the filtration on $V$) to a well defined $x\in \operatorname{MC}(V)$, and we have  $\rho^{-1}(y,h(v))=x$. Finally, $\rho^{-1}(-,0)$ coincides with $\operatorname{MC}(F):\operatorname{MC}(W)\to\operatorname{MC}(V)$: this induces a bijective correspondence between the sets $\operatorname{MC}(W)$ and $\operatorname{Ker}(h)\bigcap\operatorname{MC}(V)$, whose inverse is the restriction of $g_1$. \end{theorem}

Finally, we consider the dual notion of $L_\infty[-1]$ coalgebra. To motivate the definition, we observe that given a finite dimensional $L_\infty[1]$ algebra $(V,q_1,\ldots,q_n,\ldots)$, the dual maps $q_n^\vee:V^\vee\to(V^{\odot n})^\vee\to(V^\vee)^{\odot n}$ assemble to a degree one map $q^\vee=q_1^\vee+\cdots+q_n^\vee+\cdots:V^\vee\to\prod_{n\ge0}(V^\vee)^{\odot n}=\widehat{S}(V^\vee)$, which in turn extends uniquely to a derivation $Q^\vee$ of the completed symmetric algebra $\widehat{S}(V^\vee)$ squaring to zero. 

\begin{definition} An \emph{$L_\infty[-1]$ coalgebra} structure on a graded space $V$ is a DG algebra structure $d=d_1+d_2+\cdots+d_n+\cdots$ on $\widehat{S}(V)$, where $d_n:V\to V^{\odot n}$. Given $L_\infty[-1]$ coalgebras $(V,d=d_1+d_2+\cdots+d_n+\cdots)$, $(W,d'=d'_1+d'_2+\cdots+d'_n+\cdots)$, a morphism between them is a morphism of DG algebras $f=f_1+f_2+\cdots+f_n+\cdots:(\widehat{S}(V),d)\to(\widehat{S}(W),d')$.
\end{definition} 

\begin{remark} In particular, $d_1:V\to V$ squares to zero, and the (shifted) cobracket associated to $d_2$ satisfies the co-Jacobi identity up to the homotopy $d_3$. Hence, the shifted cohomology $H(V)[1]$ is a graded Lie coalgebra.
\end{remark}

\begin{remark} If $(V,d=d_1+\cdots+d_n+\cdots)$ is an $L_\infty[-1]$ coalgebra, then the dual maps $d_n^\vee:(V^\vee)^{\odot n}\to(V^{\odot n})^\vee\to V^\vee$ assemble to an $L_\infty[1]$ algebra structure $(d_1^\vee,\ldots,d_n^\vee,\ldots)$ on $V^\vee$, even when $V$ is infinite dimensional.
\end{remark}

\begin{remark} $L_\infty[-1]$ coalgebra structures transfer along contractions, and, as in the $L_\infty[1]$ algebra case, this can be made explicit by combining the symmetrized tensor trick with the homological perturbation lemma (see \cite{BerglundHPT} for a proof, as well as Remark \ref{rem:Loocotrans}).
\end{remark}

\subsection{Cumulants and Koszul brackets}\label{sec:cumulantsandkoszul} 
	
	Given a graded commutative unitary algebra $A$, we denote by $E:S(A)\to S(A)$ the coalgebra automorphism with Taylor coefficients
	\[pE=(e_1,\ldots,e_n,\ldots),\qquad e_n(a_1,\ldots,a_n):=a_1\cdots a_n. \]
	($p:S(A)\to A$ being the natural projection). We denote by $L:S(A)\to S(A)$ the inverse of $E$. It is easily checked that
	\[pL=(l_1,\ldots,l_n,\ldots),\qquad l_n(a_1,\ldots,a_n)= (-1)^{n-1}(n-1)! a_1\cdots a_n. \]
	We call $E$ and $L$ respectively the \emph{exponential automorphism} and the \emph{logarithmic automorphism} of $S(A)$.
	\begin{remark} In certain situations we shall need to work in a complete augmented setting. In this case, we shall assume that $A$ has an augmentation $\epsilon:A\to\mathbb{K}$ and is equipped with a compatible complete filtration $F^0A=A\supset F^1A=\operatorname{Ker}(\epsilon)\supset\cdots\supset F^p A\supset\cdots$. In this situation the push-forwards of $a\in F^1A^0$ along $E$ and $L$ are given by the usual exponential and logarithmic series
\begin{multline*} E_*(a):=\sum_{n\ge1}\frac{1}{n!}e_n(a^{\odot n})=\sum_{n\ge1}\frac{1}{n!}a^n= e^a-\mathbf{1}_A,\\ L_*(a):=\sum_{n\ge1}\frac{1}{n!}l_n(a^{\odot n})=\sum_{n\ge1}\frac{(-1)^{n-1}}{n}a^n= \log(\mathbf{1}_A+a),
\end{multline*}		
		which justifies the names. 	
	\end{remark}

\begin{remark} Here and in the rest of the paper, we shall use the following notations: given a graded unitary algebra $A$, we shall denote by \[ \operatorname{End}_u(A):=\{\Delta\in \operatorname{End}(A)\,\,\mbox{s.t.}\,\, \Delta(\mathbf{1}_A)=0 \}.\]
Given a pair of graded unitary algebras $A$ and $B$, we shall denote by 
\[ \operatorname{Hom}^0_u(A,B):=\{f\in \operatorname{Hom}^0(A,B)\,\,\mbox{s.t.}\,\, f(\mathbf{1}_A)=\mathbf{1}_B \}\]
(notice in particular that with these notations $\operatorname{End}^0_u(A)$ and $\operatorname{Hom}^0_u(A,A)$ have different meanings).\end{remark}
	\begin{definition}\label{def:cumulants} Given a pair of graded commutative unitary algebras $A$ and $B$, together with a degree zero linear map  $f\in\Hom_u^0(A,B)$, the \emph{cumulants} of $f$ are the Taylor coefficients $\kappa(f)_n:A^{\odot n}\to B$ of the coalgebra morphism $\kappa(f):S(A)\to S(B)$ defined as the composition
		\begin{equation}\label{eq:defcumulants} \kappa(f) := S(A)\xrightarrow{E}S(A)\xrightarrow{S(f)}S(B)\xrightarrow{L}S(B),  \end{equation} 
where $S(f):S(A)\to S(B)$ is the coalgebra morphism  
\[ S(f)(\mathbf{1}_{S(A)})=\mathbf{1}_{S(B)},\qquad S(f)(a_1\odot\cdots\odot a_n)=f(a_1)\odot\cdots\odot f(a_n).\]
	\end{definition}
	
	\begin{remark} The cumulants of $f$ measure the deviation of $f$ from being a morphism of unitary algebras. For instance, a direct computation shows that the first few cumulants of $f$ are 
		\[ \kappa(f)_1(a)=f(a),\qquad \kappa(f)_2(a,b)=f(ab)-f(a)f(b), \]
		\begin{multline*} \kappa(f)_3(a,b,c)=f(abc)-f(ab)f(c)-(-1)^{|b||c|}f(ac)f(b)-\\-(-1)^{|a|(|b|+|c|)}f(bc)f(a)+2f(a)f(b)f(c). 
		\end{multline*}
		In general we have the following explicit formula for the higher cumulants (cf. for instance with \cite{SulRan,Ranade2}), which again can be checked by a straightforward computation,
		\begin{equation}\label{eq:cumulants} \kappa(f)_n(a_1,\ldots,a_n) = \sum_{\stackrel{1\leq k\leq n,\,I_1,\ldots,I_k\neq\emptyset}{I_1\bigsqcup\cdots\bigsqcup I_k=\{1,\ldots,n \}}}(-1)^{k-1}(k-1)!\pm_K\prod_{1\leq j\leq k} f\left(\prod_{i\in I_j} a_i\right). 
		\end{equation}
		Here the sum runs over the ${n\brace k}$ \emph{non-ordered} partitions of $\{1,\ldots,n\}$ into the disjoint union of $k$ non-empty subsets $I_1,\ldots,I_k$, and we denote by $\pm_K$ the appropriate Koszul sign associated with the resulting permutation of $a_1,\ldots,a_n$. In particular, $f$ is a morphism of unitary graded algebras if and only if $\kappa(f)_n=0$ for all $n>1$.
		
	\end{remark} 
	It follows immediately from the definitions that cumulants are compatible with composition, as in the claim of the following lemma.
\begin{lemma}\label{prop:cumulantsandcomposition} Given graded commutative algebras $A,B,C$ and maps $f\in\Hom^0_{u}(A,B),g\in\Hom^0_{u}(B,C)$, then 
	\[ \kappa(g\circ f) =\kappa(g)\circ\kappa(f). \] 
\end{lemma}
\begin{proof} This follows immediately from the definitions $\kappa(f):=L\circ S(f)\circ E$, $\kappa(g):=L\circ S(g)\circ E$, the fact that $E:S(B)\to S(B)$ and $L:S(B)\to S(B)$ are inverses to each other and the fact that $S(g\circ f)=S(g)\circ S(f)$.
\end{proof}

	Yet another way to define the higher cumulants is the following recursion: starting with $\kappa(f)_1(a)=f(a)$, for all $n\geq0$ we put
	\begin{multline}\label{eq:recursioncumulants} \kappa(f)_{n+2}(a_1,\ldots,a_n,b,c)=\kappa(f)_{n+1}(a_1,\ldots,a_n,bc)-\\-\sum_{i=0}^{n} \sum_{\sigma\in S(i,n-i)}\pm_K\kappa(f)_{i+1}(a_{\sigma(1)},\ldots,a_{\sigma(i)},b)\kappa(f)_{n-i+1}(a_{\sigma(i+1)},\ldots,a_{\sigma(n)},c). 
	\end{multline}
	Here we denote by $S(i,n-i)$ the set of $(i,n-i)$-unshuffles, and by $\pm_K$ the Koszul sign associated to the permutation $a_1,\ldots,a_n,b,c\mapsto a_{\sigma(1)},\ldots,a_{\sigma(i)},b,a_{\sigma(i+1)},\ldots,a_{\sigma(n)},c$. This recursion can be shown by a simple calculation, whose details are left to the reader, using the explicit formula \eqref{eq:cumulants} for the cumulants. Essentially, if we expand both the left and the right hand side of \eqref{eq:recursioncumulants} according to \eqref{eq:cumulants}, then $\kappa(f)_{n+1}(a_1,\ldots,a_n,bc)$ accounts for all the terms in the expansion of $\kappa(f)_{n+2}(a_1,\ldots,a_n,b,c)$ such that $b$ and $c$ belong to the same block of the partition, while $-\sum_{i=0}^{n} \sum_{\sigma\in S(i,n-i)}(\pm)\kappa(f)_{i+1}(a_{\sigma(1)},\ldots,a_{\sigma(i)},b)\kappa(f)_{n-i+1}(a_{\sigma(i+1)},\ldots,a_{\sigma(n)},c)$ accounts for all those terms such that $b$ and $c$ belong to different blocks. 
\newcommand{\Kos}{\mathcal{K}}

\begin{remark} It follows immediately from the definitions that in the complete setting, assuming furthermore that $f$ is continuous, push-forward along $\kappa(f)$ sends $a\in F^1A^0$ to
\begin{equation}\label{eq:MCcumulants}
 \kappa(f)_*(a) := \sum_{n\ge 1}\frac{1}{n!}\kappa(f)_n(a^{\odot n})=\log(f(e^a)).
\end{equation}
\end{remark}
	
Next we recall the classical construction of Koszul brackets (see \cite{Kos,Markl1} for the equivalence between the following definitions), and explore its relation with the previous construction of cumulants. 
\begin{definition}\label{def:koszulbrackets} Given a degree $k$  endomorphism $\Delta\in\End_u^k(A)$ of $A$, we denote by $\widetilde{\Delta}:S(A)\to S(A)$ the coderivation given in Taylor coefficients by $\widetilde{\Delta}_1=\Delta$ and $\widetilde{\Delta}_n=0$ for $n\neq1$.
	The \emph{Koszul brackets} $\Kos(\Delta)_n:A^{\odot n}\to A$, $n\ge0$, of $\Delta$ are the Taylor coefficients of the coderivation
$\Kos(\Delta):S(A)\to S(A)$ defined as the composition		\begin{equation}\label{eq:defkoszulbrackets} \Kos(\Delta) := S(A)\xrightarrow{E}S(A)\xrightarrow{\widetilde{\Delta}}S(A)\xrightarrow{L}S(A).  \end{equation} 
\end{definition}

\begin{remark} The Koszul brackets of $\Delta$ measure the deviation of $\Delta$ from being a graded algebra derivation. For instance, once again by a direct computation, the first few Koszul brackets are	
	\[ \Kos(\Delta)_1(a)=\Delta(a)-\Delta(1_A)a,\]\[ \Kos(\Delta)_2(a,b)=\Delta(ab)-\Delta(a)b-(-1)^{|a||b|}\Delta(b)a+\Delta(1_A)ab, \]
	\begin{multline*} \Kos(\Delta)_3(a,b,c)=\Delta(abc)-\Delta(ab)c-(-1)^{|b||c|}\Delta(ac)b-(-1)^{|a|(|b|+|c|)}\Delta(bc)a+\\+\Delta(a)bc +(-1)^{|a||b|}\Delta(b)ac+(-1)^{(|a|+|b|)|c|}\Delta(c)ab-\Delta(1_A)abc. 
	\end{multline*}
In general, \begin{equation}\label{eq:koszul}
	\Kos(\Delta)_n(a_1,\ldots,a_n) = \sum_{i=1}^n\sum_{\sigma\in S(i,n-i)}\pm_K(-1)^{n-i}\Delta(a_{\sigma(1)}\cdots a_{\sigma(i)})a_{\sigma(i+1)}\cdots a_{\sigma(n)}.\end{equation}
\end{remark}
Equivalently, the Koszul brackets $\Kos(\Delta)_n$ might be defined by the recursion $\Kos(\Delta)_1(a)=\Delta(a)$ and for $n\geq0$
	\begin{multline}\label{eq:recursionkoszul} \Kos(\Delta)_{n+2}(a_1,\ldots,a_n,b,c)=\Kos(\Delta)_{n+1}(a_1,\ldots,a_n,bc)-\\-\Kos(\Delta)_{n+1}(a_1,\ldots,a_n,b)c-(-1)^{|b||c|}\Kos(\Delta)_{n+1}(a_1,\ldots,a_n,c)b. 
	\end{multline}
We have the following results, the latter two illustrating the compatibility between cumulants and Koszul brackets.
\begin{lemma}\label{lem:KosDGLiemor} If we regard $\End_{u}(A)$ and $\Coder(S(A))$ as graded Lie algebras via the commutator brackets, then the correspondence $\Kos(\#):\End_{u}(A)\to\Coder(S(A)):\Delta\to\Kos(\Delta)$ is a morphism of graded Lie algebras.
\end{lemma}
\begin{proposition}\label{prop:cumulantsvsKoszuldg} Given graded commutative unitary algebras $A$ and $B$, endomorphisms $\Delta\in\End_u(A)$, $\Delta'\in\End_u(B)$ and a degree zero map $f\in\Hom_u^0(A,B)$, if $f\circ\Delta =\Delta'\circ f$, then also $\kappa(f)\circ\Kos(\Delta)=\Kos(\Delta')\circ\kappa(f)$. \end{proposition}

\begin{proposition}\label{prop:cumulantsvsKoszulexp} Given a degree zero endomorphism $\Delta\in\End_u^0(A)$ as above, such that the exponential $\exp(\Delta)\in\operatorname{Hom}_u^0(A,A)$ is well defined (i.e., the corresponding infinte sum makes sense, either by nilpotency of $\Delta$ or some other appropriate hypothesis), then the exponential $\exp(\Kos(\Delta)):S(A)\to S(A)$ of the coderivation $\Kos(\Delta)$ is also well defined, and an automorphism of the symmetric coalgebra $S(A)$. Moreover, \[ \exp(\Kos(\Delta)) = \kappa(\exp(\Delta)). \]
\end{proposition}
By Definition \ref{def:koszulbrackets} the proof of the first result reduces to show  that $\End_u(A)\to\operatorname{Coder}(S(A)):\Delta\to\widetilde{\Delta}$ is a morpshims of graded Lie algebras, which is immediate. The second result follows straightforwardly by comparing the Definitions \ref{def:cumulants} and \ref{def:koszulbrackets}. The same applies for the last result, with the further observation that in the hypotheses of the enunciate the exponential of the coderivation $\widetilde{\Delta}$ is well defined, and is the coalgebra automorphism $\exp(\widetilde{\Delta})=S(\exp(\Delta))$.

Let us further point out the following formula, valid in the complete setting for any $a\in F^1A^0$,
\begin{equation}\label{eq:MCkoszul} \sum_{n\ge1}\frac{1}{n!}\Kos(\Delta)_n(a^{\odot n}) = e^{-a}\Delta(e^a). 
\end{equation}
In particular, $\sum_{n\ge1}\frac{1}{n!}\Kos(\Delta)_n(a^{\odot n})=0 $ if and only if $\Delta(e^a) =0$.
\newcommand{\Diff}{\operatorname{Diff}}

We conclude this subsection by recalling the definition of differential operators on a graded commutative unitary algebra.
\begin{definition}\label{def:diffop}
	Given a graded commutative algebra $A$ and $k\geq1$, we denote by 
\[ \Diff_{u,\leq k}(A) := \{ \Delta\in\End_u(A)\,\,\mbox{s.t.}\,\,\Kos(\Delta)_{k+1}=0 \}= \{ \Delta\in\End_u(A)\,\,\mbox{s.t.}\,\,\Kos(\Delta)_{n}=0,\,\forall\,n>k \},  \]
where the second identity follows immediately from the recursion \ref{eq:recursionkoszul} for the Koszul brackets. In other words, $\Delta\in\Diff_{u,\leq k}(A)$ if $\Delta$ is a differential operator of order $\leq k$ on $A$, such that moreover $\Delta(\1_A)=0$. We denote by $\Diff_u(A) = \bigcup_{k\geq 1} \Diff_{u,\leq k}(A)$. It is well known that $\Diff_{u,\leq j}(A)\circ\Diff_{u,\leq k}(A)\subset\Diff_{u,\leq (k+j)}(A)$ and $\big[\Diff_{u,\leq j}(A),\Diff_{u,\leq k}(A)\big]\subset\Diff_{u,\leq (k+j-1)}(A)$, so that $\Diff_u(A)\subset\End_u(A)$ is both a graded subalgebra and a graded Lie subalgebra. 
\end{definition}

\subsection{Homological perturbation theory}\label{subsec:comp}
\newcommand{\s}{\breve{\sigma}}
\renewcommand{\t}{\breve{\tau}}
\newcommand{\h}{\breve{h}}
\renewcommand{\d}{\breve{d}}

\begin{definition}\label{def:contraction}
	Given a pair of DG spaces $(A,d_A)$ and $(B,d_B)$, a \emph{contraction} $(\sigma,\tau,h)$ of $A$ onto $B$ is the datum of DG maps $\sigma: (A,d_A)\to(B,d_B)$, $\tau: (B,d_B)\to(A,d_A)$
	and a contracting homotopy $h: A\to A[-1]$ such that
	\[ \sigma\tau = \id_B,\qquad hd_A+d_Ah = \tau\sigma-\id_A ,\]
	\[ \sigma h=0,\qquad h\tau=0,\qquad h^2=0.\]
\end{definition}

We shall often consider the datum of a graded commutative unitary algebra $(A,\cdot)$ equipped with a degree one differential $d_A$ squaring to zero, which \textbf{might not be an algebra derivation},  and is only required to satisfy $d_A(\textbf{1}_A)=0$: we shall call this datum a \emph{graded commutative unitary algebra with a differential}, whereas as usual if $d_A$ is an algebra derivation we call $(A,d_A,\cdot)$ a \emph{DG commutative unitary algebra}.



\begin{definition}\label{def:semifullalgebracontraction} Given graded commutative algebras $A$ and $B$ together with differentials $d_A$, $d_B$  (not necessarily algebra derivations, but as explained above we do require $d_A(\mathbf{1}_A)=0$, $d_B(\mathbf{1}_B)=0$) and a contraction $(\sigma,\tau,h)$ of $(A,d_A)$ onto $(B,d_B)$, we say that $(\sigma,\tau,h)$ is a \emph{semifull algebra contraction} if the following identities are satisfied for all $a,b\in A$, $x,y\in B$
\begin{multline*}  h(h(a)h(b)) = h(h(a)\tau(x))=h(\tau(x)\tau(y)) = h(\mathbf{1}_A) = 0,\\ \sigma(h(a)h(b)) = \sigma(h(a)\tau(x))=0,\quad\sigma(\tau(x)\tau(y)) = xy,\quad\sigma(\mathbf{1}_A)=\mathbf{1}_B.
\end{multline*}
If furthermore $d_A$ is an algebra derivations, we say that $(\sigma,\tau,h)$ is a \emph{semifull DG algebra contraction}. 
Semifull algebra contractions were introduced by Real \cite{Real}: we refer to loc. cit. for further discussion on this class of contractions and several examples. \end{definition}

\begin{remark}\label{rem:semifullDG} When $(\sigma,\tau,h)$ is a semifull DG algebra contraction, i.e., when $d_A$ is an algebra derivation, then the identities in the previous definition are equivalent to the seemingly stronger 
\begin{eqnarray}
\label{eqA1bis} h(\,(-1)^{|a|+1} h(a)b + ah(b)\,) &=& h(a)h(b),\\
\label{eqA2bis} h( a\tau(x)) &=& h(a)\tau(x),\\\sigma(\,(-1)^{|a|+1} h(a)b + ah(b)\,) &=& 0,\\
\label{eqA4bis} \sigma( a\tau(x)) &=& \sigma(a)x.
\end{eqnarray}
as can be seen from the following identities \eqref{eq:failureA1}, \eqref{eq:failureA2}, \eqref{eq:failureA3}, \eqref{eq:failureA4}. Moreover, the following Proposition \ref{prop:transfer} will imply immediately that if $(\sigma,\tau,h)$ is a semifull DG algebra contraction, then necessarily $d_B$ is an algebra derivation and $\tau$ is a morphism of graded algebras.\end{remark}

The main technical result of this section shows that given graded commutative unitary algebras with a differential $(A,d_A)$ and $(B,d_B)$, together with a semifull algebra contraction $(\sigma,\tau,h)$ of $(A,d_A)$ onto $(B,d_B)$,  the Koszul brackets $\Kos(d_B)_n$ on $B$ and the ones $\Kos(d_A)_n$ on $A$ are related via homotopy transfer for $L_\infty[1]$ algebras. More precisely, the sequence of Koszul brackets $\Kos(d_A)_n$, $n\geq1$, induces a (homotopy abelian) $L_\infty[1]$ algebra structure on $A$. Via homotopy transfer along the contraction $(\sigma,\tau,h)$ (see Theorem \ref{th:transfer}), there is an induced $L_\infty[1]$ algebra structure on $B$, as well as an $L_\infty[1]$ morphism from $B$ to $A$. We shall see that the former is precisely $\Kos(d_B):S(B)\to S(B)$, and the latter precisely $\kappa(\tau):S(B)\to S(A)$.

Before we do this, given unitary graded commutative algebras $A,B$ with differentials (not necessarily derivations) $d_A,d_B$ and a semifull algebra contraction $(\sigma,\tau,h)$ between them, we need to look more closely at the failure of the previous identities \eqref{eqA1bis}-\eqref{eqA4bis}.

	We have 
	\[ h\left(ah(b)\right) = h\left((\tau\sigma - d_Ah-hd_A)(a)h(b)\right) = -h\left(d_Ah(a)h(b)\right), \]
	and similarly $h\left(h(a)b\right)=-h\left(h(a)d_Ah(b)\right)$. Moreover $h(a)h(b)=(\tau\sigma-d_Ah-hd_A)\left(h(a)h(b)\right) = -hd_A\left(h(a)h(b)\right)$. Thus,  
\begin{multline}\label{eq:failureA1}h\left((-1)^{|a|+1}h(a)b +ah(b) \right) -h(a)h(b) =\\=h\left( d_A\left(h(a)h(b)\right) -d_Ah(a)h(b)-(-1)^{|a|+1}h(a)d_Ah(b) \right) =\\= h\Kos(d_A)_2\left(h(a),h(b)\right). 
\end{multline}
In the same way, since $h(a)\tau(x)=-hd_A\left(h(a)\tau(x)\right)$ and \begin{multline*} h\left(a\tau(x)\right)=-h\left(d_Ah(a)\tau(x)\right)=-h\left(d_Ah(a)\tau(x)+(-1)^{|a|+1}h(a)\tau d_B(x)\right)=\\=-h\left(d_Ah(a)\tau(x)+(-1)^{|a|+1}h(a)d_A\tau(x)\right), 
\end{multline*}
we see that
\begin{multline}\label{eq:failureA2}h\left(a\tau(x) \right) -h(a)\tau(x) =\\=h\left( d_A\left(h(a)\tau(x)\right) -d_Ah(a)\tau(x)-(-1)^{|a|+1}h(a)d_A\tau(x) \right) =\\= h\Kos(d_A)_2\left(h(a),\tau(x)\right). 
\end{multline}
The same kind of computations shows that
\begin{equation}\label{eq:failureA3} \sigma\left((-1)^{|a|+1}h(a)b +ah(b) \right) = \sigma\Kos(d_A)_2\left(h(a),h(b)\right). 
\end{equation}
\begin{equation}\label{eq:failureA4}\sigma\left(a\tau(x) \right) -\sigma(a)x = \sigma\Kos(d_A)_2\left(h(a),\tau(x)\right). 
\end{equation}

\begin{proposition}\label{prop:transfer} Given a semifull algebra contraction as above, the $L_\infty[1]$ algebra structure on $A$ associated with the Koszul brackets $\Kos(d_A)_n$ transfers to the one on $B$ associated with the Koszul brackets $\Kos(d_B)_n$. Moreover, the induced $L_\infty[1]$ morphism $(S(B),\Kos(d_B))\to(S(A),\Kos(d_A))$ is precisely $\kappa(\tau)$.
\end{proposition}	

\begin{proof} This is a rather long computation. We need to prove that the $\kappa(\tau)_n$, $\Kos(d_B)_n$ satisfy the following recursions: starting from $\kappa(\tau)_1=\tau$, $\Kos(d_B)_1=d_B$, for all $n\geq2$ and $x_1,\ldots,x_n\in B$, we have to show
	\[ \kappa(\tau)_n(x_1,\ldots,x_n) = \sum_{\stackrel{k\geq2,i_1,\ldots,i_k\geq1}{i_1+\cdots+i_k=n}}\sum_{\sigma\in S(i_1,\ldots,i_k)}\pm_K\frac{1}{k!}h\Kos(d_A)_k\left(\kappa(\tau)_{i_1}(x_{\sigma(1)}\ldots),\ldots,\kappa(\tau)_{i_k}(\ldots,x_{\sigma(n)})\right) \]
	\[ \Kos(d_B)_n(x_1,\ldots,x_n) = \sum_{\stackrel{k\geq2,i_1,\ldots,i_k\geq1}{i_1+\cdots+i_k=n}}\sum_{\sigma\in S(i_1,\ldots,i_k)}\pm_K\frac{1}{k!}\sigma\Kos(d_A)_k\left(\kappa(\tau)_{i_1}(x_{\sigma(1)}\ldots),\ldots,\kappa(\tau)_{i_k}(\ldots,x_{\sigma(n)})\right) \]
	As usual, in the above formulas we denote by $\pm_K$ the Koszul sign associated to the permutation $x_1,\ldots,x_n\mapsto x_{\sigma(1)},\ldots,x_{\sigma(n)}$.
	
	We prove first that the cumulants $\kappa(\tau)_n$ obey the above recursion, by induction on $n$. As a warm up, we begin by considering the cases $n=2$ and $n=3$. For $n=2$, since $h(\tau(y)\tau(z))=0$ and $\sigma(\tau(y)\tau(z))=yz$,
	\[ \kappa(\tau)_2(y,z) = \tau(yz)-\tau(y)\tau(z) = (\tau\sigma-\id_A)(\tau(y)\tau(z)) = hd_A(\tau(y)\tau(z))=h\Kos(d_A)_2(\tau(y),\tau(z)),  \]
	which is what we needed to show. In the last passage, we used the fact that $h\left(d_A\tau(y)\tau(z)\right)=h\left(\tau d_B(y)\tau(z)\right)=0$, and for the same reason $h\left(\tau(y)d_A\tau(z)\right)=0$.
	
	Next, we consider the case $n=3$: the following computation should already give a hint of our inductive argument for the general case, which is based on the recursions \eqref{eq:recursioncumulants} and \eqref{eq:recursionkoszul} for the cumulants and the Koszul brackets respectively, as well as the previous identities \eqref{eq:failureA1}-\eqref{eq:failureA2}.  For simplicity, we get rid of Koszul signs by showing the necessary relation only for degree zero arguments $x,y,z\in B^0$.
	\begin{multline*} \kappa(\tau)_3(x,y,z) = (\mbox{using \eqref{eq:recursioncumulants}})=  \kappa(\tau)_2(x,yz)-\kappa(\tau)_2(x,y)\tau(z)-\kappa(\tau)_2(x,z)\tau(y) = \\ = (\mbox{using the already done $n=2$ case})=\\= h\Kos(d_A)_2(\tau(x),\tau(yz)) - h\Kos(d_A)_2(\tau(x),\tau(y))\tau(z) - h\Kos(d_A)_2(\tau(x),\tau(z))\tau(y) = \\ = \mbox{(using  \eqref{eq:failureA2})}  = h\Big( \Kos(d_A)_2(\tau(x),\tau(y)\tau(z)) - \Kos(d_A)_2(\tau(x),\tau(y))\tau(z) - \Kos(d_A)_2(\tau(x),\tau(z))\tau(y) \Big) + \\ + h\Kos(d_A)_2\left(\tau(x),\kappa(\tau)_2(y,z)\right)+ h\Kos(d_A)_2\left(\tau(y),h\Kos(d_A)_2(\tau(x),\tau(z))\right)+\\+ h\Kos(d_A)_2\left(\tau(y),h\Kos(d_A)_2(\tau(x),\tau(z))\right) = \mbox{(using  \eqref{eq:recursionkoszul})} = h\Kos(d_A)_3(\tau(x),\tau(y),\tau(z)) + \\ + h\Kos(d_A)_2(\tau(x),\kappa(\tau)_2(y,z))+ h\Kos(d_A)_2(\tau(y),\kappa(\tau)_2(x,z))+ h\Kos(d_A)_2(\tau(z),\kappa(\tau)_2(x,y)).
	\end{multline*}
	
	Now we do the general case. For simplicity, in the following computations we denote by $\kappa_n:=\kappa(\tau)_n$, $\Kos_k:=\Kos(d_A)_k$. Furthermore, as before, we shall get rid of Koszul signs by showing the necessary relation only for $x_1,\ldots,x_n,y,z\in B^0$. Finally, in order to further abbreviate the following equations, we shall omit the arguments $x_1,\ldots,x_n$, which should fill the suspension dots in a way which should be obvious from the context. For instance, with these notational simplifications, the recursion \eqref{eq:recursioncumulants} for the cumulants $\kappa(\tau)_{n+2}$ becomes
	\begin{equation*} \kappa_{n+2}(\ldots,y,z)=\kappa_{n+1}(\ldots,yz)-\sum_{i=0}^{n} \sum_{\sigma\in S(i,n-i)}\kappa_{i+1}(\ldots,y)\kappa_{n-i+1}(\ldots,z). 
	\end{equation*}
	As another example, the recursion for the cumulants we need to show becomes
	\begin{multline}\label{eq:needtoprove}
	\kappa_{n+2}(\ldots,y,z) = 
		\sum_{\stackrel{k\geq2,i_1,\ldots,i_{k-1}\geq1, i_k\geq0}{i_1+\cdots+i_k=n}}\sum_{\sigma\in S(i_1,\ldots,i_k)}\frac{1}{(k-1)!}h\Kos_k\left(\kappa_{i_1}(\ldots),\ldots,\kappa_{i_k+2}(\ldots,y,z)\right) +  \\ +\sum_{\stackrel{k,i_1,\ldots,i_{k-1}\geq1, i_{k},i_{k+1}\geq0}{i_1+\cdots+i_{k+1}=n}}\sum_{\sigma\in S(i_1,\ldots,i_{k+1})}\frac{1}{(k-1)!}h\Kos_{k+1}\left(\kappa_{i_1}(\ldots),\ldots,\kappa_{i_{k}+1}(\ldots,y),\kappa_{i_{k+1}+1}(\ldots,z)\right).
	\end{multline}	
	
	Next, by using in order the inductive hypothesis, the recursions \eqref{eq:recursioncumulants} for the cumulants and the one \eqref{eq:recursionkoszul} for the Koszul brackets,  we see that 
	\begin{multline}\label{longeq1} \kappa_{n+1}(\ldots,yz) =  \sum_{\stackrel{k\geq2,i_1,\ldots,i_{k-1}\geq1,i_k\geq0}{i_1+\cdots+i_k=n}}\sum_{\sigma\in S(i_1,\ldots,i_k)}\frac{1}{(k-1)!}h\Kos_{k}\left(\kappa_{i_1}(\ldots),\ldots,\kappa_{i_{k}+1}(\ldots,yz)  \right) = \\  \sum_{\stackrel{k\geq2,i_1,\ldots,i_{k-1}\geq1,i_k\geq0}{i_1+\cdots+i_k=n}}\sum_{\sigma\in S(i_1,\ldots,i_k)}\frac{1}{(k-1)!}h\Kos_{k}\left(\kappa_{i_1}(\ldots),\ldots,\kappa_{i_{k}+2}(\ldots,y,z)  \right) +\\ 
	\sum_{\stackrel{k\geq2,i_1,\ldots,i_{k-1}\geq1,i_k,i_{k+1}\geq0}{i_1+\cdots+i_{k+1}=n}}\sum_{\sigma\in S(i_1,\ldots,i_{k+1})}\frac{1}{(k-1)!}h\Kos_{k}\left(\kappa_{i_1}(\ldots),\ldots,\kappa_{i_{k}+1}(\ldots,y)\kappa_{i_{k+1}+1}(\ldots,z)  \right) = \\ \sum_{\stackrel{k\geq2,i_1,\ldots,i_{k-1}\geq1,i_k\geq0}{i_1+\cdots+i_k=n}}\sum_{\sigma\in S(i_1,\ldots,i_k)}\frac{1}{(k-1)!}h\Kos_{k}\left(\kappa_{i_1}(\ldots),\ldots,\kappa_{i_{k}+2}(\ldots,y,z)  \right) +\\ 
	\sum_{\stackrel{k\geq2,i_1,\ldots,i_{k-1}\geq1,i_k,i_{k+1}\geq0}{i_1+\cdots+i_{k+1}=n}}\sum_{\sigma\in S(i_1,\ldots,i_{k+1})}\frac{1}{(k-1)!}h\Kos_{k+1}\left(\kappa_{i_1}(\ldots),\ldots,\kappa_{i_{k}+1}(\ldots,y),\kappa_{i_{k+1}+1}(\ldots,z)  \right) +  \\ \sum_{\stackrel{k\geq2,i_1,\ldots,i_{k-1}\geq1,i_k,i_{k+1}\geq0}{i_1+\cdots+i_{k+1}=n}}\sum_{\sigma\in S(i_1,\ldots,i_{k+1})}\frac{1}{(k-1)!}h\Big( \Kos_{k}\left(\kappa_{i_1}(\ldots),\ldots,\kappa_{i_{k}+1}(\ldots,y)  \right)\kappa_{i_{k+1}+1}(\ldots,z)\Big) +\\ \sum_{\stackrel{k\geq2,i_1,\ldots,i_{k-1}\geq1,i_k,i_{k+1}\geq0}{i_1+\cdots+i_{k+1}=n}}\sum_{\sigma\in S(i_1,\ldots,i_{k+1})}\frac{1}{(k-1)!}h\Big( \Kos_{k}\left(\kappa_{i_1}(\ldots),\ldots,\kappa_{i_{k+1}+1}(\ldots,z)  \right) \kappa_{i_{k}+1}(\ldots,y)\Big) \end{multline}
	We notice that the first two lines of the rightmost term of the previous equations account for all the summands in the right hand side of \eqref{eq:needtoprove} except for the ones
	\begin{equation}\label{shorteq} \sum_{i=0}^n\sum_{\sigma\in S(i,n-i)}h\Kos_{2}\left(\kappa_{i+1}(\ldots,y),\kappa_{n-i+1}(\ldots,z)\right).
	\end{equation}
	We consider the $i=n$ term in the above sum. Using in order the hynductive hypothesis, Equation \eqref{eq:failureA2} and again the inductive hypothesis, we may rewrite this as 
	\begin{multline}\label{longeq2}
	  h\Kos_2\left(\kappa_{n+1}(\ldots,y),\tau(z)\right) = \\ = \sum_{\stackrel{k\geq2,i_1,\ldots,i_{k-1}\geq1,i_k\geq0}{i_1+\cdots+i_k=n}}\sum_{\sigma\in S(i_1,\ldots,i_k)}\frac{1}{(k-1)!}h\Kos_2\left( h\Kos_{k}(\kappa_{i_1}(\ldots),\ldots,\kappa_{i_k+1}(\ldots,y)),\tau(z)\right) = \\ = - \frac{1}{(k-1)!} h\Kos_{k}(\kappa_{i_1}(\ldots),\ldots,\kappa_{i_k+1}(\ldots,y))\tau(z) + \\ +\sum_{\stackrel{k\geq2,i_1,\ldots,i_{k-1}\geq1,i_k\geq0}{i_1+\cdots+i_k=n}}\sum_{\sigma\in S(i_1,\ldots,i_k)}\frac{1}{(k-1)!}h\Big(\Kos_{k}(\kappa_{i_1}(\ldots),\ldots,\kappa_{i_k+1}(\ldots,y))\tau(z)\Big)= \end{multline}\begin{equation*}= -\kappa_{n+1}(\ldots,y)\tau(z) + \sum_{\stackrel{k\geq2,i_1,\ldots,i_{k-1}\geq1,i_k\geq0}{i_1+\cdots+i_k=n}}\sum_{\sigma\in S(i_1,\ldots,i_k)}\frac{1}{(k-1)!}h\Big(\Kos_{k}(\kappa_{i_1}(\ldots),\ldots,\kappa_{i_k+1}(\ldots,y))\tau(z)\Big)
	 \end{equation*}
	Reasoning in the same say, we may rewrite the $i=0$ term in \eqref{shorteq} as 
	\begin{multline}\label{longeq3}  h\Kos_2\left(\tau(y),\kappa_{n+1}(\ldots,z)\right) = -\tau(y)\kappa_{n+1}(\ldots,z) + \\ + \sum_{\stackrel{k\geq2,i_1,\ldots,i_{k-1}\geq1,i_k\geq0}{i_1+\cdots+i_k=n}}\sum_{\sigma\in S(i_1,\ldots,i_k)}\frac{1}{(k-1)!}h\Big(\Kos_{k}(\kappa_{i_1}(\ldots),\ldots,\kappa_{i_k+1}(\ldots,z))\tau(y)\Big).
	\end{multline}
	When $0<i<n$ we can mimick the same argument, but this time using Equation \eqref{eq:failureA1} in place of \eqref{eq:failureA2}, to conclude that
	\begin{multline}\label{longeq4} \sum_{\sigma\in S(i,n-i)}h\Kos_2\left(\kappa_{i+1}(\ldots,y),\kappa_{n-i+1}(\ldots,z)\right) = -\sum_{\sigma\in S(i,n-i)}\kappa_{i+1}(\ldots,y)\kappa_{n-i+1}(\ldots,z) + \\ + \sum_{\stackrel{k\geq2,i_1,\ldots,i_{k-1}\geq1,i_k\geq0}{i_1+\cdots+i_{k}=i}}\sum_{\sigma\in S(i_1,\ldots,i_{k},n-i)}\frac{1}{(k-1)!}h\Big( \Kos_{k}\left(\kappa_{i_1}(\ldots),\ldots,\kappa_{i_{k}+1}(\ldots,y)  \right)\kappa_{n-i+1}(\ldots,z)\Big) +\\+ \sum_{\stackrel{k\geq2,i_1,\ldots,i_{k-1}\geq1,i_k\geq0}{i_1+\cdots+i_{k+}=n-i}}\sum_{\sigma\in S(i_1,\ldots,i_{k},i)}\frac{1}{(k-1)!}h\Big( \Kos_{k}\left(\kappa_{i_1}(\ldots),\ldots,\kappa_{i_{k}+1}(\ldots,z)  \right) \kappa_{i+1}(\ldots,y)\Big).
	\end{multline}
	Putting all the above computations together, we can plug the last three Equations \eqref{longeq2}-\eqref{longeq4} into \eqref{longeq1} to conclude that
	\begin{multline*}
\kappa_{n+1}(\ldots,yz) = \sum_{i=0}^{n} \sum_{\sigma\in S(i,n-i)}\kappa_{i+1}(\ldots,y)\kappa_{n-i+1}(\ldots,z)+ \\+
\sum_{\stackrel{k\geq2,i_1,\ldots,i_{k-1}\geq1, i_k\geq0}{i_1+\cdots+i_k=n}}\sum_{\sigma\in S(i_1,\ldots,i_k)}\frac{1}{(k-1)!}h\Kos_k\left(\kappa_{i_1}(\ldots),\ldots,\kappa_{i_k+2}(\ldots,y,z)\right) +  \\ +\sum_{\stackrel{k,i_1,\ldots,i_{k-1}\geq1, i_{k},i_{k+1}\geq0}{i_1+\cdots+i_{k+1}=n}}\sum_{\sigma\in S(i_1,\ldots,i_{k+1})}\frac{1}{(k-1)!}h\Kos_{k+1}\left(\kappa_{i_1}(\ldots),\ldots,\kappa_{i_{k}+1}(\ldots,y),\kappa_{i_{k+1}+1}(\ldots,z)\right)
\end{multline*}	
Finally, using the recursion \eqref{eq:recursioncumulants} for the cumulants we have proved the desired identity \eqref{eq:needtoprove}.	
	
	The fact that the Koszul brackets $\Kos(d_B)_n$ obey the necessary recursion could be shown by a similar computation, replacing the identities \eqref{eq:failureA1}-\eqref{eq:failureA2} with the ones \eqref{eq:failureA3}-\eqref{eq:failureA4} in the above argument. On the other hand, we observe that since $\kappa(\tau)_1=\tau$ is injective, so is the coalgebra morphism $\kappa(\tau):S(B)\to S(A)$. In particular, there is at most one DG coalgebra strucure on $S(B)$ making $\kappa(\tau)$ into a morphism of DG coalgebras from $S(B)$ to $(S(A),\Kos(d_A))$. Since both the $L_\infty[1]$ algebra structure on $B$ induced via homotopy transfer and the one associated to $\Kos(d_B)$ satisfy this property, the former by the first part of the proof and the latter  by Proposition \ref{prop:cumulantsvsKoszuldg}, we conclude that they have to coincide. \end{proof}	

We conclude this subsection by reviewing the Standard Perturbation Lemma. This is a very well known and classical result, see for instance \cite{Shih,Brown,Gug1,Gug2,Hueb2}, and has been applied to perform homotopical transfer of algebraic structures since the work of Kadeishvili \cite{Kad1,Kad2}.

\begin{definition} Given a DG space $(A,d_A)$ and a degree one map $\delta_A:A\to A[1]$, we say that $\delta_A$ is a perturbation of the differential $d_A$ on $A$ if $\breve{d}_A:=d_A+\delta_A$ satisfies $(\breve{d}_A)^2=0$.
\end{definition}

\begin{lemma}\label{lem:SPL} Given a pair of DG spaces $(A,d_A)$ and $(B,d_B)$, a contraction $(\sigma,\tau,h)$ of $A$ onto $B$ and a perturbation $\delta_A$ of the differential on $A$, there is (under appropriate hypotheses ensuring convergence) an induced perturbation
	\[ \delta_B := \sum_{n\geq0} \sigma\delta_A(h\delta_A)^n\tau \]
	of the differential $d_B$ on $B$, as well as a perturbed contraction
	\[ \breve{\sigma}:= \sum_{n\geq0}\sigma(\delta_Ah)^n\]
	\[ \breve{\tau}:= \sum_{n\geq0}(h\delta_A)^n\tau \]
	\[ \breve{h}:= \sum_{n\geq0}(h\delta_A)^nh \]
	of $(A,\breve{d}_A)$ onto $(B,\breve{d}_B)$.
	
\end{lemma}

\begin{lemma}\label{lem:semifullstable} The class of semifull algebra contractions is stable under arbitrary perturbations. The class of semifull DG algebra contractions is stable under perturbations by algebra derivations.
\end{lemma}
\begin{proof} The first claim follows immediately from Definition \ref{def:semifullalgebracontraction}. The second claim was proved in \cite{Real}, see also \cite[Proposition 2.17]{BCSX}. We further remark that it follows immediately from Proposition \ref{prop:transfer}. In fact, by the first claim we know already that the perturbed contraction $(\breve{\sigma},\breve{\tau},\breve{h})$ is a semifull algebra contraction, and if both $d_A$ and the perturbation $\delta_A$ are algebra derivations so is the perturbed differential $\breve{d}_A$ (cf. also the previous Remarl \ref{rem:semifullDG}). 
\end{proof}

\subsection{Cocumulants and Koszul cobrackets}\label{sec:coetcet} All of the previous constructions and results admit dual versions in the context of coalgebras.

Here and in the rest of the paper, we shall always work with coalgebras which are coassociative, cocommutative, counitary, cougmented and cocomplete. Given such a coalgebra $C$, we shall denote: by $\Delta_C:C\to C\otimes C$ the coproduct; by $\epsilon_C:C\to\mathbb{K}$ the counit; by $\mathbb{K}\to C:1\to\mathbf{1}_C$ the coaugmentation; by $\overline{C}=\operatorname{Ker}(\epsilon_C)$ the reduced coalgebra, with the reduced coproduct \[\overline{\Delta}_C:\overline{C}\to\overline{C}\otimes\overline{C}:c\to\overline{\Delta}_C(c):=\Delta_C(c)-\mathbf{1}_C\otimes c-c\otimes\mathbf{1}_C\] 
(notice in particular that $C=\mathbb{K}\mathbf{1}_C\oplus\overline{C}$); by $\overline{\Delta}^{n-1}_C:\overline{C}\to\overline{C}^{\otimes n}$ the iterated coproducts. We say that $C$ is \emph{cocomplete} if $\overline{C}=\bigcup_{n\geq1}\operatorname{Ker}(\overline{\Delta}^{n-1}_C)$, or in other words if for any $c\in \overline{C}$ there exists $N=N(c)\gg0$ such that $\overline{\Delta}^{N}(c)=0$.

\begin{remark} In the rest of this paper, all the coalgebras considered shall be coassociative, cocommutative, counitary, coaugmented and cocomplete, \textbf{without further mention of these properties}.\end{remark}

We shall use the following notation
\[\End_{cu}(C):=\{\delta\in\End(C)\,\,\mbox{s.t.}\,\,\epsilon_C\circ \delta = 0,\,\delta(\mathbf{1}_C)=0 \}.\] 

Moreover, given a pair of counitary cocommutative of graded coalgebras $C$ and $D$ we shall denote by \[\Hom^0_{cu}(C,D):=\{f\in\Hom^0(C,D)\,\,\mbox{s.t.}\,\,\epsilon_D\circ f = \epsilon_C,\,f(\mathbf{1}_C)=\mathbf{1}_D \}.\] 
Given $f\in\Hom^0_{cu}(C,D)$, the \emph{cocumulants} of $f$ are degree zero maps $\kappa^{co}(f)_n:C\to D^{\odot n}$ measuring the deviation of $f$ from being a morphism of graded coalgebras. Similarly, given $\delta\in\End^k_{cu}(C)$, the associated \emph{Koszul cobrackets} $\Kos^{co}(\delta)_n:C\to C^{\odot n}$ are degree $k$ maps measuring the deviation of $\delta$ from being a coalgebra coderivation.

The most convenient way to introduce these maps is via the dual of formulas \eqref{eq:defcumulants} and \eqref{eq:defkoszulbrackets} respectively. More precisely, for $n\geq1$ we denote as before by $\Delta^{n-1}_C:C\to C^{\otimes n}$ the iterated coproduct ($\Delta^1_C=:\id_{C}$), by $\pi:C^{\otimes n}\to C^{\odot n}$ the projection and by $E^{co}:\widehat{S}(C)\to \widehat{S}(C)$ the unique unitary algebra automorphism extending
\[ E^{co}(x) = \sum_{n\geq1} \frac{1}{n!}\pi \Delta^{n-1}_C(x), \]
for $x\in C$ (where we denote by $\widehat{S}(C)=\prod_{n\ge0}C^{\odot n}$ the completed symmetric algebra over $C$). We denote by $L^{co}:\widehat{S}(C)\to \widehat{S}(C)$ the inverse of $E^{co}$. It is the unique unitary algebra automorphism extending
\[L^{co}(x) = \sum_{n\geq1} \frac{(-1)^{n-1}}{n}\pi\Delta^{n-1}_C(x),\qquad x\in C. \]
We call $E^{co}$ and $L^{co}$ respectively the \emph{coexponential and cologarithmic automorphisms} of $\widehat{S}(C)$.

Given $f\in\Hom^0_{cu}(C,D)$, we denote by $\widehat{\kappa}^{co}(f):\widehat{S}(C)\to \widehat{S}(D)$ the morphism of unitary graded algebras 
given by the composition 
\[ \widehat{\kappa}^{co}(f)\colon \widehat{S}(C)\xrightarrow{L^{co}} \widehat{S}(C)\xrightarrow{\widehat{S}(f)} \widehat{S}(D)\xrightarrow{E^{co}} \widehat{S}(D).\]
The \emph{cocumulants} $\kappa^{co}(f)_n:C\to D^{\odot n}$ are then defined as the composition \[ \kappa^{co}(f)_n\colon C\hookrightarrow \widehat{S}(C)\xrightarrow{\widehat{\kappa}^{co}(f)}\widehat{S}(D)\twoheadrightarrow D^{\odot n}, \]
where the left and right maps are the canonical inclusions and projections respectively. 

It is easy to compute explicitly the first few cocumulants. To simplify formulas, we adopt Sweedler's notation and write $\Delta^{n-1}_C(x) = \sum_{(x)} x_{(1)}\otimes\cdots\otimes x_{(n)}$. With these notations, we have
\[ \kappa^{co}(f)_1(x) = f(x),\qquad \kappa^{co}(f)_2(x) = \frac{1}{2}\sum_{(f(x))}  f(x)_{(1)}\odot f(x)_{(2)} - \frac{1}{2}\sum_{(x)} f(x_{(1)})\odot f(x_{(2)})  \]\begin{multline*} \kappa^{co}(f)_3(x) = \frac{1}{6}\sum_{(f(x))} f(x)_{(1)}\odot f(x)_{(2)}\odot f(x)_{(3)} -\frac{1}{4}\sum_{(x),(f(x_{(1)}))} f(x_{(1)})_{(1)}\odot f(x_{(1)})_{(2)}\odot f(x_{(2)})-\\ -\frac{1}{4}\sum_{(x),(f(x_{(2)}))} f(x_{(1)})\odot f(x_{(2)})_{(1)}\odot f(x_{(2)})_{(2)} + \frac{1}{3}\sum_{(x)} f(x_{(1)})\odot f(x_{(2)})\odot f(x_{(3)})=\\ =\frac{1}{6}\sum_{(f(x))} f(x)_{(1)}\odot f(x)_{(2)}\odot f(x)_{(3)} -\frac{1}{2} \sum_{(x),(f(x_{(1)}))}f(x_{(1)})_{(1)}\odot f(x_{(1)})_{(2)}\odot f(x_{(2)})+\\  + \frac{1}{3} \sum_{(x)}f(x_{(1)})\odot f(x_{(2)})\odot f(x_{(3)}),  
\end{multline*}
where in the last passage we used the fact that $\sum_{(x),(f(x_{(1)}))} f(x_{(1)})_{(1)}\odot f(x_{(1)})_{(2)}\odot f(x_{(2)})= \sum_{(x),(f(x_{(2)}))}f(x_{(1)})\odot f(x_{(2)})_{(1)}\odot f(x_{(2)})_{(2)}$, due to cocommutativity of $\Delta_C$. 

Notice that $f\in\Hom^0_{cu}(C,D)$ is a coalgebra morphism if and only if $\kappa^{co}(f)_n=0$ for all $n\ge2$. 

\begin{remark}\label{rem:recursioncocumulants} The cocumulants of $f$ obey a recursion dual to \eqref{eq:recursioncumulants}. More precisely, define maps $\widetilde{\kappa}^{co}(f)_{n}:C\to C^{\otimes n}$, $n\ge1$, recursively by $\widetilde{\kappa}^{co}(f)_{1}=f$ and 
	\[ \widetilde{\kappa}^{co}(f)_{n+1} = \left(\id_D^{\otimes n-1}\otimes \Delta_D\right)\widetilde{\kappa}^{co}(f)_n-\sum_{k=0}^{n-1}\left(\shuffle_{k,n-k-1}\otimes\id_D^{\otimes2}\right)\tau_k\big(\widetilde{\kappa}^{co}(f)_{k+1}\otimes\widetilde{\kappa}^{co}(f)_{n-k}\big)\Delta_C, \]
	where we denote by $\tau_k:D^{\otimes n+1}\to D^{\otimes n+1}$ the permutation \[\tau_k(y_1\otimes\cdots\otimes y_k\otimes y_{k+1}\otimes y_{k+2}\otimes \cdots\otimes y_n\otimes y_{n+1})=\pm_K y_1\otimes\cdots\otimes y_k\otimes y_{k+2}\otimes \cdots\otimes y_n\otimes y_{k+1}\otimes y_{n+1}\] and by $\shuffle_{k,n-k-1}$ the $(k,n-k-1)$-component $\shuffle_{k,n-k-1}:D^{\otimes n-1}=D^{\otimes k}\otimes D^{\otimes n-k-1}\xrightarrow{\shuffle} D^{\otimes n-1}$ of the shuffle product $\shuffle$. Thus for instance a straightforward computation shows that $\widetilde{\kappa}^{co}(f)_2=\Delta_Df-f^{\otimes 2}\Delta_C$ and
	$\widetilde{\kappa}^{co}(f)_3=\Delta_D^2f-\big(f\shuffle\Delta_Df\big)\Delta_C+2f^{\otimes 3}\Delta_C^2$. 
	It can be shown that the image of $\widetilde{\kappa}^{co}(f)_n$ is contained in the $S_n$-invariant part of $U^{\otimes n}$, and the cocumulants $\kappa^{co}(f):C\to C^{\odot n}$ are given by $\kappa^{co}(f)_n=\frac{1}{n!}\pi\widetilde{\kappa}^{co}(f)_n$, where we denote by $\pi:C^{\otimes n}\to C^{\odot n}$ the natural projection.
\end{remark}

\begin{remark}\label{rem:conilpotencycocumulants} Assuming, as we are, that $C$ and $D$ are cocomplete, it is not hard to see that  $\widehat{\kappa}^{co}(f)(c)\in S(D)\subset\widehat{S}(D)$ for all $c\in C$, or in other words that for all $c\in C$ there exists $n(c)\gg0$ such that $\kappa^{co}(f)_{N}(c)=0$ for all $N\ge n(c)$. Thus $\widehat{\kappa}^{co}(f):\widehat{S}(C)\to\widehat{S}(D)$ restricts to a morphism of unitary graded algebras \[\kappa^{co}(f)=\sum_{n\ge1}\kappa^{co}(f)_n:S(C)\to S(D). \]
In fact, one can show a stronger statement. If $f\in\Hom^0_{cu}(C,D)$, then $f(\mathbf{1}_C)=\mathbf{1}_D$ and $f$ restricts to $\overline{f}:=f_{|\overline{C}}:\overline{C}\to\overline{D}$. We denote by $\overline{E}^{co}:S(\overline{D})\to S(\overline{D})$ and $\overline{L}^{co}:S(\overline{C})\to S(\overline{C})$ the morphisms of unitary graded algebras extending
\[ \overline{E}^{co}(y) = \sum_{n\geq1} \frac{1}{n!}\pi \overline{\Delta}^{n-1}_D(y),\qquad \overline{L}^{co}(x) = \sum_{n\geq1} \frac{(-1)^{n-1}}{n}\pi\overline{\Delta}^{n-1}_C(x),\qquad x\in \overline{C},y\in\overline{D}, \]
where $\overline{\Delta}^{n-1}_C:\overline{C}\to \overline{C}^{\otimes n}$, $\overline{\Delta}_D^{n-1}:\overline{D}\to \overline{D}^{\otimes n}$ are the iterated reduced coproducts, and the above infinite sums make sense since $C$ and $D$ are cocomplete. Finally, we define $\kappa^{co}(\overline{f}):=\overline{E}^{co}\circ S(\overline{f})\circ \overline{L}^{co}:S(\overline{C})\to S(\overline{D})$.  With these notations, it is not hard to show that the following diagram is commutative
	\[\xymatrix{ \overline{C}\ar[d]\ar[r]^-{\kappa^{co}(\overline{f})}& S(\overline{D})\ar[d] \\ C\ar[r]_-{\widehat{\kappa}^{co}(f)} & \widehat{S}(D) } \]
where the vertical arrows are the natural inclusions. This implies the claim at the beginning of the remark, and shows moreover that we may compute the cocumulants using the reduced coproducts in place of the unreduced ones. This applies in particular to the previous explicit formulas for $\kappa^{co}(f)_n$, $n\le3$, where Sweedler's notation might be intended for the reduced coproducts  instead of the unreduced ones, and to the previous Remark \ref{rem:recursioncocumulants}, where in the recursive formula for $\widetilde{\kappa}^{co}(f)_{n+1}$ we might replace $\Delta_C,\Delta_D$ with $\overline{\Delta}_C,\overline{\Delta}_D$. 

\end{remark}

Given $\delta\in\End_{cu}(C)$, we denote by $\widehat{\Kos}^{co}(\delta):\widehat{S}(C)\to \widehat{S}(C)$ the derivation defined as the composition
\[ \widehat{\Kos}^{co}(\delta)\colon \widehat{S}(C)\xrightarrow{L^{co}} \widehat{S}(C)\xrightarrow{\widetilde{\delta}} \widehat{S}(C)\xrightarrow{E^{co}} \widehat{S}(C),\]
where as in subsection \ref{sec:cumulantsandkoszul} we denote by $\widetilde{\delta}: \widehat{S}(C)\to \widehat{S}(C)$ the linear derivation extending $\delta$. The \emph{Koszul cobrackets} $\Kos^{co}(\delta)_n:C\to D^{\odot n}$ are defined as the compositions \[ \Kos^{co}(\delta)_n\colon C\hookrightarrow \widehat{S}(C)\xrightarrow{\Kos^{co}(\delta)}\widehat{S}(C)\twoheadrightarrow C^{\odot n}. \]
The first few Koszul cobrackets of $\delta\in\End_{cu}(C)$ are
\[ \Kos^{co}(\delta)_1(x) = \delta(x),\qquad \Kos^{co}(\delta)_2(x) = \frac{1}{2}\sum_{(\delta(x))} \delta(x)_{(1)}\odot \delta(x)_{(2)} - \sum_{(x)}\delta(x_{(1)})\odot x_{(2)}\]\begin{multline*} \Kos^{co}(\delta)_3(x) = \frac{1}{6}\sum_{(\delta(x))} \delta(x)_{(1)}\odot \delta(x)_{(2)}\odot \delta(x)_{(3)} - \\ -\frac{1}{2} \sum_{(x),(\delta(x_{(1)}))} \delta(x_{(1)})_{(1)}\odot \delta(x_{(1)})_{(2)}\odot x_{(2)}+ \frac{1}{2} \sum_{(x)} \delta(x_{(1)})\odot x_{(2)}\odot x_{(3)}.
\end{multline*}
We notice that $\delta$ is a colagebra coderivation if and only if $\Kos^{co}(\delta)_n=0$ for all $n\ge2$.
\begin{remark}\label{rem:recursioncoKos} The Koszul cobrackets of $\delta$ obey a recursion dual to \eqref{eq:recursionkoszul}. More precisely, define maps $\widetilde{\Kos}^{co}(\delta)_{n}:C\to C^{\otimes n}$, $n\ge1$, recursively by $\widetilde{\Kos}^{co}(\delta)_{1}=\delta$ and 
	\[ \widetilde{\Kos}^{co}(\delta)_{n+1} = \left(\id_C^{\otimes n-1}\otimes \Delta_C\right)\widetilde{\Kos}^{co}(\delta)_n-(\id_C^{\otimes n+1}+\tau_{n,n+1})\big(\widetilde{\Kos}^{co}(f)_{n}\otimes\id_C\big)\Delta_C, \]
	where we denote by $\tau_{n,n+1}:C^{\otimes n+1}\to C^{\otimes n+1}$ the transposition $\tau_{n,n+1}(c_1\otimes\cdots\otimes c_n\otimes c_{n+1})=\pm_Kc_1\otimes \cdots\otimes c_{n+1}\otimes c_n$. Thus for instance \[\widetilde{\Kos}^{co}(\delta)_2=\Delta_C\delta-(\delta\otimes\id_C+\id_C\otimes\delta)\Delta_C\] 
	\[\widetilde{\Kos}^{co}(\delta)_3=\Delta_C^2\delta-\big(\id_C\shuffle\Delta_C\delta\big)\Delta_C+(\delta\otimes\id_C\otimes\id_C+\id_C\otimes\delta\otimes\id_C+\id_C\otimes\id_C\otimes\delta)\Delta_C^2.\] 
	It can be shown that the image of $\widetilde{\Kos}^{co}(\delta)_n$ is contained in the $S_n$-invariant part of $U^{\otimes n}$, and the Koszul cobrackets $\Kos^{co}(\delta):C\to C^{\odot n}$ are given by $\Kos^{co}(\delta)_n=\frac{1}{n!}\pi\widetilde{\Kos}^{co}(\delta)_n$, where we denote by $\pi:C^{\otimes n}\to C^{\odot n}$ the natural projection.
\end{remark}
\begin{remark}\label{rem:conilpotencycoKos} As in the previous remark \ref{rem:conilpotencycocumulants}, when $C$ is cocomplete we can define a derivation $\Kos^{co}(\overline{\delta}):=\overline{E}^{co}\circ\widetilde{\overline{\delta}}\circ \overline{L}^{co}:S(\overline{C})\to S(\overline{C})$, where we denote by $\overline{\delta}$ the restriction $\overline{\delta}:=\delta_{|\overline{C}}:\overline{C}\to \overline{C}$ and by $\widetilde{\overline{\delta}}:S(\overline{C})\to S(\overline{C})$ the linear derivatione extending it. It is not hard to check that the following diagram is commutative 
	\[\xymatrix{ \overline{C}\ar[d]\ar[r]^-{\Kos^{co}(\overline{\delta})}& S(\overline{C})\ar[d] \\ C\ar[r]_-{\widehat{\Kos}^{co}(\delta)} & \widehat{S}(C) } \]
	where the vertical arrows are the natural inclusions.
	In particular $\widehat{\Kos}^{co}(\delta)(c)\in S(C)\subset\widehat{S}(C)$ for all $c\in C$, and thus $\widehat{\Kos}^{co}(\delta):\widehat{S}(C)\to\widehat{S}(C)$ restricts to a derivation \[\Kos^{co}(\delta)=\sum_{n\ge1}\Kos^{co}(\delta)_n:S(C)\to S(C). \]
\end{remark}

All the results from subsection \ref{sec:cumulantsandkoszul} (namely, Lemmas \ref{prop:cumulantsandcomposition}, \ref{lem:KosDGLiemor} and Propositions \ref{prop:cumulantsvsKoszuldg}, \ref{prop:cumulantsvsKoszulexp}) have analogues in this setting, which can be proved by the same arguments, mutatis mutandis. We shall also need to consider the analogue of Definition \ref{def:diffop}.

\begin{definition}\label{def:codiffop}
	Given a graded (cocommutative, et cet.) coalgebra $C$ and $k\geq1$, we denote by 
	\[ \operatorname{coDiff}_{cu,\leq k}(C) := \{ \delta\in\End_{cu}(C)\,\,\mbox{s.t.}\,\,\Kos^{co}(\delta)_{k+1}=0 \}  \]
(that is, the subset of $\End_{cu}(C)$ consisting of codifferential operators of order $\le k$). We denote by $\operatorname{coDiff}_{cu}(C) = \bigcup_{k\geq 1} \operatorname{coDiff}_{cu,\leq k}(C)$. It is easy to check that  $\operatorname{coDiff}_{cu}(C)\subset\End_{cu}(C)$ is both a graded subalgebra and a graded Lie subalgebra. 
\end{definition}
We turn to the results from subsection \ref{subsec:comp}. 

\begin{definition} Given a pair of graded (cocommutative, et cet.) coalgebras $(C,\Delta_C,\epsilon_C)$ and $(D,\Delta_D,\epsilon_D)$ equipped with differentials $d_C,d_D$ (we assume $d_C\in\End_{cu}(C)$, $d_D\in\End_{cu}(D)$, but not require $d_C,d_D$ to be coderivations), a \emph{semifull coalgebra contraction} $(\sigma,\tau,h)$ of the former onto the latter is a contraction of $(C,d_C)$ onto $(D,d_D)$ satisfying the further requirements \[\sigma\in\Hom^0_{cu}(C,D),\qquad\tau\in\Hom^0_{cu}(D,C),\qquad h\in\End_{cu}(C),\]\begin{multline*}
	 (h\otimes h)\circ\Delta_C \circ h = (h\otimes\sigma)\circ\Delta_C\circ h=(\sigma\otimes\sigma)\circ\Delta_C\circ h=0,\\ (h\otimes h)\circ\Delta_C \circ\tau = (h\otimes\sigma)\circ\Delta_C\circ\tau=0,\quad (\sigma\otimes\sigma)\circ\Delta_C\circ\tau=\Delta_D. 
	 \end{multline*}

In the above hypotheses, if furthermore $d_C$ is a coalgebra coderivations we say that $(\sigma,\tau,h)$ a \emph{semifull DG coalgebra contraction}.
\end{definition}

\begin{remark} The following Proposition \ref{prop:cotransfer} will imply immediately that if $(\sigma,\tau,h)$ is a semifull DG algebra contraction then necessarily $d_D$ is a coderivation and $\sigma$ is a morphism of graded coalgebras.
\end{remark}

\begin{lemma}\label{lem:semifullcostable} The class of semifull contractions of coalgebras with a differential is stable under arbitrary perturbations. The class of semifull DG coalgebra contractions is stable under perturbations by coalgebra coderivations.
\end{lemma}
\begin{proof} It follows straightforwardly from the definitions. \end{proof}
\begin{proposition}\label{prop:cotransfer} Given graded (cocommutative, et cet.) coalgebras $(C,\Delta_C,\epsilon_C)$, $(D,\Delta_D,\epsilon_D)$, together with differentials (not necessarily coderivations) $d_C\in\End^1_{cu}(C),d_D\in\End^1_{cu}(D)$ and a semifull coalgebra contraction $(\sigma,\tau,h)$ of $(C,d_C)$ onto $(D,d_D)$, the Koszul cobrackets $\Kos^{co}(d_D)_n$ and the cocumulants $\kappa^{co}(\sigma)_n$ are induced via homotopy transfer from the $L_\infty[-1]$ coalgebra structure on $C$ associated with the Koszul cobrackets $\Kos^{co}(d_C)_n$. 
\end{proposition}

\begin{proof} Reasoning as at the end of the proof of \ref{prop:transfer}, it is enough to prove the statement about the cocumulants. 
	
We have to show that the morphism of graded algebras $\kappa^{co}(\sigma):S(C)\to S(D)$ obeyes (and is recursively defined from)
	 
	\[ \kappa^{co}(\sigma)(c) = \sigma(c) + \kappa^{co}(\sigma)\Big(\Kos^{co}(d_C) _+\Big(h(c)\Big)\Big) \]
for all $c\in C$, where we denote by $\Kos^{co}(d_C)_+:=\Kos^{co}(d_C)-\widetilde{d_C}:S(C)\to S(C)$ (here as usual $\widetilde{d_C}$ is the linear derivation extending $d_C$). These recursions are the duals of the recursions for $\kappa(\tau)$, $\Kos(d_B)$, at the beginning of the proof of Proposition \ref{prop:transfer}. They can be shown by dualizing the computations jn the same proof. For instance, we have to show that the first two cocumulants of $\sigma$ are $\kappa^{co}(\sigma)_1=\sigma$ (it is so by definition) and $\kappa^{co}(\sigma)_2 = (\sigma\odot\sigma)\Kos^{co}(d_C)_2h$. For the latter identity, using the fact that $\Delta_D=(\sigma\otimes\sigma)\Delta_C\tau$ and $(\sigma\otimes\sigma)\Delta_C h=0$ since $(\sigma,\tau,h)$ is a semifull coalgebra contraction, we see that (where $\pi:C^{\otimes 2}\to C^{\odot 2}$ is the natural projection)
\[\kappa^{co}(\sigma)_2=\frac{1}{2}\pi\Big(\Delta_D\sigma-(\sigma\otimes\sigma)\Delta_C\Big) = \frac{1}{2}\pi(\sigma\otimes\sigma)\Delta_C(\tau\sigma-\id_C) = \frac{1}{2}\pi(\sigma\otimes\sigma)\Delta_Cd_Ch. \]
On the other hand $\Kos^{co}(d_C)_2=\frac{1}{2}\pi\Big(\Delta_Cd_C-(d_C\otimes \id_C+\id_C\otimes d_C)\Delta_C\Big)$. Thus \[ (\sigma\odot\sigma)\Kos^{co}(d_C)_2h = \frac{1}{2}\pi(\sigma\otimes\sigma)\Big(\Delta_Cd_C-(d_C\otimes \id_C+\id_C\otimes d_C)\Delta_C\Big)h =\frac{1}{2}\pi(\sigma\otimes\sigma)\Delta_Cd_Ch=\kappa^{co}(\sigma)_2, \]
where we used that $(\sigma\otimes\sigma)(d_C\otimes \id_C+\id_C\otimes d_C)\Delta_Ch=(d_D\otimes\id_D+\id_D\otimes d_D)(\sigma\otimes\sigma)\Delta_Ch=0$. 

In the general case the claim might be proved by adapting the recursive computation in the proof of Proposition \ref{prop:transfer} to this dual setting. 

Alternatively, one might reason as follows. We denote by $R=r_1+\cdots+r_n+\cdots$, $r_n:D\to D^{\odot n}$ $r_1=d_D$, the $L_\infty[-1]$ coalgebra structure on $D$ induced via homotopy transfer, and likewise we denote by $G=g_1+\cdots+g_n+\cdots$, $g_n:C\to D^{\odot n}$, $g_1=\sigma$, the induced $L_\infty[-1]$ coalgebra morphism. To show that $g_n=\kappa^{co}(\sigma)_n$, $\forall\, n\ge1$, it is enough to show that the dual maps coincide, i.e.,  $g_n^\vee=\kappa^{co}(\sigma)_n^\vee:(D^\vee)^{\odot n}\to (D^{\odot n} )^\vee\to C^\vee$. On the one hand, it follows directly from the definitions that cumulants and cocumulants are dual to each other, i.e., $\kappa^{co}(\sigma)^\vee_n=\kappa(\sigma^\vee)_n$, where the cumulants of $\sigma^\vee\in\Hom^0_u(D^\vee,C^\vee)$ are computed with respect to the dual algebra structures on $D^\vee$ and $C^\vee$. Similarly, Koszul brackets and cobrackets are dual to each other in the sense that $\Kos^{co}(d_C)_n^\vee = \Kos(d_C^\vee)_n$. On the other hand, it is easy to see that homotopy transfer is compatible with duality in the following sense: the $L_\infty[1]$ algebra structure $(r_1^\vee,\ldots,r_n^\vee,\ldots)$ on $D^\vee$ and the $L_\infty[1]$ morphism $(g_1^\vee,\ldots,g_n^\vee,\ldots):D^\vee\to C^\vee$ are induced by transferring the dual $L_\infty[1]$ algebra structure $(d^\vee_C,\ldots,\Kos^{co}(d_C)_n^\vee,\ldots) = (d_C^\vee,\ldots,\Kos(d_C^\vee)_n,\ldots)$ on $C^\vee$ along the dual contraction $(\sigma^\vee,\tau^\vee,h^\vee)$. Finally, putting these two observations togetether with Proposition \ref{prop:transfer} (notice that $(\sigma^\vee,\tau^\vee,h^\vee)$ is a semifull algebra contraction) we can conclude, as desired, that $g_n^\vee=\kappa(\sigma^\vee)_n=\kappa^{co}(\sigma)_n^\vee$.
 \end{proof}

\subsection{Homotopy transfer for $L_\infty[1]$-algebras (revisited)}\label{subsec:Lootransfer}

Given a pair of symmetric coalgebras $S(U),S(V)$, the graded space $\Hom(S(U),S(V))$ becomes a unitary graded  commutative algebra via the convolution product $\star$: explicitly,  
	\[ F\star G = \odot\circ(F\otimes G)\circ\Delta. \]
where $\Delta$ is the unshuffle coproduct on $S(U)$ and $\odot$ is the symmetric product on $S(V)$. The unit is the map $\varepsilon:S(U)\to S(V)$ defined by $\varepsilon(\1_{S(U)})=\1_{S(V)}$ and $\varepsilon(x_1\odot\cdots\odot x_n)=0$ for all $n\geq1$ and $x_1,\ldots,x_n\in U$.

Given $\varphi\in\Hom^0(S(U),S(V))$ such that $\varphi(\1_{S(U)})=0$, it is well defined (by cocompleteness of $S(U)$) the exponential $\exp_\star(\varphi)=\varepsilon+\sum_{k\geq1}\frac{1}{k!}\varphi^{\star k}$ of $\varphi$ with respect to the convolution product. Conversely, given $F\in\Hom^0(S(U),S(V))$ with $F(\1_{S(U)})=\1_{S(V)}$, it is well defined the logarithm 
$\log_\star(F)=\sum_{k\geq1}\frac{(-1)^{k-1}}{k}(F-\varepsilon)^{\star k}$ of $F$ with respect to the convolution product.

By definition, given $F\in\Hom^0_{cu}(S(U),S(V))$, it is a morphism of graded coalgebras if and only if the higher cocumulants vanish, i.e., if and only if $\kappa^{co}(F)_n=0$ for all $n\ge 2$. In the following lemma we assume instead $F\in\Hom^0_{u}(S(U),S(V))$ (that is, $F(\mathbf{1}_{S(U)})=\mathbf{1}_{S(V)}$), and give an equivalent condition for $F$ to be a morphism of graded coalgebras, this time in terms of the cumulants $\kappa(F)_n$.

\begin{lemma}\label{lem:coalgebramorphisms} Given a pair of graded spaces $U,V$ and $F\in\Hom^0_u(S(U),S(V))$, then $F$ is a morphism of graded coalgebras if and only if $\kappa(F)_n(U,\ldots,U)\subset V\subset S(V)$ for all $n\geq1$. If this happens, denoting by $f_n:U^{\odot n}\to V$, $n\geq1$, the maps defined by $f_n(x_1,\ldots,x_n)=\kappa(F)_n(x_1,\ldots,x_n)$, then $F$ is the unique morphism of coalgebras with Taylor coefficients $pF=(f_1,\ldots,f_n,\ldots)$ ($p:S(V)\to V$ being the natural projection).\end{lemma}
\begin{proof} A straightforward computation, using Formula \eqref{eq:cumulants} for the cumulants and the definition of the convolution product $\star$, shows that for all $n\geq1$ and $x_1,\ldots,x_n\in U$
	\begin{multline*} \kappa(F)_n(x_1,\ldots,x_n) = \\ = \sum_{\stackrel{k,i_1,\ldots,i_k\geq1}{i_1+\cdots+i_k=n}}\sum_{\sigma\in S(i_1,\ldots,i_k)}\pm_K\frac{(-1)^{k-1}}{k}F(x_{\sigma(1)}\odot\cdots)\odot\cdots\odot F(\cdots\odot x_{\sigma(n)}) = \log_\star(F)(x_1\odot\cdots\odot x_n) 
	\end{multline*}
	In particular, we see that \begin{multline*}  F(x_1\odot\cdots\odot x_n) = \exp_\star(\log_\star(F))(x_1\odot\cdots\odot x_n) =\\=  \sum_{\stackrel{k,i_1,\ldots,i_k\geq1}{i_1+\cdots+i_k=n}}\sum_{\sigma\in S(i_1,\ldots,i_k)}\pm_K\frac{1}{k!}\kappa(F)_{i_1}(x_{\sigma(1)},\ldots)\odot\cdots\odot\kappa(F)_{i_k}(\ldots,x_{\sigma(n)}).  
	\end{multline*}
	The above identity, together with the one \eqref{morfromtaylor}, shows that if there exist maps $f_n:U^{\odot n}\to V$ such that $\kappa(F)_n(x_1,\ldots,x_n)=f_n(x_1,\ldots,x_n)$, then $F$ is precisely the morphism of coalgebras with corestriction $pF=(f_1,\ldots,f_n,\ldots)$. Conversely, if $F$ is a morphism of  graded coalgebras with corestriction $pF=(f_1,\ldots,f_n,\ldots)$, the above identity and a straightforward induction show that $\kappa(F)_n(x_1,\ldots,x_n)=f_n(x_1,\ldots,x_n)\in V\subset S(V)$ for all $n\geq1$ and $x_1,\ldots,x_n\in U$, and thus that $\kappa(F)_n(U,\ldots,U)\subset V$.  
\end{proof}

Analagously, given $Q\in\End_{cu}(S(U))$, we know that it is a coderivation if and only if $\Kos^{co}(Q)_n=0$ for all $n\ge 2$. In the following lemma we assume $Q\in\End_{u}(S(U))$ and we give an equivalent condition for $Q$ to be a coderivation in terms of the Koszul brackets $\Kos(Q)_n$.

\begin{lemma}\label{lem:transkos} Given $Q\in\End_{u}(S(U))$, then $Q$ is a coderivation if and only if $\Kos(Q)_n(U,\ldots,U)\subset U\subset S(U)$ for all $n\geq1$. If this happens, denoting by $q_n:U^{\odot n}\to U$, $n\geq1$, the maps defined by $q_n(x_1,\ldots,x_n)=\Kos(Q)_n(x_1,\ldots,x_n)$, then $Q$ is the unique coalgebra coderivation with corestriction $pQ=(0,q_1,\ldots,q_n,\ldots)$ ($p:S(U)\to U$ being as usual the natural projection).
\end{lemma}
\begin{proof} We notice that the identity $\id_{S(U)}$ is invertible with respect to the convolution product of $\End(S(U))$, with inverse the map $s:S(U)\to S(U)$ defined by 
\[ s(\1_{S(U)})=\1_{S(U)},\qquad s(x_1\odot\cdots\odot x_k) =(-1)^k x_1\odot\cdots\odot x_k \]
for all $k\geq1$ and $x_1,\ldots,x_k\in U$ (in other words, $s$ is the antipode of the natural Hopf algebra structure on $S(U)$). The explicit formula \eqref{eq:koszul} for the Koszul brackets $\mathcal{K}(Q)_n$ shows that
\[ \Kos(Q)_n(x_1,\ldots,x_n) = (Q\star s)(x_1\odot\cdots\odot x_n). \]
In other words, denoting by $i:U\to S(U)$ the natural inclusion, by $S(i)$ the induced inclusion $S(U)\to S(S(U)):x_1\odot\cdots\odot x_n\to i(x_1)\odot\cdots\odot i(x_n)$, and by $k:S(U)\to S(U)$ the composition
\[ k: S(U)\xrightarrow{S(i)} S(S(U))\xrightarrow{\Kos(Q)}S(S(U))\xrightarrow{p} S(U),\]
the above shows that $Q\star s = k$ in the convolution algebra $\End(S(U))$, and thus that $Q=k\star\operatorname{id}_{S(U)}$, that is, 
\begin{equation}\label{eq:Q}  Q(x_1\odot\cdots\odot x_n) = \sum_{i=1}^n\sum_{\sigma\in S(i,n-i)} \pm_K \Kos(Q)_i(x_{\sigma(1)},\ldots,x_{\sigma(i)})\odot x_{\sigma(i+1)}\odot\cdots\odot x_{\sigma(n)}
\end{equation}
for all $n\geq1$ and $x_1,\ldots,x_n\in U$. If $\Kos(Q)_n(x_1,\ldots,x_n)=q_n(x_1,\ldots,x_n)\in U$ for certain $q_n:U^{\odot n}\to U$, then \eqref{coderfromtaylor} shows that $Q$ is precisely the coderivation with Taylor coefficients $q_0=0,q_1,\ldots,q_n,\ldots$. Conversely, if $Q$ is a coderivation, the above and a straightforward induction show that $\Kos(Q)_n(x_1,\ldots,x_n)=q_n(x_1,\ldots,x_n)\in U$ for all $n\geq1$ and $x_1,\ldots,x_n\in U$, and thus that $\Kos(Q)_n(U,\ldots,U)\subset U\subset S(U)$, as desired.   
\end{proof}

As a final preparation, we shall show that the contraction obtained from the symmetrized tensor trick is a semifull contraction with respect to both the algebraic and coalgebraic structures.

\begin{lemma}\label{lem:symmetrizedcontraction} Given a pair of complexes $(U,d_U),(V,d_V)$ and a contraction $(\sigma,\tau,h)$ of $(U,d_U)$ onto $(V,d_V)$, there is an induced contraction $(S(\sigma),S(\tau),\widehat{h})$ of $(S(U),\widetilde{d_U})$ onto $(S(V),\widetilde{d_V})$, where the contracting homotopy $\widehat{h}$ is given by the symmetrized tensor trick: explicitly,  $\widehat{h}(\1_{S(U)})=0$ and
	\begin{equation*} \widehat{h}(x_1\odot\cdots\odot x_n) =  \frac{1}{n!}\sum_{j=1}^{n}\sum_{\sigma\in S_n}\pm_K \tau\sigma(x_{\sigma(1)})\odot\cdots\odot\tau\sigma(x_{\sigma(j-1)})\odot h(x_{\sigma(j)})\odot x_{\sigma(j+1)}\odot\cdots\odot x_{\sigma(n)},
	\end{equation*}
	 where $\pm_K$ is the Koszul sign associated to $h,x_1,\ldots,x_n\mapsto x_{\sigma(1)},\ldots,x_{\sigma(j-1)},h,x_{\sigma(j)},\ldots,x_{\sigma(n)}$ (keeping in mind that $|h|=-1$). This is both a semifull DG algebra contraction with respect to the symmetric products and a semifull DG coalgebra contraction with respecto to the unshuffle coproducts (in other words, we might call it a semifull DG bialgebra contraction).
\end{lemma}
\begin{proof} The fact that $(S(\sigma),S(\tau),\widehat{h})$ defines a contraction of $(S(U),\widetilde{d_U})$ onto $(S(V),\widetilde{d_V})$ is well known, and in any case easy to show.  Furthermore, it is clear that $S(\sigma),S(\tau)$ are both algebra and coalgebra morphisms, which immediately implies some of the necessary relations. The only  relations to be shown which are not immediate are the following ones
	\[ \widehat{h}\circ\odot\circ(\widehat{h}\otimes \widehat{h}) = 0,\quad \widehat{h}\circ\odot\circ(\widehat{h}\otimes S(\tau)) = 0,\quad (\widehat{h}\otimes \widehat{h})\circ \Delta\circ \widehat{h}=0,\quad(\widehat{h}\otimes S(\sigma))\circ \Delta\circ \widehat{h}=0, \]
	where $\odot$ and $\Delta$ are respectively the symmetric product and the unshuffle coproduct on $S(U)$. For the latter two, by a standard polarization argument it is enough to show them on elements of the form $x^{\odot n}$, for some $x\in U^0$ and $n\ge1$. In this case, we notice that \[\widehat{h}\Big(h(x)\odot\tau\sigma(x)^{\odot i}\odot x^{\odot j}\Big)=0,\qquad S(\sigma)\Big(h(x)\odot\tau\sigma(x)^{\odot i}\odot x^{\odot j}\Big)=0,\] for all $i,j\ge0$, since $h^2=h\tau=\sigma h=0$ and $h(x)\odot h(x)=0$ by degree reasons. Thus, denoting by $K$ either $\widehat{h}$ or $S(\sigma)$, 
	\begin{multline*} (\widehat{h}\otimes K)\Delta \widehat{h}(x^{\odot n})=(\widehat{h}\otimes K)\Delta\left( \sum_{i+j=n-1}  h(x)\odot \tau\sigma(x)^{\odot i}\odot x^{\odot j}\right) =\\=\sum_{p+q+r+s=n-1} \binom{p+r}{p}\binom{q+s}{q}\widehat{h}\Big(h(x)\odot \tau\sigma(x)^{\odot p}\odot x^{\odot q}\Big)\otimes K\Big(\tau\sigma(x)^{\odot r}\odot x^{\odot s}\Big)+\\+\binom{p+r}{p}\binom{q+s}{q}\widehat{h}\Big(\tau\sigma(x)^{\odot p}\odot x^{\odot q}\Big)\otimes K\Big(h(x)\odot \tau\sigma(x)^{\odot r}\odot x^{\odot s}\Big) =0.	\end{multline*} 
	The first two identities are shown similarly. First of all, it is enough to prove them on elements of the form $x^{\odot m}\otimes y^{\odot n}$, with $x$ and $y$ of degree zero. For this, we compute
	\[ \widehat{h}\Big(\widehat{h}(x^{\odot m})\odot \widehat{h}(y^{\odot n})\Big)=\sum_{i+j=m-1}\sum_{p+q=n-1}\widehat{h}\Big( h(x)\odot h(y)\odot\tau\sigma(x)^{\odot i}\odot\tau\sigma(y)^{\odot p}\odot x^{\odot j}\odot y^{\odot q} \Big)=0, \]
	\[ \widehat{h}\Big(\widehat{h}(x^{\odot m})\odot \tau(y)^{\odot n}\Big)=\sum_{i+j=m-1}\widehat{h}\Big( h(x)\odot\tau\sigma(x)^{\odot i}\odot x^{\odot j}\odot\tau(y)^{\odot n}\Big)=0, \]
once again, using the identities $h^2=h\tau=0$ and the fact that $h(x)^{\odot 2}=h(y)^{\odot 2}=0$ by degree reasons.
\end{proof}

We are ready to revisit the homotopy transfer Theorem \ref{th:transfer} for $L_\infty[1]$ algebras, providing a new proof based on the results from this section. As presented here, the argument might seem a bit circular, as we used Theorem \ref{th:transfer} in the proof of Proposition \ref{prop:transfer}: this issue is addressed in the following Remark \ref{rem:circular}. For our purposes, this will serve more as a preparation for the similar proofs of the homotopy transfer theorems for commutative $BV_\infty$ and $IBL_\infty$ algebras (Theorems \ref{th:transferBV} and \ref{th:transferIBL} respectively). 
We point out that a simpler proof of homotopy transfer for (curved) $L_\infty[1]$ algebras can be given in the spirit of this paper, but avoiding the use of Proposition \ref{prop:transfer} by replacing it with a much simpler argument: this will appear elsewhere \cite{prep}.

Let $(g,f,h)$ be a contraction of the complex $(U,d_U)$ onto the one $(V,d_V)$, and let $Q$ be an $L_\infty[1]$ structure on $U$ with linear part $q_1=d_U$. There is an induced contraction $(S(g),S(f),\widehat{h})$ of $(S(U),\widetilde{d_U})$ onto $(S(V),\widetilde{d_V})$ as in the previous lemma. Putting $Q_+=Q-\widetilde{d_U}$, and regarding it as a perturbation of the differential $\widetilde{d_U}$ on $S(U)$, by the Standard Perturbation Lemma \ref{lem:SPL} there is an induced perturbation $R_+=\sum_{k\geq1} S(\sigma)Q_+\big(\widehat{h}Q_+\big)^k S(\tau)$ of the differential $\widetilde{d_V}$ on $S(V)$, as well as a perturbed contraction $(G,F,H)$ of $(S(U),Q)$ onto $(S(V),R:=\widetilde{d_V}+R_+)$.

\begin{remark}\label{rem:H} By the previous Lemma \ref{lem:symmetrizedcontraction}, we know that $(S(g),S(f),\widehat{h})$ is both a semifull DG coalgebra contraction and a semifull DG algebra contraction of $(S(U),\widetilde{d_U})$ onto $(S(V),\widetilde{d_V})$, and since the perturbation $Q_+$ is a coderivation of $S(U)$, according to Lemmas \ref{lem:semifullstable} and \ref{lem:semifullcostable} we know that $(G,F,H)$ is both a semifull DG coalgebra contraction and a semifull algebra contraction (not a semifull DG algebra contraction, though). For future reference, we also point out that since both $\widehat{h}$ and $Q_+$ preserve the subspaces $S_{\le n}(U):=\oplus_{i=0}^n U^{\odot i}$, so does the perturbed homotopy $H=\sum_{k\ge0}(\widehat{h}Q_+)^k\widehat{h}$. Finally, we notice that since the subspace $U\subset S(U)$ is in the kernel of the perturbation $Q_+$, we have $H(x)=\sum_{k\ge0}(\widehat{h}Q_+)^k\widehat{h}(x) = h(x)\in U\subset S(U)$, $\forall\,x\in U$.
\end{remark}
\begin{theorem}\label{th:transferL} In the previous hypotheses, $R$ is an $L_\infty[1]$ algebra structure on $V$, and $G,F$ are $L_\infty[1]$ morphisms.
\end{theorem}
\begin{proof} As observed in the previous remark, $(G,F,H)$ is both a semifull DG coalgebra contraction and a semifull algebra contraction. In particular, $R$ is a coderivation and $G:S(U)\to S(V)$ is a morphism of graded coalgebras: in fact, the same computations as in the proof of Proposition \ref{prop:cotransfer} show that $\kappa^{co}(G)_2 = (G\odot G)\Kos^{co}(Q)_2H=0$, hence $G$ is a morphism of graded coalgebras, and $\Kos^{co}(R)_2 = (G\odot G)\Kos^{co}(Q)_2F=0$, hence $R$ is a coderivation. 
	
It remains to show that $F:S(V)\to S(U)$ is a morphism of graded coalgebras: this follows from Proposition \ref{prop:transfer} and the above Lemmas \ref{lem:coalgebramorphisms} and \ref{lem:transkos}. More precisely, starting with \[ \kappa(F)_1(x)=F(x)=\left(\sum_{k\geq0}(\widehat{h}Q_+)^kS(f)\right)(x) =\sum_{k\geq0}(\widehat{h}Q_+)^kf(x) = f(x)=:f_1(x)\in U\subset S(U),\] using Proposition \ref{prop:transfer}, Lemmas \ref{lem:coalgebramorphisms}, \ref{lem:transkos} and induction on $n$ we see that
\begin{multline*} \kappa(F)_n(x_1,\ldots,x_n) =\\= \sum_{\stackrel{k\geq2,i_1,\ldots,i_k\geq1}{i_1+\cdots+i_k=n}}\sum_{\sigma\in S(i_1,\ldots,i_k)}\pm_K \frac{1}{k!} H\Kos(Q_+)_k(\kappa(F)_{i_1}(x_{\sigma(1)},\ldots),\ldots,\kappa(F)_{i_k}(\ldots,x_{\sigma(n)}) )=\\ =\sum_{\stackrel{k\geq2,i_1,\ldots,i_k\geq1}{i_1+\cdots+i_k=n}}\sum_{\sigma\in S(i_1,\ldots,i_k)}\pm_K \frac{1}{k!} hq_k(f_{i_1}(x_{\sigma(1)},\ldots),\ldots,f_{i_k}(\ldots,x_{\sigma(n)}) )=\\=: f_n(x_1,\ldots,x_n) \in U\subset S(U).
\end{multline*} 
Using Lemma \ref{lem:coalgebramorphisms} we can conclude that $F$ is the morphism of graded coalgebras, and we also recovered the usual recursion for its Taylor coefficients $f_1,\ldots,f_n,\ldots$. \end{proof}

\begin{remark}\label{rem:circular}
As the previous proof depends on Proposition \ref{prop:transfer}, and in the claim of the latter we invoke the homotopy transfer Theorem \ref{th:transfer}, the whole argument might seem circular. A more careful look at the proofs shows that this is not really the case, as the only thing we actually need from Proposition \ref{prop:transfer} are the recursions stated at the beginning of its proof, and the rest of said proof is just a (rather long) direct computation to prove these recursions, and doesn't actually depend on Theorem \ref{th:transfer}.
\end{remark}

\begin{remark}\label{rem:Loocotrans} One might prove the homotopy transfer Theorem for $L_\infty[-1]$ coalgebras along the same lines. In fact, given an complexes $(U,d_U)$ and $(V,d_V)$ together with a contraction $(\sigma,\tau,h)$ of $U$ onto $V$, there is an induced contraction of $\widehat{S}(U)$ onto $\widehat{S}(V)$ as in Lemma \ref{lem:symmetrizedcontraction}. Given an $L_\infty[-1]$ coalgebra structure on $U$, i.e., a DG algebra structure on $\widehat{S}(U)$, we might regard it as a perturbation of the differential induced by $d_U$, and thus via the standard perturbation Lemma there are induced a perturbed differential on $\widehat{S}(V)$ and a perturbed contraction $(F,G,H)$ of $\widehat{S}(U)$ onto $\widehat{S}(V)$. Since the perturbation was an algebra derivation, it follows at once from Lemmas \ref{lem:symmetrizedcontraction} and \ref{lem:semifullstable} (see also Remark \ref{rem:semifullDG}) that the perturbed differential is an $L_\infty[-1]$ coalgebra structure on $V$, and that $F:\widehat{S}(V)\to\widehat{S}(U)$ is a morphism of $L_\infty[-1]$ coalgebras. The only thing that requires a little more work is to show that $G:\widehat{S}(U)\to\widehat{S}(V)$ is also a morphism of $L_\infty[-1]$ coalgebras. This can be achieved along the lines of the previous proof: we leave details to the interested reader.\end{remark}
\section{Derived $BV$ algebras}

\emph{Commutative $BV_\infty$ algebras} were introduced by O. Kravchenko \cite{Krav} and have been applied in several contexts, such as deformation quantization, quantum field theory and Poisson geometry, just to name a few, see for instance \cite{Hueb1,CL,DSV1,DSV2,BrLaz,Vit,BashVor,Campos,ParkQFT,LRS,AVor} (in the references \cite{ParkQFT,LRS} they are called \emph{binary QFT algebras}, but in fact the two notions seem to be equivalent).
As explained in the introduction, these are not homotopy $BV$ algebras in the full operadic sense \cite{GCTV}, hence the name might be misleading. In the rest of the paper we shall call \emph{derived BV algebras} the commutative $BV_\infty$ algebras in the sense of Kravchenko, following a terminology introduced in \cite{BCSX}.

\subsection{Derived $BV$ algebras}

\begin{definition} Let $A$ be a graded commutative algebra and $k\in\mathbb{Z}$ be an odd integer. Let $t$ be a central variable of (even) degree $1-k$: we denote by $A[[t]]$ the corresponding algebra of formal power series. A \emph{degree $k$ derived $BV$ algebra} structure on $A$ is the datum of a degree one $\mathbb{K}[[t]]$-linear map $\Delta = \sum_{n\geq0}t^n\Delta_n\in\End_{u,\K[[t]]}(A[[t]])=\K[[t]]\otimes\End_u(A)$, $|\Delta_n| = 1 + n(k-1)$, such that
	\begin{itemize} \item $\Delta^2=\frac{1}{2}[\Delta,\Delta]=0$, which is equivalent to $\sum_{i=0}^n[\Delta_i,\Delta_{n-i}]=0$ for all $n\geq0$; and
		\item $\Kos(\Delta)_{n}(a_1,\ldots,a_n)\equiv 0\pmod{t^{n-1}}$ for all $n\geq2$ and $a_1,\ldots,a_n\in A$. 
	\end{itemize}
\end{definition} 
\begin{remark} It is easy to check that the last requirement is equivalent to  $\Delta_n\in\Diff_{u,\leq n+1}(A)$ for all $n\geq0$. For instance, $\Kos(\Delta)_2(a,b)=\Kos(\Delta_0)_2(a,b)+t\Kos(\Delta_1)_2(a,b)+\cdots\equiv\Kos(\Delta_0)_2(a,b)\pmod{t}$, thus \[\Kos(\Delta)_2(a,b)\equiv0\pmod{t}\qquad\Leftrightarrow\qquad \Kos(\Delta_0)_2(a,b)=0\qquad\Leftrightarrow\qquad\Delta_0\in\Diff_{u,\le1}(A)=\operatorname{Der}(A). \] In general, if we assume that $\Delta_i\in\Diff_{u,\leq(i+1)}(A)$ for all $i<n$, this implies in particular that $\Kos(\Delta_i)_{n+2}=0$ for all $i<n$, and thus that \[\Kos(\Delta)_{n+2}(a_1,\ldots,a_{n+2})\equiv t^{n}\Kos(\Delta_n)_{n+2}(a_1,\ldots,a_{n+2})\pmod{t^{n+1}}.\] for all $a_1,\ldots,a_{n+2}\in A$. Hence, \[\Kos(\Delta)_{n+2}\equiv0\pmod{t^{n+1}}\quad\Leftrightarrow\quad\Kos(\Delta_n)_{n+2}=0\quad\Leftrightarrow\quad\Delta_n\in\Diff_{u,\leq(n+1)}(A). \]
\end{remark}
 
\begin{remark}\label{rem:BVvsPoisson} By definition, given a degree $k$ derived $BV$ algebra $(A,\Delta)$,  we have
	\[ \Kos(\Delta)_n(a_1,\ldots,a_n) \equiv t^{n-1}\Kos(\Delta_{n-1})_{n}(a_1,\ldots,a_n)\pmod{t^n}. \]
	According to Lemma \ref{lem:KosDGLiemor}, this shows that 
	\begin{eqnarray} \nonumber 0 & = & \frac{1}{2}\Kos([\Delta,\Delta])_n(a_1,\ldots,a_n) \\ \nonumber &=& \frac{1}{2}[\Kos(\Delta),\Kos(\Delta)]_n(a_1,\ldots,a_n) \\
	\nonumber &=& \sum_{i=1}^{n}\sum_{\sigma\in S(i,n-i)}\pm_K \Kos(\Delta)_{n-i+1}(\Kos(\Delta)_i(a_{\sigma(1)},\ldots,a_{\sigma(i)}),a_{\sigma(i+1)},\ldots,a_{\sigma(n)})\\
	\nonumber &\equiv& \sum_{i=1}^{n}\sum_{\sigma\in S(i,n-i)}\pm_K t^{n-1} \Kos(\Delta_{n-i})_{n-i+1}(\Kos(\Delta_{i-1})_i(a_{\sigma(1)},\ldots,a_{\sigma(i)}),a_{\sigma(i+1)},\ldots,a_{\sigma(n)})\pmod{t^n},
	\end{eqnarray}
	which is equivalent to say that the degree $1+(n-1)(k-1)$ maps $\Kos(\Delta_{n-1})_n:A^{\odot n}\to A$, $n\geq1$, define a structure of $L_\infty[1]$ algebra on $A[1-k]$. Here we denote by $\pm_K$ the Koszul sign associated to the permutation $a_1,\ldots,a_n\mapsto a_{\sigma(1)},\ldots,a_{\sigma(n)}$: notice that, since $k$ is supposed to be odd, this is the same whether we consider $a_1,\ldots,a_n$ as elements of $A$ or of $A[1-k]$.
	
	Furthermore, since by hypothesis $\Kos(\Delta_{n-1})_{n+1}=0$, the recursive formula \eqref{eq:recursionkoszul} for the Koszul brackets shows that 
	\[ \Kos(\Delta_{n-1})_n(a_1,\ldots,a_{n-1},bc) = \Kos(\Delta_{n-1})_n(a_1,\ldots,a_{n-1},b)c+(-1)^{|b||c|}\Kos(\Delta_{n-1})_n(a_1,\ldots,a_{n-1},c)b \]
	for all $n\geq1,a_1,\ldots,a_{n-1},b,c\in A$, i.e., the map $\Kos(\Delta_{n-1})_{n}$ is a multiderivation for all $n\geq1$. With the terminology of \cite{BCSX}, this shows that the sequence of maps $\Kos(\Delta_{n-1})_n,n\geq1$, defines a structure of degree $k$ derived Poisson algebra on $A$.
	
\end{remark} \newcommand{\LieBV}{\mathcal{BV}_{[k]}(A)}

\begin{remark}\label{rem:LieBV} We shall denote by $\LieBV\subset\End_{u,\mathbb{K}[[t]]}(A[[t]])\cong\End_u(A)[[t]]$ the graded subspace spanned by those $\Delta=\sum_{n\geq0}t^n\Delta_n$ such that $\Delta_n\in\Diff_{u,\leq(n+1)}(A),\:\forall n\geq1$. It follows immediately from the fact that $\big[\Diff_{u,\leq j}(A),\Diff_{u,\leq k}(A)\big]\subset\Diff_{u,\leq (k+j-1)}(A)$ that $\LieBV$ is a graded Lie subalgebra of $\End_{u,\mathbb{K}[[t]]}(A[[t]])$. The degree $k$ derived $BV$ algebra structures on $A$ are precisely the Maurer-Cartan elements of $\LieBV$, that is, the solutions $\Delta\in\LieBV^1$ of the equation \[ \frac{1}{2}[\Delta,\Delta] = 0. \]
\end{remark} 
Given $\Delta=\sum_{n\ge0}t^n\Delta_n\in\LieBV$, we denote by $\mathcal{P}(\Delta)\in\Coder(S(A[1-k]))$ the coderivation given in Taylor coefficients by \[p\mathcal{P}(\Delta)=(0,\Kos(\Delta_0)_1,\ldots,\Kos(\Delta_{n-1})_n,\ldots).\]
As in Remark \ref{rem:BVvsPoisson}, we notice that the Taylor coefficients $\Kos(\Delta_{n-1})_n$ of $\mathcal{P}(\Delta)$ are multiderivations. Moreover, denoting by $\mathcal{P}_{[k]}(A)\subset\operatorname{Coder}(S(A[1-k]))$ the subspace spanned by those coderivations whose Taylor coefficients are multiderivations (and vanishing on $\mathbf{1}_{S(A[1-k])}$), it is easy to check that $\mathcal{P}_{[k]}(A)$ is closed under the commutator brackets, hence a graded Lie subalgebra of $\operatorname{Coder}(S(A[1-k]))$. Finally, the same kind of computations as in Remark \ref{rem:BVvsPoisson} show more in general the following fact. 
\begin{proposition}\label{prop:morphismfromBVtoPoisson}  The correspondence \[ \mathcal{P}\colon\LieBV\to\mathcal{P}_{[k]}(A)\colon\Delta\mapsto \mathcal{P}(\Delta)\]
	is a morphism of graded Lie algebras.
\end{proposition}
We conclude this subsection by recalling the definition of Maurer-Cartan elements of a derived BV-algebra $(A,\Delta)$. For this, we assume that $A$ is equipped with a complete filtration $F^\bullet A$ which is compatible with both the multiplicative structure and the BV operator $\Delta$. We call $(F^\bullet A,\Delta)$ a \emph{complete derived $BV$ algebra}. 
\begin{definition}\label{def:BVMC} Given a complete degree $k$ derived $BV$ algebra $(A,\Delta)$, we say that a degree zero element \[a = \sum_{i\geq-1}t^ia_i = \frac{a_{-1}}{t} + a_0 + ta_1 +\cdots\in\frac{1}{t} A[[t]],\qquad a_i\in A^{i(k-1)},  \]
is a \emph{Maurer-Cartan element} of $A$ (cf. \cite{CL}) if
\[  \Delta(e^a)=0, \]
or equivalently (see \eqref{eq:MCkoszul}) if
\[ \sum_{n\geq1}\frac{1}{n!}\Kos(\Delta)_n(a,\ldots,a)=0. \]
Here we consider the algebras  $\mathbb{K}((t))=\cup_{n\in\mathbb{Z}}\,t^n\mathbb{K}[[t]]$, $A((t))=\cup_{n\in\mathbb{Z}}\,t^nA[[t]]$ of formal Laurent series, and we continue to denote by $\Delta:A((t))\to A((t))$ and $\Kos(\Delta)_n:A((t))^{\odot n}\to A((t))$ the $\K((t))$-linear extensions of $\Delta$ and its Koszul brackets.
\end{definition}
Notice that since $\Kos(\Delta)_n\equiv t^{n-1}\Kos(\Delta_{n-1})_n\pmod{t^n}$ the left hand side of the last identity becomes
\[\frac{1}{t}\sum_{n\geq1} \frac{1}{n!} \Kos(\Delta_{n-1})_n(a_{-1},\ldots,a_{-1})\quad+\quad\mbox{terms in $A[[t]]$}.  \]
In other words, $a_{-1}\in A^{1-k}$ has to be a Maurer-Cartan element of the associated derived Poisson algebra $(A,\mathcal{P}(\Delta))$. 
\subsection{Morphisms of derived $BV$ algebras}\label{sec:BVmor}
\begin{definition}\label{def:morphismsBV} Given a pair of degree $k$ derived $BV$ algebras $(A,\Delta),(B,\Delta')$, a morphism between them is a degree zero map $f\in\Hom^0_{u,\K[[t]]}(A[[t]],B[[t]])$ such that
	\begin{itemize}
	\item  $f\circ\Delta =\Delta'\circ f$; 
	\item $\kappa(f)_n(a_1,\ldots,a_n)\equiv0\pmod{t^{n-1}}$ for all $n\geq2$ and $a_1,\ldots,a_n\in A$.	\end{itemize}
\end{definition}
Our first aim is to provide some justification for the previous definition, which, to the best of our knowledge, hasn't appeared before in the literature\footnote{Actually, after a first draft of this paper was ready, we learned that the same definition has been considered (with a different terminology) by J.-S. Park \cite{ParkQFT}.}. Of course, the first thing we need to show is the following proposition, saying that with the above definition of morphisms, degree $k$ derived $BV$ algebras form indeed a category. 
\begin{proposition}\label{prop:BVcat} Given deree $k$ derived $BV$ algebras $(A,\Delta)$, $(B,\Delta')$ and $(C,\Delta'')$, together with morphisms $f:A[[t]]\to B[[t]]$ and $g:B[[t]]\to C[[t]]$ of derived $BV$ algebras, the composition $gf:A[[t]]\to C[[t]]$ is a morphism of derived $BV$ algebras.
\end{proposition}
\begin{proof} This follows from Proposition \ref{prop:cumulantsandcomposition}, which shows
	\[ \kappa(g f)_n(a_1,\ldots, a_n)= \sum_{\stackrel{k,i_1,\ldots,i_k\geq1}{i_1+\cdots+i_k=n}}\sum_{\sigma\in S(i_1,\ldots,i_k)} \frac{1}{k!}\pm_K\kappa(g)_k\left(\kappa(f)_{i_1}(a_{\sigma(1)},\ldots),\ldots,\kappa(f)_{i_k}(\ldots,a_{\sigma(n)})\right) \]
	for all $a_1,\ldots,a_n\in A$, where as usual we denote by $\pm_K$ the Koszul sign associated with $a_1,\ldots,a_n\mapsto a_{\sigma(1)},\ldots,a_{\sigma(n)}$.
	
	The above formula shows immediately that if \[ \kappa(f)_{i_j}(a'_1,\ldots,a'_{i_j})\equiv 0\pmod{t^{i_j-1}}\qquad\mbox{and}\qquad \kappa(g)_k(b_1,\ldots,b_k)\equiv 0\pmod{t^{k-1}}\] for all $1\leq j\leq k$, $a'_1,\ldots,a'_{i_j}\in A$, $b_1,\ldots,b_k\in B$, then $\kappa(gf)_n(a_1,\ldots,a_n)\cong 0 \pmod{t^{n-1}}$ for all $a_1,\ldots,a_n\in A$. 
\end{proof}
A more convincing justification comes from the following proposition (analogous to \cite[Proposition 2.11]{BCSX}), which shows that (formally) the exponential group of the Lie algebra $\LieBV^0$ from the previous Remark \ref{rem:LieBV} identifies with the group of $\K[[t]]$-module automorphisms of $A[[t]]$ satisfying the second condition from the previous Definition \ref{def:morphismsBV} (and preserving the unit $\1_{A[[t]]}$). In the framework of deformation theory (see \cite{ManDef}), given a derived $BV$ algebra $(A,\Delta)$, this shows that the DG Lie algebra $\big(\LieBV,[\Delta,-],[-,-]\big)$ controls the deformations of $(A,\Delta)$ in the category of derived $BV$ algebras. 

Given $\Delta\in\End^0_{u,\K[[t]]}(A[[t]])$, we consider the associated formal flow $\exp(s\Delta)$: to avoid convergence issues, we consider the parameter $s$ as a central variable of degree $0$, and the formal flow $\exp(s\Delta)\in\Hom_{u,\K[[s,t]]}(A[[s,t]],A[[s,t]])$ as a $\mathbb{K}[[s,t]]$-linear endomorphism of the algebra of formal power series $A[[s,t]]$.
\begin{proposition} Given $\Delta\in\End^0_{u,\K[[t]]}(A[[t]])$, we have 
	\[  \kappa(\exp(s\Delta))_n(a_1,\ldots,a_n)\equiv 0\pmod{t^{n-1}} \]
	for all $n\geq2$ and $a_1,\ldots,a_n\in A$, if and only if	
	\[  \Kos(\Delta)_n(a_1,\ldots,a_n)\equiv 0\pmod{t^{n-1}} \]
	for all $n\geq2$ and $a_1,\ldots,a_n\in A$, that is, if and only if $\Delta\in\LieBV^0$.
\end{proposition}
\begin{proof} To simplify notations, we denote by $\kappa(s):=\kappa(\exp(s\Delta))=\exp(s\Kos(\Delta))$, where the second identity follows from  Proposition \ref{prop:cumulantsvsKoszulexp}. Denoting as usual by $p:S(A[[t]])\to A[[t]]$ the natural projection, we see that  \begin{multline}\label{eq:partials}
	 \partial_s\Big(\kappa(s)_n(a_1,\ldots,a_n)\Big) =  p\Big(\partial_s\kappa(\exp(s\Delta))\Big)(a_1\odot\dots\odot a_n)= \\ = p\Big( \partial_s\exp(s\Kos(\Delta))\Big)(a_1\odot\dots\odot a_n)  =  p\Big(\Kos(\Delta)\circ\kappa(s)\Big)(a_1\odot\dots\odot a_n) = \\ = \sum_{\stackrel{k,i_1,\ldots,i_k\geq1}{i_1+\cdots+i_k=n}}\sum_{\sigma\in S(i_1,\ldots,i_k)}\pm_K\frac{1}{k!}\Kos(\Delta)_k\left(\kappa(s)_{i_1}(a_{\sigma(1)},\ldots),\ldots,\kappa(s)_{i_k}(\ldots,a_{\sigma(n)})\right). \end{multline}
	 First we assume that $\kappa(s)_n\equiv0\pmod{t^{n-1}}$, $\forall\,n\ge2$, and show that $\Kos(\Delta)_n\equiv0\pmod{t^{n-1}}$, $\forall n\ge2$, by induction on $n$. Suppose inductively that we have shown $\Kos(\Delta)_i\equiv 0\pmod{t^{i-1}}$ for all $i<n$. Using the above Equation \ref{eq:partials} together with the inductive hypothesis we see that
	 \[ 0\equiv \partial_s\big(\kappa(s)_n(a_1,\ldots,a_n)\big)\pmod{t^{n-1}}\equiv \Kos(\Delta)_n\big(\kappa(s)_1(a_1),\ldots,\kappa(s)_1(a_n)\big)\pmod{t^{n-1}},  \]
	 for all $a_1,\ldots,a_n\in A$. Since $\kappa(s)_1=\exp(s\Delta)$ is an isomorphism, this shows that $\Kos(\Delta)_n\equiv0\pmod{t^{n-1}}$, and completes the inductive step.
	 
	 Next we assume $\Kos(\Delta)_n\equiv0\pmod{t^{n-1}}$, $\forall\,n\ge2$, and show that $\kappa(s)_n\equiv0\pmod{t^{n-1}}$, $\forall n\ge2$, once again by induction on $n$. Suppose inductively that $\kappa(s)_i\equiv 0\pmod{t^{i-1}}$ for all $i<n$: together with Equation \eqref{eq:partials}, this implies that
	 \[ \partial_s\big(\kappa(s)_n(a_1,\ldots,a_n)\big)\equiv \Kos(\Delta)_1\big(\kappa(s)_n(a_1,\ldots,a_n)\big)\pmod{t^{n-1}}\equiv\Delta\big(\kappa(s)_n(a_1,\ldots,a_n)\big)\pmod{t^{n-1}}. \]
	Thus $(\partial_s)^k\big(\kappa(s)_n(a_1,\ldots,a_n)\big)\equiv \Delta^k\big(\kappa(s)_n(a_1,\ldots,a_n)\big)\pmod{t^{n-1}}$. Expanding $\kappa(s)_n(a_1,\ldots,a_n)$ in formal Taylor series with respect to $s$, and noticing that $\kappa(0)_n=0$ for all $n\geq2$, we conclude that
	\[ \kappa(s)_n(a_1,\ldots,a_n) \equiv \sum_{k\geq0} \frac{s^k}{k!}\Delta^k\left(\kappa(0)_n(a_1,\ldots,a_n)\right)\pmod{t^{n-1}}\equiv0\pmod{t^{n-1}}. \] 
\end{proof}
In the following proposition, we promote the correspondence between derived $BV$ algebras and derived Poisson algebras explained in the previous Remark \ref{rem:BVvsPoisson} to a full-fledged functor (morphisms of derived Poisson algebras were introduced in \cite{BCSX}).
\begin{proposition}\label{prop:BVPois} Given a pair of degree $k$ derived $BV$ algebras $(A,\Delta),(B,\Delta')$ and a morphism $f:A[[t]]\to B[[t]]$ between them, the maps $\mathcal{P}(f)_n:A^{\odot n}\to B$, $n\geq1$, defined by the identities
	\[ \kappa(f)_n(a_1,\ldots,a_n)\equiv t^{n-1}\mathcal{P}(f)_n(a_1,\ldots,a_n)\pmod{t^n}  \]
	for all $a_1,\ldots,a_n\in A$, induce a morphism $\mathcal{P}(f):(S(A[1-k]),\mathcal{P}(\Delta))\to (S(B[1-k]),\mathcal{P}(\Delta'))$ between the associated degree $k$ derived Poisson algebras $A$ and $B$.
\end{proposition}
\begin{proof} Denoting by $p:S(B[[t]])\to B[[t]]$ the natural projection, we have
	\begin{multline*}  p\Big(\kappa(f)\circ\Kos(\Delta)\Big)(a_1,\ldots,a_n)=\\=\sum_{i=1}^n\sum_{\sigma\in S(i,n-i)}\pm_K\kappa(f)_{n-i+1}\Big(\Kos(\Delta)_i(a_{\sigma(1)},\ldots),\ldots,a_{\sigma(n)}\Big)\equiv\\\equiv \sum_{i=1}^n\sum_{\sigma\in S(i,n-i)}\pm_K t^{n-1}\mathcal{P}(f)_{n-i+1}\Big(\mathcal{P}(\Delta)_i(a_{\sigma(1)},\ldots),\ldots,a_{\sigma(n)}\Big)\pmod{t^n},
	\end{multline*}
	where the last congruence follows directly from the definitions of $\mathcal{P}(f)$ and $\mathcal{P}(\Delta)$. Also,
	\begin{multline*}  p\Big(\Kos(\Delta')\circ\kappa(f)\Big)(a_1,\ldots,a_n)=\\=\sum_{\stackrel{k,i_1,\ldots,i_k\ge1}{i_1+\cdots+i_k=n}}\frac{1}{k!}\sum_{\sigma\in S(i_1,,\ldots,i_k)}\pm_K\Kos(\Delta')_{k}\Big(\kappa(f)_{i_1}(a_{\sigma(1)},\ldots),\ldots,\kappa(f)_{i_k}(\ldots,a_{\sigma(n)})\Big)\equiv\\\equiv \sum_{\stackrel{k,i_1,\ldots,i_k\ge1}{i_1+\cdots+i_k=n}}\frac{1}{k!}\sum_{\sigma\in S(i_1,,\ldots,i_k)}\pm_Kt^{n-1}\mathcal{P}(\Delta')_{k}\Big(\mathcal{P}(f)_{i_1}(a_{\sigma(1)},\ldots),\ldots,\mathcal{P}(f)_{i_k}(\ldots,a_{\sigma(n)})\Big)\pmod{t^n},
	\end{multline*}
	
	Since by hypothesis $f\circ\Delta=\Delta'\circ f$, using Proposition \ref{prop:cumulantsvsKoszuldg} we have $\kappa(f)\circ\Kos(\Delta)=\Kos(\Delta')\circ\kappa(f)$. By the above, this shows that
	\begin{multline*}  \sum_{i=1}^n\sum_{\sigma\in S(i,n-i)}\pm_K \mathcal{P}(f)_{n-i+1}\Big(\mathcal{P}(\Delta)_i(a_{\sigma(1)},\ldots),\ldots,a_{\sigma(n)}\Big)=\\= \sum_{\stackrel{k,i_1,\ldots,i_k\ge1}{i_1+\cdots+i_k=n}}\frac{1}{k!}\sum_{\sigma\in S(i_1,,\ldots,i_k)}\pm_K\mathcal{P}(\Delta')_{k}\Big(\mathcal{P}(f)_{i_1}(a_{\sigma(1)},\ldots),\ldots,\mathcal{P}(f)_{i_k}(\ldots,a_{\sigma(n)})\Big)
\end{multline*}
for all $n\ge1$ and $a_1,\ldots,a_n\in A$, i.e., $\mathcal{P}(f):(S(A[1-k]),\mathcal{P}(\Delta))\to (S(B[1-k]),\mathcal{P}(\Delta'))$ is an $L_\infty[1]$ morphism.
	
	
Furthermore, using the recursion \eqref{eq:recursioncumulants} for the cumulants, we have
\begin{eqnarray} \nonumber 0 &\equiv & \kappa(f)_{n+2}(\ldots,b,c)\pmod{t^{n+1}} \\
\nonumber &\equiv& t^{n}\left( \mathcal{P}(f)_{n+1}(\ldots,bc)-\sum_{i=0}^n\sum_{\sigma\in S(i,n-i)}\pm_K\mathcal{P}(f)_{i+1}(\ldots,b)\mathcal{P}(f)_{n-i+1}(\ldots,c)  \right)\pmod{t^{n+1}},
\end{eqnarray}
that is,
\[ \mathcal{P}(f)_{n+1}(a_1,\ldots,a_n,bc)=\sum_{i=0}^n\sum_{\sigma\in S(i,n-i)}\pm_K\mathcal{P}(f)_{i+1}(a_{\sigma(1)},\ldots,a_{\sigma(i)},b)\mathcal{P}(f)_{n-i+1}(a_{\sigma(i+1)},\ldots,a_{\sigma(n)},c)\]
for all $n\geq0$ and $a_1,\ldots,a_n,b,c\in A$. This is precisely the additional requirement the Taylor coefficients $\mathcal{P}(f)_n$ have to satisfy in order for $\mathcal{P}(f)$ to define a morphism of derived Poisson algebras, see  \cite[\S 2.2]{BCSX}.
\end{proof}
\begin{remark} Putting $f=\sum_{n\geq0}t^nf_n$, where $f_n\in\Hom^{n(k-1)}(A,B)$, $f_0(\1_A)=\1_B$, $f_n(\1_A)=0$ for $n\geq1$, we can write down the conditions for $f$ to be a morphism of derived $BV$ algebras more explicitly as follows. On the one hand, the condition $f\circ\Delta = \Delta'\circ f$ translates into \[ \sum_{i=0}^n f_i\circ\Delta_{n-i} = \sum_{i=0}^n \Delta'_i\circ f_{n-i},\qquad\forall n\geq0. \]
On the other hand, the second condition in Definition \ref{def:morphismsBV} translates into a hierarchy of identities, the first three being \begin{eqnarray}\nonumber f_0(ab) &=& f_0(a)f_0(b),\\ \nonumber f_1(abc) &=& f_1(ab)f_0(c) + f_1(ac)f_0(b)+f_1(bc)f_0(a) \\\nonumber &-& f_1(a)f_0(b)f_0(c)-f_1(b)f_0(a)f_0(c)-f_1(c)f_0(a)f_0(b),\\ \nonumber f_2(abcd)  & = &  f_2(abc)f_0(d)+f_2(abd)f_0(c)+f_2(acd)f_0(b)+f_2(bcd)f_0(a)  \\\nonumber &+& f_1(ab)f_1(cd)+f_1(ac)f_1(bd)+f_1(ad)f_1(bc)\\\nonumber &-&f_2(ab)f_0(c)f_0(d)-f_2(ac)f_0(b)f_0(d)-f_2(ad)f_0(b)f_0(c)\\\nonumber &-&f_2(bc)f_0(a)f_0(d)-f_2(bd)f_0(a)f_0(c)-f_2(cd)f_0(a)f_0(b)\\\nonumber &-&f_1(ab)f_1(c)f_0(d)-f_1(ab)f_1(d)f_0(c)-f_1(ac)f_1(b)f_0(d)-f_1(ac)f_1(d)f_0(b)\\\nonumber &-&f_1(ad)f_1(b)f_0(c)-f_1(ad)f_1(c)f_0(b)-f_1(bc)f_1(a)f_0(d)-f_1(bc)f_1(d)f_0(a)\\\nonumber &-&f_1(bd)f_1(a)f_0(c)-f_1(bd)f_1(c)f_0(a)-f_1(cd)f_1(a)f_0(b)-f_1(cd)f_1(b)f_0(a)\\\nonumber &+&f_2(a)f_0(b)f_0(c)f_0(d)+f_2(b)f_0(a)f_0(c)f_0(d)+f_2(c)f_0(b)f_0(a)f_0(d)+f_2(d)f_0(b)f_0(c)f_0(a)\\\nonumber &+&f_1(a)f_1(b)f_0(c)f_0(d)+f_1(a)f_1(c)f_0(b)f_0(d)+f_1(a)f_1(d)f_0(b)f_0(c)\\\nonumber &+&f_1(b)f_1(c)f_0(a)f_0(d)+f_1(b)f_1(d)f_0(a)f_0(c)+f_1(c)f_1(d)f_0(a)f_0(b),\\\nonumber \cdots&&
\end{eqnarray}
for all $a,b,c,d\in A$. Notice that $\kappa(f)_{n+1}\equiv0\pmod{t^{n}}$ automatically implies $\kappa(f)_{N}\equiv0\pmod{t^{n}}$ for all $N>n$ (by the recursion \eqref{eq:recursioncumulants} and a straightforward induction), thus, given the first $(n-1)$ identities, the condition $\kappa(f)_{n+2}\equiv 0\pmod{t^{n+1}}$ translates into a single additional identity on $f_0,\ldots,f_n$. In fact, we are imposing that the coefficient of $t^i$ in the expansion of $\kappa(f)_{n+2}(a_1,\ldots,a_{n+2})\in A[[t]]$ vanishes for all $i\leq n$ and $a_1,\ldots,a_{n+2}\in A$: but for $i<n$ this already follows from the first $(n-1)$ identities, thus we only need to look at the coefficient of $t^n$. Essentially, the $n$-th identity shows how to express $f_n(a_1\cdots a_{n+2})$ as a linear combination of products of the form $f_{i_1}(a_{\sigma(1)}\cdots a_{\sigma(k_1)})\cdots f_{i_j}(a_{\sigma(n-k_j+3)}\cdots a_{\sigma(n+2)})$ with $j\geq2$, $i_1+\cdots+i_j=n$, $i_1\ge\cdots\ge i_j\geq0$, $k_1+\cdots+k_j=n+2$ , $1\leq k_1\leq i_1+1,\ldots, 1\leq k_j\leq i_j+1$, and $\sigma\in S(k_1,\ldots,k_j)$.
\end{remark}

Another important property is that Maurer-Cartan elements can be pushed forward along morphism of derived $BV$ algebras. 

\begin{proposition}\label{prop:MCpushforward} Given a pair of complete derived $BV$ algebras $(A,\Delta), (B,\Delta')$, a continuous morphism $f:A[[t]]\to B[[t]]$ between them and a Maurer-Cartan element $a=\sum_{k\geq-1} t^ka_k\in\frac{1}{t}A[[t]]$ of $A$, then \[ b:=\sum_{n\ge1}\frac{1}{n!}\kappa(f)_n(a,\ldots,a) \]
is a Maurer-Cartan element of $B$ (where we continue to denote by $\kappa(f)_n:A((t))^{\odot n}\to B((t))$ the $\K((t))$-linear extension of the cumulants of $f$). Moreover, $b_{-1}$ is the push-forward of the Maurer-Cartan element $a_{-1}$ along the morphism $\mathcal{P}(f):(A,\mathcal{P}(\Delta))\to(B,\mathcal{P}(\Delta'))$ of derived Poisson algebras from the previous Proposition \ref{prop:BVPois}. 
\end{proposition}

\begin{proof} The Koszul brackets $\Kos(\Delta)_n$ define an $L_\infty[1]$ algebra structures on $A((t))$. Denoting by $\operatorname{MC}_{BV}(A)$ the set of Maurer-Cartan elements of the derived $BV$ algebra $A$, as in Definition \ref{def:BVMC}, and by $\operatorname{MC}_{L_\infty[1]}(A((t)))$ the set of Maurer-Cartan elements of the $L_\infty[1]$ algebra $A((t))$, we have by definition $\operatorname{MC}_{BV}(A) = \operatorname{MC}_{L_\infty[1]}(A((t)))\bigcap \frac{1}{t}A[[t]]$. Similarly for $B$.
	
According to Proposition \ref{prop:cumulantsvsKoszuldg}, the cumulants $\kappa(f)_n:A((t))^{\odot n}\to B((t))$ induce a morphism of $L_\infty[1]$ algebras from $A((t))$ to $B((t))$. In particular, they induces a push-forward 
\[ \operatorname{MC}(f) : \operatorname{MC}_{L_\infty[1]}(A((t)))\to \operatorname{MC}_{L_\infty[1]}(B((t))):x\to\sum_{n\ge1}\frac{1}{n!}\kappa(f)_n(x,\ldots,x) \] 
between the sets of Maurer-Cartan elements.
 	
Finally, given a Maurer-Cartan element $a=\sum_{i\geq-1}t^ia_i \in\operatorname{MC}_{BV}(A)$ of the derived $BV$ algebra $A$, since $\kappa(f)_n\equiv t^{n-1}\mathcal{P}(f)_n\pmod{t^n}$ (using the notations of the previous Proposition \ref{prop:BVPois}), we see that 
	\[\operatorname{MC}(f)(a) = \frac{1}{t}\sum_{n\geq1} \frac{1}{n!} \mathcal{P}(f)_n(a_{-1},\ldots,a_{-1})\quad+\quad\mbox{terms in $B[[t]]$}.  \]
This shows $b:=\operatorname{MC}(f)(a)\in\operatorname{MC}_{L_\infty[1]}(B((t)))\bigcap \frac{1}{t}B[[t]]=\operatorname{MC}_{BV}(B)$, as well as the last claim.
	
\end{proof}

In \cite{CL} the authors introduced and studied morphisms between derived $BV$ algebras, but only when the source derived $BV$ algebra is free as a graded commutative algebra. We close this subsection by comparing the previous Definition \ref{def:morphismsBV} of morphism with the one from \cite{CL}, in the situations when the latter applies. We shall see that the two definitions are equivalent: more precisely, there is a bijective correspondence (given by the exponential and logarithm in the convolution algebra, see below) between morphism of derived $BV$ algebras in our sense and in the sense of \cite{CL}. 

Given a graded vector space $U$ and a unitary graded commutative algebra $B$, the graded cocommutative coalgebra structure on $S(U)$ given by the unshuffle coproduct induces a unitary graded commutative algebra structure on $\Hom(S(U),B[[t]])$, via the convolution product $\star$, cf. Subsection \ref{subsec:Lootransfer}.

Given  derived $BV$ algebra structures $\Delta,\Delta'$ on $S(U)$ and $B$ respectively, then, according to the definition given in \cite{CL}, a map $\varphi\in\Hom^0(S(U),B[[t]])$ such that $\varphi(\1_{S(U)})=0$ defines a morphism of derived $BV$ algebras if the following conditions are satisfied:
\begin{itemize} \item $\exp_\star(\varphi)\circ\Delta=\Delta'\circ\exp_\star(\varphi)$, where we continue to denote by $\exp_\star(\varphi):S(U)[[t]]\to B[[t]]$ the $K[[t]]$-linear extension of $\exp_\star(\varphi)$;
	\item writing $\varphi = \sum_{n\geq0}t^n\varphi_n$, then $\varphi_{n}$ vanishes on $S_{>n+1}(U)=\oplus_{i\geq n+2}U^{\odot n}$ for all $n\geq0$.
\end{itemize}
\begin{proposition}\label{prop:morphismscomparison}. In the previous situation, $\varphi$ defines a morphism of derived $BV$ algebras according to the definition from \cite{CL} if and only if $\exp_\star(\varphi)$ defines a morphism of derived $BV$ algebras according to Definition \ref{def:morphismsBV}.
\end{proposition}  
\begin{proof} As already observed in the proof of Lemma \ref{lem:coalgebramorphisms}, given $f\in\Hom^0_u(S(U),B[[t]])$, we have
	\begin{equation}\label{eq}  
	\log_\star(f)(x_1\odot\ldots\odot x_n) = \kappa(f)_n(x_1,\ldots,x_n)  
	\end{equation}
for all $n\geq1$ and $x_1,\ldots,x_n\in U\subset S(U)$, as can be seen  via a direct computation using the formula \eqref{eq:cumulants} for the cumulants. In other words, denoting by $i:U\to S(U)$ and $p:S(B)\to B$ the natural inclusion and projection respectively (as well as their $\mathbb{K}[[t]]$-linear extension), the following diagram is commutative 
\begin{equation}\label{diagram} \xymatrix{ S(S(U))\ar[r]^-{\kappa(f)} & S(B)[[t]]\ar[d]^{p}  \\ S(U)\ar[u]^{S(i)}\ar[r]_-{\log_\star(f)} & B[[t]] }
\end{equation}

If $\exp_\star(\varphi)$ is a morphism of derived $BV$ algeras in the sense of Definition \ref{def:morphismsBV}, then the above identity  \eqref{eq} shows that
	\[ \varphi(x_1\odot\cdots\odot x_n) = \log_\star(\exp_\star(\varphi))(x_1\odot\cdots\odot x_n) = \kappa(\exp_\star(\varphi))_n(x_1,\ldots,x_n)\equiv0\pmod{t^{n-1}}, \]
	hence that $\varphi_i(x_1\odot\cdots\odot x_n)=0$ for $i<n-1$, as desired.
	
	Conversely, we want to show that if $\varphi$ is a morphism in the sense of \cite{CL}, that is, it satisfies the two conditions stated before the proposition, then $\exp_\star(\varphi)$ is a morphism in our sense. Since the first condition from Definition \ref{def:morphismsBV} is satisfied by hypothesis, we only have to show the second condition, that is,
	\begin{equation}\label{eq:} \kappa(\exp_\star(\varphi))_n(X_1,\ldots,X_n)\equiv 0\pmod{t^{n-1}},\qquad \forall\:n\geq1,X_1\in U^{\odot i_1},\ldots, X_n\in U^{\odot i_n}. 
	\end{equation}
	The recursion \eqref{eq:recursioncumulants} and a straightforward induction show that for any pair of graded commutative algebras $A,B$ and $f\in\Hom^0_u(A,B)$, we have $\kappa(f)_n(a_1,\ldots,a_n)=0$ whenever $a_i=\1_A$ for some $1\leq i\leq n$. In particular, we only have to show \eqref{eq:} for $i_1,\ldots,i_n\geq1$: we shall proceed by induction on $i_1+\cdots+i_n-n$. If this quantity is zero then $i_1=\cdots=i_n=1$, and the same reasoning as in the first part of the proof (using \eqref{eq} and the hypotheses on $\varphi$) shows the desired result. Otherwise, it is not restrictive to assume $i_n>1$ and that there is a non-trivial factorization $X_n=X'_n\odot X''_{n}$ for some $X'_n\in U^{\odot j_1}, X''_{n}\in U^{\odot j_2}, j_1,j_2\geq1, j_1+j_2=i_n$. Using the recursion \eqref{eq:recursioncumulants} for the cumulants,    
	\begin{multline}\label{eqq} \kappa(\exp_\star(\varphi))_n(X_1,\ldots,X_n)=\kappa(\exp_\star(\varphi))_{n+1}(X_1,\ldots,X'_n,X''_{n}) + \\ \sum_{k=0}^{n-1}\sum_{\sigma\in S(k,n-k-1)}\pm_K\kappa(\exp_\star(\varphi))_{k+1}(X_{\sigma(1)},\ldots,X_{\sigma(k)},X'_{n})\kappa(\exp_\star(\varphi))_{n-k}(X_{\sigma(k+1)},\ldots,X_{\sigma(n-1)},X''_{n}). 
		\end{multline}
	Since \begin{multline*} \Big(i_{\sigma(1)}+\cdots+i_{\sigma(k)}+j_1-k-1\Big)+\Big(i_{\sigma(k+1)}+\cdots+i_{\sigma(n-1)}+j_2-n+k\Big)= \\ = i_1+\cdots+i_{n-1}+j_1+j_2-n-1=i_1+\cdots+i_n-n-1<i_1+\cdots+i_n-n,
	\end{multline*}
	 we can apply the inductive hypothesis to the right hand side of the previous identity \eqref{eqq} to deduce that it is congruent to zero modulo $t^{n-1}$, as desired.
\end{proof}

\subsection{Homotopy transfer for derived $BV$ algebras}\label{sec:BVtrans} Our aim in this subsection is to show that derived $BV$ algebra structures can be transferred along semifull DG algebra contractions (Definition \ref{def:semifullalgebracontraction}). An analogous homotopy transfer theorem for derived Poisson algebras was proved in \cite[Theorem 2.18]{BCSX}.

Given a pair of DG commutative algebras $(A,d_A)$ and $(B,d_B)$ and a contraction $(\sigma,\tau,h)$ of $A$ onto $B$, by abuse of notations we shall continue to denote by $(\sigma,\tau,h)$ the induced contraction $(\id_{\K[[t]]}\otimes\sigma,\id_{\K[[t]]}\otimes\tau,\id_{\K[[t]]}\otimes h)$ of $(A[[t]],\id_{\K[[t]]}\otimes d_A)$ onto $(B[[t]],\id_{\K[[t]]}\otimes d_B)$. Given a degree $k$ derived $BV$ algebra structure $\Delta=\sum_{n\geq0}t^n\Delta_n$ on $A$ with $\Delta_0=d_A$, we write $\Delta=d_A+\Delta_+$, where $\Delta_+=\sum_{n\geq1} t^n\Delta_n$. Considering $\Delta_+$ as a perturbation of the differential $d_A$ on $A[[t]]$, we are in the hypotheses of the Standard Perturbation Lemma \ref{lem:SPL}, therefore there is induced a perturbation
\[ \Delta'_+ := \sum_{n\geq0} \sigma\Delta_+(h\Delta_+)^n\tau \]
of the differential $\id_{\K[[t]]}\otimes d_B$ on $B[[t]]$, as well as a perturbed contraction
\[ \breve{\sigma}:= \sum_{n\geq0}\sigma(\Delta_+h)^n\]
\[ \breve{\tau}:= \sum_{n\geq0}(h\Delta_+)^n\tau \]
\[ \breve{h}:= \sum_{n\geq0}(h\Delta_+)^nh \]
of $(A[[t]],\Delta)$ onto $(B[[t]],\Delta':=d_B+\Delta'_+)$. 
\begin{theorem}\label{th:transferBV} In the above hypotheses, if the contraction $(\sigma,\tau,h)$ is a semifull DG algebra contraction, then $\Delta'$ is a degree $k$ derived $BV$ algebra structure on $B$, and $\breve{\tau}$ is a morphism of derived $BV$ algebras from $(B,\Delta')$ to $(A,\Delta)$. Furthermore, the induced degree $k$ derived Poisson algebra structure $\mathcal{P}(\Delta')$ on $B$ and the induced morphism $\mathcal{P}(\breve{\tau}):(B,\mathcal{P}(\Delta'))\to(A,\mathcal{P}(\Delta))$ of derived Poisson algebras coincide with those induced via homotopy transfer from the derived Poisson algebra structure $\mathcal{P}(\Delta)$ on $A$, as in \cite[Theorem 2.18]{BCSX}.
\end{theorem}	

\begin{proof} The relations $(\Delta')^2=0$, $\t\circ\Delta'=\Delta\circ\t$, are satisfied by the standard perturbation Lemma. If $(\sigma,\tau,h)$ is a semifull DG algebra contraction, according to Lemma \ref{lem:semifullstable} the perturbed contraction $(\breve{\sigma},\breve{\tau},\breve{h})$ is a semifull algebra contraction of $(A[[t]],\Delta)$ onto $(B[[t]],\Delta')$ (not a semifull DG algebra contraction, though). The rest of the theorem follows easily from (the proof of) Proposition \ref{prop:transfer}, which shows that the cumulants $\kappa(\t)_n$ and the Koszul brackets $\Kos(\Delta')_n$ obey the recursions
\[ \kappa(\t)_n(x_1,\ldots,x_n) = \sum_{\stackrel{k\geq2,i_1,\ldots,i_k\geq1}{i_1+\cdots+i_k=n}}\sum_{\sigma\in S(i_1,\ldots,i_k)}\pm_K\frac{1}{k!}\h\Kos(\Delta)_k\left(\kappa(\t)_{i_1}(x_{\sigma(1)}\ldots),\ldots,\kappa(\t)_{i_k}(\ldots,x_{\sigma(n)})\right), \]
\[ \Kos(\Delta')_n(x_1,\ldots,x_n) = \sum_{\stackrel{k\geq2,i_1,\ldots,i_k\geq1}{i_1+\cdots+i_k=n}}\sum_{\sigma\in S(i_1,\ldots,i_k)}\pm_K\frac{1}{k!}\s\Kos(\Delta)_k\left(\kappa(\t)_{i_1}(x_{\sigma(1)}\ldots),\ldots,\kappa(\t)_{i_k}(\ldots,x_{\sigma(n)})\right), \]
for all $x_1,\ldots,x_n\in B[[t]]$.

The above formulas and a straightforward induction on $n$, together with the fact that $\Delta$ is a derived $BV$ algebra structure on $A$, show that $\kappa(\t)_n\equiv 0\pmod{t^{n-1}}$, $\Kos(\Delta')_n\equiv 0\pmod{t^{n-1}}$ for all $n\ge2$, that is, $\Delta'$ is a derived $BV$ algebra structure on $B$ and $\t$ is a morphism of derived $BV$ algebras. In fact, if we assume inductively to have shown that $\kappa(\t)_i(x_1,\ldots,x_i)\equiv t^{i-1}\mathcal{P}(\t)_i(x_1,\ldots,x_i)\pmod{t^i}$ (for certain maps $\mathcal{P}(\t)_i:B^{\odot i}\to A$) for all $x_1,\ldots,x_i\in B$ and $i<n$, we see that 
\begin{multline*} \kappa(\t)_n(x_1,\ldots,x_n) \equiv\\ 
\equiv \sum_{\stackrel{k\geq2,i_1,\ldots,i_k\geq1}{i_1+\cdots+i_k=n}}\sum_{\sigma\in S(i_1,\ldots,i_k)}\pm_K\frac{t^{n-1}}{k!}h\mathcal{P}(\Delta)_k\left(\mathcal{P}(\t)_{i_1}(x_{\sigma(1)}\ldots),\ldots,\mathcal{P}(\t)_{i_k}(\ldots,x_{\sigma(n)})\right)\pmod{t^n}
\end{multline*}
\begin{multline*} \Kos(\Delta')_n(x_1,\ldots,x_n) \equiv \\
\equiv \sum_{\stackrel{k\geq2,i_1,\ldots,i_k\geq1}{i_1+\cdots+i_k=n}}\sum_{\sigma\in S(i_1,\ldots,i_k)}\pm_K\frac{t^{n-1}}{k!}\sigma\mathcal{P}(\Delta)_k\left(\mathcal{P}(\t)_{i_1}(x_{\sigma(1)}\ldots),\ldots,\mathcal{P}(\t)_{i_k}(\ldots,x_{\sigma(n)})\right)\pmod{t^n}
\end{multline*}
This proves the inductive step, and thus completes the proof that $\Delta'$ is a derived BV algebra structure and $\t:B[[t]]\to A[[t]]$ a morphism of derived BV algebras. Furthermore, it shows that the Taylor coefficients of $\mathcal{P}(\t)$ and $\mathcal{P}(\Delta')$ obey the recursions 
\[ \mathcal{P}(\t)_n(x_1,\ldots,x_n) = \sum_{\stackrel{k\geq2,i_1,\ldots,i_k\geq1}{i_1+\cdots+i_k=n}}\sum_{\sigma\in S(i_1,\ldots,i_k)}\pm_K\frac{1}{k!}h\mathcal{P}(\Delta)_k\left(\mathcal{P}(\t)_{i_1}(x_{\sigma(1)}\ldots),\ldots,\mathcal{P}(\t)_{i_k}(\ldots,x_{\sigma(n)})\right), \]
\[ \mathcal{P}(\Delta')_n(x_1,\ldots,x_n) = \sum_{\stackrel{k\geq2,i_1,\ldots,i_k\geq1}{i_1+\cdots+i_k=n}}\sum_{\sigma\in S(i_1,\ldots,i_k)}\pm_K\frac{1}{k!}\sigma\mathcal{P}(\Delta)_k\left(\mathcal{P}(\t)_{i_1}(x_{\sigma(1)}\ldots),\ldots,\mathcal{P}(\t)_{i_k}(\ldots,x_{\sigma(n)})\right). \]
In other words, the induced $L_\infty[1]$ algebra structure $\mathcal{P}(\Delta')$ on $B[1-k]$ and the induced $L_\infty[1]$ morphism $\mathcal{P}(\t):(S(B[1-k]),\mathcal{P}(\Delta'))\to(S(A[1-k]),\mathcal{P}(\Delta))$ coincide with those obtained by transferring the $L_\infty[1]$ algebra structure $\mathcal{P}(\Delta)$ on $A[1-k]$ along the contraction $(\sigma,\tau,h)$, as in Theorem \ref{th:transfer}, which is precisely what we needed to show in order to prove the last claim of the theorem (see \cite{BCSX}).
\end{proof}

In the same hypotheses as in the previous Theorem \ref{th:transferBV}, we assume moreover that $A$ is equipped with a complete filtration making it into a complete derived BV algebra, and that $h,\tau\sigma$ are continuous with respect to this filtration. Then it is easy to check that $(B,\Delta')$ is a complete derived BV algebra with respect to the induced filtration (i.e., the only one making $\sigma$ and $\tau$ continuous) and that $\t:B[[t]]\to A[[t]]$ is continuous. In this situation we might consider the Maurer-Cartan sets $\operatorname{MC}_{BV}(A):=\operatorname{MC}_{L_\infty[1]}(A((t)))\cap \frac{1}{t}A[[t]]$ and $\operatorname{MC}_{BV}(B):=\operatorname{MC}_{L_\infty[1]}(B((t)))\cap \frac{1}{t}B[[t]]$ (cf. the proof of Proposition \ref{prop:MCpushforward} for the notations). Since $(\breve{\sigma},\breve{\tau},\breve{h})$ is a semifull algebra contraction of $(A((t)),\Delta)$ onto $(B((t)),\Delta')$, according to Proposition \ref{prop:transfer} the Koszul brackets $\Kos(\Delta')_n:B((t))^{\odot n}\to B((t))$ and the cumulants $\kappa(\t)_n:B((t))^{\odot n}\to A((t))$ are induced via homotopy transfer from the Koszul brackets $\Kos(\Delta)_n:A((t))^{\odot n}\to A((t))$ along the contraction $(\s,\t,\h)$, which was the key point in the proof of the previous theorem. Applying the formal Kuranishi Theorem \ref{th:kuranishi}, together with Proposition \ref{prop:MCpushforward}, we obtain the following result.
\begin{theorem}\label{th:BVKur} In the above hypotheses, there is a bijective correspondence
	\[ \operatorname{MC}_{\operatorname{BV}}(B) \cong \operatorname{MC}_{\operatorname{BV}}(A)\cap\operatorname{Ker}(\breve{h}). \] 
In one direction it's given by the push-forward $\operatorname{MC}(\t):\operatorname{MC}_{BV}(B)\to\operatorname{MC}_{BV}(A)$ along the continuous morphism of complete derived $BV$ algebras $\breve{\tau}:(B,\Delta')\to(A,\Delta)$; in the other direction, it sends $a\in\operatorname{MC}_{BV}(A)\cap\operatorname{Ker}(\h)$ to the Maurer-Cartan element $\s(a)\in\operatorname{MC}_{BV}(B)$.
\end{theorem}

\subsection{Derived $BV$ coalgebras}\label{sec:BVco}

\begin{definition}\label{def:coBV} A \emph{degree $k$ derived $BV$ coalgebra} is a (cocommutative, et cet.) graded coalgebra $C$ together with a degree one $\K[[t]]$-linear (where $t$ is again a central variable of degree $1-k$) endomorphism $\delta\in\End^1_{cu,\mathbb{K}[[t]]}(C[[t]])$ such that $\delta^2=0$ and
	\[ \Kos^{co}(\delta)_n\equiv0\pmod{t^{n-1}}\qquad\mbox{for all $n\geq2$}, \]
where the last condition is equivalent to $\delta=\sum_{n\geq0}t^n\delta_n$ and $\delta_n\in\operatorname{coDiff}_{cu,\leq(n+1)}(C)$ for all $n\geq0$ (see Definition \ref{def:codiffop}).
	
Given a pair of degree $k$ derived $BV$ coalgebras $(C,\delta)$ and $(D,\delta')$, a morphism between them is a $\K[[t]]$-linear morphism $f\in\Hom^0_{cu,\mathbb{K}[[t]]}(C[[t]],D[[t]])$ such that $f\circ\delta=\delta'\circ f$ and
	\[ \kappa^{co}(f)_n\equiv0\pmod{t^{n-1}}\qquad\mbox{for all $n\geq2$}. \]
\end{definition}

As in Proposition \ref{prop:BVcat}, with the above definitions degree $k$ derived $BV$ coalgebras form a category. Likewise, the various results from the previous subsections admit dual versions in the context of derived $BV$ coalgebras.

We are particularly interested in the Homotopy Transfer Theorem \ref{th:transferBV} and Proposition \ref{prop:morphismscomparison}. For the former, given a derived $BV$ coalgebra $(C,\delta)$ together with a semifull DG coalgebra contraction $(\sigma,\tau,h)$ of $(C,d_C:=\delta_0)$ onto a DG coalgebra $(D,d_D)$, the perturbation $\delta_+:= \sum_{n\geq1}t^n\delta_n$ of $\id_{\K[[t]]}\otimes d_C$ induces (via the standard Perturbation Lemma \ref{lem:SPL}) a perturbed differential $\delta'$ on $D[[t]]$, together with a perturbed contraction $(\s,\t,\h)$ of $(C[[t]],\delta)$ onto $(D[[t]],\delta')$.
\begin{theorem}\label{th:transfercoBV} In the above hypotheses, $\delta'$ is a derived $BV$ coalgebra structure on $D$, and $\s$ is a morphism of derived $BV$ coalgebras.
\end{theorem}
\begin{proof} This is shown as in the proof of Theorem \ref{th:transferBV}, using Proposition \ref{prop:cotransfer} in place of Proposition \ref{prop:transfer}.
\end{proof}

Let $(C,\delta)$ be a derived $BV$ coalgebra, and $(S(V),\delta')$ a second derived $BV$ coalgebra whose underlying graded coalgebra is cofree. The graded cocommutative coalgebra structure on $C$ and the symmetric algebra structure on $S(V)[[t]]$ induce a graded commutative algebra structure on $\Hom(C,S(V)[[t]])\cong\Hom(C,S(V))[[t]]$ via the convolution product $\star$. We denote the unit in this algebra by $\varepsilon:C\to S(V)[[t]]$: it is the map sending $\mathbf{1}_C$ to $\mathbf{1}_{S(V)}$ and $c\in\overline{C}$ to zero. As in subsection \ref{subsec:Lootransfer}, given $f\in\Hom_{cu}(C,S(V)[[t]])$ we denote by $\log_\star(f)=\sum_{n\ge1}\frac{(-1)^{n-1}}{n}(f-\varepsilon)^{\star n}$. In the following lemma we give necessary and sufficient conditions for  $f:C[[t]]\to S(V)[[t]]$ to be a morphism of derived BV coalgebras. More precisely, we show that this is equivalent to $f\circ\delta=\delta'\circ f$ and item (b) in the claim of the following lemma: this should be compared with Proposition \ref{prop:morphismscomparison}. When the source DG coalgebra $C$ is also cofree, we give further equivalent conditions (items (c) and (d) in the claim of the following lemma), this time in terms of the cumulants $\kappa(f)_k$, $k\ge1$. This should be compared with Lemma \ref{lem:coalgebramorphisms}, and will play a similar role in the proof of the homotopy transfer Theorem \ref{th:transferIBL} for $IBL_\infty$ algebras as the latter did in the proof of Theorem \ref{th:transferL}.

\begin{lemma}\label{lem:comorphisms} Given derived $BV$ coalgebras $(C,\delta)$, $(S(V),\delta')$, together with a degree zero map  $f=\sum_{n\ge0}t^nf_n\in\Hom^0(C,S(V)[[t]])$ such that $f(\mathbf{1}_C)=\mathbf{1}_{S(V)}$, if furthermore (the $\K[[t]]$-linear extension of) $f$ satisfies $f\circ\delta = \delta'\circ f$, then the following conditions are equivalent (where we denote by $\varphi_n:C\to S(V)$, $n\ge0$, the maps  defined by $\log_\star(f)=:\varphi= \sum_{n\geq0}t^n\varphi_n$):\begin{enumerate} \item[(a)] $f$ is a morphism of derived BV-coalgebras;
\item[(b)]	$\operatorname{Im}(\varphi_i)\subset S_{\leq (i+1)}(V):=\oplus_{j=0}^{i+1}V^{\odot j}$, $\forall \,i\geq0$.\end{enumerate}
If moreover $C=S(U)$ is also cofree, these are further equivalent to any of the following:
\begin{enumerate} \item[(c)]\label{itema1}   $\kappa(f)_k(U,\ldots,U)\subset\oplus_{n\geq0}t^n S_{\leq(n+1)}(V)$ for all $k\geq1$;
	\item[(d)]\label{itemb1} $\kappa(f)_k(S_{\leq i_1}(U),\ldots,S_{\leq i_k}(U))\subset\oplus_{n\geq0}t^n S_{\leq(i_1+\cdots+i_k-k+n+1)}(V)$ for all $k,i_1,\ldots,i_k\geq1$.
\end{enumerate} \end{lemma}
\begin{proof} We denote by $i:C\to S(C)$ and $p:S(V)\to V$ the natural inclusion and projection (as well as their $\mathbb{K}[[t]]$-linear extensions). We notice that $\varphi:=\log_\star(f)$ fits into the following commutative diagram
	\begin{equation}\label{diagram2} \xymatrix{ S(C)\ar[r]^-{\kappa^{co}(f)} & S(S(V))[[t]]\ar[d]^{S(p)} \\ C\ar[u]^{i}\ar[r]_-{\log_\star(f)} & S(V)[[t]] }
	\end{equation} 
which is dual to the one \eqref{diagram}, and can be similarly shown by a direct computation. In particular, 
\[ \varphi(c)=\log_\star(f)(c) = S(p)\Big(\kappa^{co}(f)(c) \Big).  \]
If $f$ is a morphism of derived $BV$ coalgebras, then the above identity shows that the composition of $\varphi$ with the natural projection $S(V)[[t]]\twoheadrightarrow V^{\odot n}[[t]]$ vanishes modulo $t^{n-1}$ for all $n\geq2$, or in other words that the composition $C\xrightarrow{\varphi_i}S(V)\twoheadrightarrow V^{\odot n}$ vanishes whenever $i<n-1$, which shows that (a) implies (b). 

The converse statement can be shown by dualizing the inductive argument in the proof of Proposition \ref{prop:morphismscomparison}, using the recursion for the cocumulants explained in Remark \ref{rem:recursioncocumulants}. More precisely, denoting by $p_i:S(V)\to V^{\odot i}$ the natural projection, then by definition $f$ is a morphism of derived BV coalgebras if and only if $\kappa^{co}(f)_{n}\equiv0\pmod{t^{n-1}}$ for all $n\ge2$ if and only if \begin{equation}\label{qe} (p_{i_1}\odot\cdots\odot p_{i_n})\kappa^{co}(f)\equiv0\pmod{t^{n-1}}
\end{equation} for all $n\ge2$ and $i_1,\ldots,i_n\ge1$. If (a) holds, then by the above commutative diagram \eqref{diagram2} equation \eqref{qe} is true for $i_1=\cdots=i_n=1$. To show that is true in general, one uses induction on $i_1+\cdots+i_n-n$ and the recursion for the cocumulants from Remark \ref{rem:recursioncocumulants}. We leave details to the reader, cf. also the proof of the following Lemma \ref{lem:IBLtransKos}.

This concludes the proof that (a) is equivalent to (b). 

Next we assume that $C=S(U)$ is cofree.

The equivalence between (c) and (d) can be seen by induction on $i_1+\cdots+i_k-k$, using the recursion \eqref{eq:recursioncumulants} for the cumulants in a similar way as what we did in order to conclude the proof of Proposition \ref{prop:morphismscomparison}. 

To show that (c) is equivalent to (b), we put together the previous commutative diagrams \eqref{diagram} and \eqref{diagram2} to conclude that the following one is also commutative
\begin{equation}\label{diagram3} \xymatrix{ S(S(U))\ar[r]^-{\kappa(f)} & S(S(V))[[t]]\ar[d]^{p} \\ S(U)\ar[u]^{S(i)}\ar[d]_{i}\ar[r]^-{\log_\star(f)} & S(V)[[t]] \\ S(S(U))\ar[r]^-{\kappa^{co}(f)}& S(S(V))[[t]]\ar[u]_{S(p)}}
\end{equation}

In particular
\[ \kappa(f)_k(x_1,\ldots,x_k) = \varphi(x_1\odot\cdots\odot x_k)=\sum_{n\ge0}t^n\varphi_n(x_1\odot\cdots\odot x_k),\]
for all $k\ge1$ and $x_1,\ldots,x_k\in U$. Thus item (c) is equivalent to $\operatorname{Im}(\varphi_n)\subset S_{\le(n+1)}(V)$ for all $n\ge0$, which is precisely item (b).
\end{proof}


\section{Homotopy transfer for $IBL_\infty$ algebras}

\emph{$IBL_\infty$ algebras} (short for Involutive Lie Bialgebras up to coherent homotopies) were introduced in the paper \cite{Fuk}, with applications to string topology, symplectic field theory and Lagrangian Floer theory, and have been further investigated in several other papers since then, for instance \cite{Campos,MarklVor,DJP,Haj1,Merk,HLV,NWill,CV,Haj}.

\subsection{$IBL_\infty[1]$ algebras}\label{sec:IBL}

We shall work with the following definition of $IBL_\infty[1]$ algebra, which is slightly different from (and in a certain sense dual to) the one usually appearing in the literature (see Remark \ref{rem:IBL} for a comparison).

\begin{definition}\label{def:IBL} An \emph{$IBL_\infty[1]$ algebra} structure on a graded space $U$ is a degree $(-1)$ derived $BV$ coalgebra structure $\delta$ on the symmetric coalgebra $S(U)$ .
	
Given a pair of $IBL_\infty[1]$-algebras $(U,\delta)$ and $(V,\delta')$, an $IBL_\infty[1]$ morphism between them is a morphism of derived $BV$ coalgebras $f:S(U)[[t]]\to S(V)[[t]]$.

\end{definition}
\newcommand{\fp}{\mathfrak{p}}

In the following lemma, given a graded space $U$ together with a $\K[[t]]$-linear differential $\delta=\sum_{n\ge0}t^n\delta_n$ on $S(U)[[t]]$, we give necessary and sufficient conditions in order for $\delta$ to define an $IBL_\infty[1]$ algebra structure on $U$ in terms of the associate Koszul brackets $\Kos(\delta_n)_k$, $n\ge0,\,k\ge1$. This result in analogous to Lemma \ref{lem:comorphisms} from the previous section, and should also be compared with Lemma \ref{lem:transkos}: as in the proof of the latter, we will denote by \[s:S(U)\to S(U):x_1\odot\cdots\odot x_n\to (-1)^nx_1\odot\cdots\odot x_n\] 
the antipode in the Hopf algebra $S(U)$ (as well as its $\mathbb{K}[[t]]$-linear extension).

\begin{lemma}\label{lem:IBLtransKos} Given a graded space $U$ and $\delta=\sum_{n\ge0}t^n\delta_n\in\End^1_{u,\K[[t]]}(S(U)[[t]])$ such that $\delta^2=0$, then the following conditions are equivalent (where we denote by $\delta\star s=:\varphi=\sum_{n\ge0}t^n\varphi_n$, and the convolution product is computed in $\Hom(S(U),S(U)[[t]])$):\begin{enumerate} \item[(a)] $\delta$ defines an $IBL_\infty[1]$ algebra structure on $U$; \item[(b)] $\operatorname{Im}(\varphi_n)\subset S_{\le(n+1)}(U)$, $\forall\, i\ge0$; \item[(c)]  $\Kos(\delta_n)_k(U,\ldots,U)\subset  S_{\leq(n+1)}(U)$ for all $n\ge0,\,k\geq1$;
		\item[(d)] $\Kos(\delta_n)_k(S_{\leq i_1}(U),\ldots,S_{\leq i_k}(U))\subset  S_{\leq(i_1+\cdots+i_k-k+n+1)}(U)$ for all $n\ge0,\,k,i_1,\ldots,i_k\geq1$.
	\end{enumerate}
\end{lemma}

\begin{proof} $\delta=\sum_{n\ge0}t^n\delta_n$ defines a derived BV coalgebra structure on $S(U)$ if and only if
	\[ \Kos^{co}(\delta)_k=\sum_{n\ge0}t^n\Kos^{co}(\delta_n)_k\equiv0\pmod{t^{k-1}}\]
	for all $n\ge0$, $k\ge2$, that is, if and only if $\Kos^{co}(\delta_n)_k=0$ whenever $n<k-1$. As in the claim of the lemma, we put $\delta\star s=:\varphi=\sum_{n\ge0}t^n\varphi_n$: then $\varphi_n=\delta_n\star s$. We have the following commutative diagram (which should be compared with the previous one \eqref{diagram3})
	\begin{equation}\label{diagram4} \xymatrix{ S(S(U))\ar[r]^-{\Kos(\delta_n)} & S(S(U))\ar[d]^{p} \\ S(U)\ar[u]^{S(i)}\ar[d]_{i}\ar[r]^-{\delta_n\star s} & S(U) \\ S(S(U))\ar[r]^-{\Kos^{co}(\delta_n)}& S(S(U))\ar[u]_{S(p)}}
	\end{equation}
	where the commutativity of the top square was already observed in the proof of Lemma \ref{lem:transkos}, and the commutativity of the bottom square can be similarly shown by a direct computation. Using this, we can proceed as in the proof of Lemma \ref{lem:comorphisms}. 
	
	If $\delta$ is a derived BV-coalgebra structure, i.e., if $\Kos^{co}(\delta_n)_k=0$ for all $0\leq n \le k-2$, then the above shows that the composition of $\varphi_n=\delta_n\star s$ and the projection $S(V)\to V^{\odot k}$ vanishes whenever $k>n+1$: thus $\operatorname{Im}(\varphi_n)\subset S_{\le(n+1)}(U)$, showing that (a) implies (b). 
	
	Conversely, assume that (b) holds: we need to show that $\Kos^{co}(\delta_n)_k=0$ if $k>n+1$. We shall use the recursion for the Koszul cobrackets explained in Remark \ref{rem:recursioncoKos}. Thus, considering the maps $\widetilde{\Kos}^{co}(\delta_n)_k:S(U)\to S(U)^{\otimes k}$ introduced there, and denoting by $p_i:S(U)\to U^{\odot i}$ the natural projection, we need to show that
	\begin{equation}\label{qee} (p_{i_1}\otimes\cdots\otimes p_{i_k})\widetilde{\Kos}^{co}(\delta_n)_k =0
	\end{equation}
for all $k>n+1$ and $i_1,\ldots,i_k\ge1$. We shall proceed by induction on $i_1+\cdots+i_k-k$. If this quantity is zero, then $i_1=\cdots=i_k=1$ and the thesis follows from (b) and the above commutative diagram \eqref{diagram4}. Otherwise, it is not restrictive to assume $i_k>1$: this follows from the fact, already observed in Remark \ref{rem:recursioncoKos}, that the image of $\widetilde{\Kos}^{co}(\delta_n)_k$ is contained in the $S_k$-invariant part of $S(U)^{\otimes k}$. When $i_k>1$, Equation \eqref{qee} is equivalent to 
  \[ 0=(id^{\otimes k-1}\otimes \overline{\Delta})(p_{i_1}\otimes\cdots\otimes p_{i_k})\widetilde{\Kos}^{co}(\delta_n)_k =\sum_{\stackrel{j_1,j_2\ge1}{j_1+j_2=i_k}}(p_{i_1}\otimes\cdots\otimes p_{i_{k-1}}\otimes p_{j_1}\otimes p_{j_2})(id^{\otimes k-1}\otimes \overline{\Delta})\widetilde{\Kos}^{co}(\delta_n)_k, \]
where we denote by $\overline{\Delta}$ the reduced coproduct on $S(U)$. Using the recursion from Remark \ref{rem:recursioncoKos} (cf. also Remark \ref{rem:conilpotencycoKos} and the discussion at the end of Remark \ref{rem:conilpotencycocumulants}), we see that 
\[ (id^{\otimes k-1}\otimes \overline{\Delta})\widetilde{\Kos}^{co}(\delta_n)_k  = \widetilde{\Kos}^{co}(\delta_n)_{k+1}+(\id^{\otimes k+1}+\tau_{k,k+1})\big(\widetilde{\Kos}^{co}(\delta_n)_{k}\otimes\id\big)\overline{\Delta}.\]
Finally, using the inductive hypothesis we have $(p_{i_1}\otimes\cdots\otimes p_{i_{k-1}}\otimes p_{j_1}\otimes p_{j_2})\widetilde{\Kos}^{co}(\delta_n)_{k+1}=0$, as well as
\begin{multline*}  (p_{i_1}\otimes\cdots\otimes p_{i_{k-1}}\otimes p_{j_1}\otimes p_{j_2})(\id^{\otimes k+1}+\tau_{k,k+1})\big(\widetilde{\Kos}^{co}(\delta_n)_{k}\otimes\id\big)= \\ =  (p_{i_1}\otimes\cdots\otimes p_{i_{k-1}}\otimes p_{j_1})\widetilde{\Kos}^{co}(\delta_n)_{k}\otimes p_{j_2}+\tau_{k,k+1}\Big( (p_{i_1}\otimes\cdots\otimes p_{i_{k-1}}\otimes p_{j_2})\widetilde{\Kos}^{co}(\delta_n)_{k}\otimes p_{j_1}\Big) = 0.
\end{multline*}
This concludes the proof of the equivalence between (a) and (b).
	
The equivalence netween (b) and (c) follows from the top square in the above commutative diagram \eqref{diagram4}, which shows that
\begin{equation*} \Kos(\delta_n)_k(x_1,\ldots,x_k) = \varphi_n(x_1\odot\cdots\odot x_k),\qquad\forall\, k\ge1,\,x_1,\ldots,x_k\in U.
\end{equation*}	
	
	Finally, the equivalence between (c) and (d) can be seen by induction on $i_1+\cdots+i_k-k$. When $i_1+\cdots+i_k-k=0$ then $i_1=\cdots=i_k=1$, and (d) reduces to (c). Otherwise, given $X_1\in U^{\odot i_i}, \ldots, X_k\in U^{\odot i_k}$, it is not restrictive to assume $i_k>1$ and $X_k=X'_k\odot X''_k$ for some $X'_k\in U^{\odot j_1}, X''_k\in U^{\odot j_2}$, $j_1,j_2\ge1$, $j_1+j_2=i_k$. Using the recursion \eqref{eq:recursionkoszul} for the Koszul brackets, we see that
	\begin{multline*} \Kos(\delta_n)_k(X_1,\ldots,X_k) = \Kos(\delta_n)_{k+1}(X_1,\ldots,X_{k-1},X'_k,X''_k)+\\+\Kos(\delta_n)_{k}(X_1,\ldots,X_{k-1},X'_k)\odot X''_k\pm_K\Kos(\delta_n)_{k}(X_1,\ldots,X_{k-1},X''_k)\odot X'_k,
	\end{multline*}
	and the right hand side of the above identity belongs to $S_{\le i_1+\cdots+i_k-k+n+1}(U)$ by the inductive hypothesis.
	\end{proof}

\begin{remark}\label{rem:SIBL} Given a graded space $U$, we denote by $IBL(U)\subset S(U)[[t]]$ the subspace \[IBL(U):= \oplus_{n\ge0}t^nS_{\le(n+1)}(U).\]
We observe that if $\delta=\sum_{n\ge0}t^n\delta_n:S(U)[[t]]\to S(U)[[t]]$ is an $IBL_\infty[1]$ algebra structure on $U$, then the Koszul brackets $\Kos(\delta)_k$ preserve the subspace $IBL(U)$, i.e., they restrict to maps $\Kos(\delta)_k:IBL(U)^{\odot k}\to IBL(U)$ inducing an $L_\infty[1]$ algebra structure on $IBL(U)$. In fact, given $x^1,\ldots,x^k\in IBL(U)$, for $j=1,\ldots,k$ we write $x^j=\sum_{i\ge0}t^ix^j_{i}$, $x^j_{i}\in S_{\le(i+1)}(U)$. Then 
\[\Kos(\delta)_k(x_1,\ldots,x_k)=\sum_{n,i_1,\ldots,i_k\ge0} t^{n+i_1+\cdots+i_k} \Kos(\delta_n)_k(x^1_{i_1},\ldots,x^k_{i_k} )\in IBL(U)\]
since by item (d) in the previous lemma we have $\Kos(\delta_n)_k(x^1_{i_1},\ldots,x^k_{i_k} )\in S_{\le n+i_1+\cdots+i_k+1}(U)$ for all $n,i_1,\ldots,i_k\ge0$. 

If $(U,\delta)$, $(V,\delta')$ are $IBL_\infty[1]$ algebras and $f:S(U)[[t]]\to S(V)[[t]]$ is an $IBL_\infty[1]$ morphism, the cumulants of $f$ restrict to maps $\kappa(f)_n:IBL(U)^{\odot n}\to IBL(V)$, which are the Taylor coefficients of an $L_\infty[1]$ morphism $\kappa(f):(IBL(U),\Kos(\delta))\to(IBL(V),\Kos(\delta'))$. In fact, with the same notations as before, 
\[\kappa(f)_k(x_1,\ldots,x_k)=\sum_{i_1,\ldots,i_k\ge0} t^{i_1+\cdots+i_k} \kappa(f)_k(x^1_{i_1},\ldots,x^k_{i_k} )\in IBL(V)\]
since by item (d) in Lemma \ref{lem:comorphisms} we have $\kappa(f)_k(x^1_{i_1},\ldots,x^k_{i_k} )\in \oplus_{n\ge0} \,t^nS_{\le n+i_1+\cdots+i_k+1}(V)$ for all $i_1,\ldots,i_k\ge0$.

\end{remark}

\begin{remark}\label{rem:IBL} Given an $IBL_\infty[1]$ algebra structure $\delta=\sum_{n\ge0}t^n\delta_n:S(U)[[t]]\to S(U)[[t]]$, according to Definition \ref{def:coBV} we have $\delta(\mathbf{1}_{S(U)})=0$. Moreover, according to Lemma \ref{lem:IBLtransKos} we have $\Kos(\delta)_i(x_1,\ldots,x_i)\in\oplus_{n\ge0}t^nS_{n+1}(U)$ for all $i\ge1$ and $x_1,\ldots,x_i\in U$. Given $i\ge1$ and $n\ge0$, we denote by $\fp_{i,n}:U^{\odot i}\to S_{\le(n+1)}(U)$ the maps defined by the identity
	\[ \Kos(\delta)_i(x_1,\ldots,x_i) = \sum_{n\ge0} t^{n}\fp_{i,n}(x_1,\ldots,x_i). \]
Moreover, given $i,j\ge1$ and $g\ge0$, we denote by $\fp_{i,j,g}:U^{\odot i}\to U^{\odot j}$ the composition of $\fp_{i,j+g-1}:U^{\odot i}\to S_{\le j+g}(U)$ and the projection $S_{\le j+g}(U)\to U^{\odot j}$. Thus, 
	\[ \Kos(\delta)_i(x_1,\ldots,x_i) = \sum_{j\ge1,\,g\ge0} t^{j+g-1}\fp_{i,j,g}(x_1,\ldots,x_i). \]
According to \eqref{eq:Q}
\begin{multline*} \delta(x_1\odot\cdots\odot x_k) = \sum_{i=1}^k\sum_{\sigma\in S(i,k-i)}\pm_K\Kos(\delta)_i\left(x_{\sigma(1)},\ldots,x_{\sigma(i)}\right)\odot\cdots\odot x_{\sigma(k)} =\\= \sum_{i=1}^k\sum_{\sigma\in S(i,k-i)}\sum_{j\ge1,\,g\ge0}\pm_K\, t^{j+g-1}\fp_{i,j,g}\left(x_{\sigma(1)},\ldots,x_{\sigma(i)}\right)\odot\cdots\odot x_{\sigma(k)} 
\end{multline*}
In other words, $\delta = \sum_{i,j\ge 1,\,g\ge0}t^{j+g-1}\widehat{\fp}_{i,j,g}$, where we denote by $\widehat{\fp}_{i,j,g}:S(U)\to S(U)$ the maps defined by $\widehat{\fp}_{i,j,g}(x_1\odot\cdots\odot x_k)=0$ if $k<i$ and 
\[\widehat{\fp}_{i,j,g}(x_1\odot\cdots\odot x_k):=\sum_{\sigma\in S(i,k-i)}\pm_K\, \fp_{i,j,g}\left(x_{\sigma(1)},\ldots,x_{\sigma(i)}\right)\odot\cdots\odot x_{\sigma(k)}. \]
if $k\ge i$. Finally, the condition $\delta^2=0$ translates into a hierarchy of identities on the maps $\fp_{i,j,g}$. More precisely, arguing as in \cite[Lemma 2.6]{Fuk} or \cite[Proposition D.2.6]{Haj}, we see that $\delta^2=0$ if and only if for all $i,j\ge1,g\ge0$ we have  
\begin{equation}\label{eq:IBL} \sum_{h=1}^{g+1}\sum_{i_1+i_2=i+h}\sum_{j_1+j_2=j+h}\sum_{g_1+g_2=g+1-h}\fp_{i_1,j_1,g_1}\circ_h\fp_{i_2,j_2,g_2}=0,
\end{equation}
where we denote by $\fp_{i_1,j_1,g_1}\circ_h\fp_{i_2,j_2,g_2}$ ``the composition of $\fp_{i_1,j_1,g_1}$ and $\fp_{i_2,j_2,g_2}$ along $h$ matching inputs/outputs'' (we refer to \cite{Fuk} or \cite{Haj} for a more precise definition). This shows that our definition of $IBL_\infty[1]$ algebra is essentially equivalent to the one in \cite{Fuk}, if not for some minor differences. Whereas our maps $\fp_{i,j,g}$ have degree $|\fp_{i,j,g}|=1-2(j+g-1)$, the ones in the definition from \cite{Fuk} have degree $2(i+g-1)-1$: thus, when $U$ is finite dimensional, an $IBL_\infty[1]$ algebra structure in our sense corresponds to an $IBL_\infty[1]$ algebra structure in the sense of \cite{Fuk} on the dual space $U^\vee$ (this might also be compared with the discussion in \cite[Remark 11]{MarklVor}). For our purposes, the previous Definition \ref{def:IBL} has some technical advantages.  
\end{remark}

\subsection{Homotopy transfer for $IBL_\infty[1]$ algebras}\label{sec:IBLtrans}
Consider an $IBL_\infty[1]$ algebra $(U,\delta)$ with linear part $d_U:=\fp_{1,1,0}:U\to U$: in particular, by the previous identity \eqref{eq:IBL} for $(i,j,g)=(1,1,0)$, we see that $d_U^2=0$, hence $d_U$ is a differential on $U$. Given a contraction $f:V\to U$, $g:U\to V$, $h:U\to U[-1]$ of $(U,d_U)$ onto a complex $(V,d_V)$, we want to show how to transfer the $IBL_\infty[1]$ algebra structure on $U$ along this contraction via the symmetrized tensor trick and the standard perturbation Lemma. 

It is convenient to break the process into two steps.

We write $\delta=\sum_{n\ge0} t^n\delta_n$, and observe that $\delta_0$ is degree one coderivation on $S(U)$ such that $\delta_0(\mathbf{1}_{S(U)})=0$: in other words, $\delta_0$ is an $L_\infty[1]$ algebra structure on $U$. We can transfer this $L_\infty[1]$ algebra structure along the contraction $(g,f, h)$ as explained in subsection \ref{subsec:Lootransfer}: first we consider the induced contraction $(S(g),S(f), \widehat{h})$ of $(S(U),\widetilde{d_U})$ onto $(S(V),\widetilde{d_V})$ as in Lemma \ref{lem:symmetrizedcontraction}, then we apply the standard perturbation Lemma \ref{lem:SPL} to the above contraction and the perturbation $\delta_{0}^+:=\delta_0-\widetilde{d_U}$ of $\widetilde{d_U}$ in order to get a transferred $L_\infty[1]$ algebra structure $\delta'_0$ on $V$, as well as a perturbed contraction $(G_0,F_0,H_0)$ of $(S(U),\delta_0)$ onto $(S(V),\delta'_0)$; moreover, $F_0$ and $G_0$ are $L_\infty[1]$ morphisms. As observed in Remark \ref{rem:H}, $(G_0,F_0,H_0)$ is both a semifull DG coalgebra contraction and a semifull algebra contraction: moreover, the perturbed homotopy $H_0$ preserves the subspaces $S_{\le n}(U)\subset S(U)$.

By abuse of notations, we continue to denote by $\delta_0$, $\delta_0'$, $(G_0,F_0,H_0)$ their extensions to $\K[[t]]$-linear differentials on $S(U)[[t]]$ and $S(V)[[t]]$ respectively and a $\K[[t]]$-linear contraction of $(S(U)[[t]],\delta_0)$ onto $(S(V)[[t]],\delta'_0)$. Finally, we apply the standard perturbation Lemma again, this time to the contraction $(G_0,F_0,H_0)$ of $(S(U)[[t]],\delta_0)$ onto $(S(V)[[t]],\delta'_0)$ and the perturbation $\delta_+:=\delta-\delta_0=\sum_{n\ge1}t^n\delta_n$ of $\delta_0$, in order to get a perturbed differential $\delta'$ on $S(V)[[t]]$ and a perturbed contraction $(G,F,H)$ of $(S(U)[[t]],\delta)$ onto $(S(V)[[t]],\delta')$. 
\begin{remark}\label{rem:HH} We notice that the perturbed differential $\delta'$ and the perturbed contraction $(G,F,H)$ are $\K[[t]]$-linear, as such are the original contraction $(G_0,F_0,H_0)$ and the perturbation $\delta_+$. As in Remark \ref{rem:SIBL}, we denote by $IBL(U):=\oplus_{n\ge0} t^nS_{\le n+1}(U)\subset S(U)[[t]]$. We notice that both $\delta_0$ and $H_0$ preserve $IBL(U)$, as they are both $\K[[t]]$-linear and preserve the subspaces $S_{\le n+1}(U)\subset S(U)$ (for $H_0$ this has already been observed, for the coderivation $\delta_0$ it follows from the formula \eqref{coderfromtaylor}). Since $\delta:S(U)[[t]]\to S(U)[[t]]$ is an $IBL_\infty[1]$ algebra structure, it also preserves the subspace $IBL(U)\subset S(U)[[t]]$, this time by Remark \ref{rem:SIBL}: hence, also the perturbation $\delta_+:=\delta-\delta_0$ preserves $IBL(U)$, as does the perturbed homotopy $H=\sum_{k\ge0}(H_0\delta_+)^kH_0$. 
\end{remark}
\begin{theorem}\label{th:transferIBL} In the above hypotheses, $\delta'$ is an $IBL_\infty[1]$ algebra structure on $V$ (with linear part $d_V$) and $F:S(V)[[t]]\to S(U)[[t]]$, $G:S(U)[[t]]\to S(V)[[t]]$ are $IBL_\infty[1]$ morphisms (with linear parts $f$ and $g$ respectively).
\end{theorem}

\begin{proof} We already observed that $(G_0,F_0,H_0)$ is a semifull DG coalgebra contraction, thus we can apply the homotopy transfer Theorem \ref{th:transfercoBV} for derived BV colagebras to deduce that $\delta'$ is an $IBL_\infty[1]$ algebra structure on $V$ and $G$ is an $IBL_\infty[1]$ morphism. 
	
In order to conclude, we still need to show that $F$ is an $IBL_\infty[1]$ morphism. We shall apply Lemma \ref{lem:comorphisms}, and in particular the equivalence (a) $\Leftrightarrow$ (c). With the notations from Remark \ref{rem:SIBL}, we denote by $IBL(U):=\oplus_{n\ge0}t^nS_{\le(n+1)}(U)\subset S(U)[[t]]$. Thus, we need to prove $\kappa(F)_i(V,\ldots,V)\subset IBL(U)$ for all $i\ge1$. We proceed by induction on $i$. For the basis of the induction, we need to check that \[\kappa(F)_1(x)=F(x):=\sum_{k\ge0}(H_0\delta_+)^kF_0(x)=\sum_{k\ge0}(H_0\delta_+)^k \big( f(x)\big)\in IBL(U)\]
for all $x\in V$, which follows since $f(x)\in U\subset IBL(U)$ and $H_0\delta_+$ preserve the subspace $IBL(U)\subset S(U)[[t]]$, according to Remark \ref{rem:HH}. 

According to Lemma \ref{lem:semifullstable}, the contraction $(G,F,H)$ is a ($\K[[t]]$-linear) semifull algebra contraction of $(S(U)[[t]],\delta)$ onto $(S(V)[[t]],\delta')$. By Proposition \ref{prop:transfer}, the cumulants $\kappa(F)_i$ are induced via homotopy transfer from the Koszul brackets $\Kos(\delta)_k$, i.e., they obey the recursion 
\[ \kappa(F)_i(x_1,\ldots,x_i) = \sum_{\stackrel{k\geq2,p_1,\ldots,p_k\geq1}{p_1+\cdots+p_k=i}}\sum_{\sigma\in S(p_1,\ldots,p_k)}\pm_K\frac{1}{k!}H\Kos(\delta)_k\left(\kappa(F)_{p_1}(x_{\sigma(1)}\ldots),\ldots,\kappa(F)_{p_k}(\ldots,x_{\sigma(i)})\right). \]
According to Remarks \ref{rem:SIBL} and \ref{rem:HH} we have that \[ H\Kos(\delta)_k\Big(IBL(U),\ldots,IBL(U)\Big)\subset H\Big(IBL(U)\Big)\subset IBL(U), \] 
which immediately implies the inductive step $\kappa(F)_i(V,\ldots,V)\subset IBL(U)$.\end{proof}

Finally, we turn our attention to the analogue of the formal Kuranishi Theorem \ref{th:kuranishi} in the context of $IBL_\infty[1]$ algebras.

\begin{definition} An $IBL_\infty[1]$ algebra $(U,\delta)$ is complete if $U$ is a complete space and all the maps $\fp_{i,j,g}$ as in the previous Remark \ref{rem:IBL} are continuous (with respect to the induced filtrations). As in Remark \ref{rem:SIBL}, we denote by $IBL(U):=\oplus_{n\ge0} t^nS_{\le n+1}(U)\subset S(U)[[t]]$, and by $\widehat{IBL}(U)$ its completion with respect to the induced filtration. According to Remark \ref{rem:SIBL}, the Koszul brackets $\Kos(\delta)_i$ induce a complete $L_\infty[1]$ algebra structure on $\widehat{IBL}(U)$. 

A \emph{Maurer-Cartan element} of the complete $IBL_\infty[1]$ algebra $(U,\delta)$ is by definition a Maurer-Cartan element of the corresponding $L_\infty[1]$ algebra $\widehat{IBL}(U)$. In other words, it is a degree zero element $x=\sum_{n\geq0}t^n x_n$, where $x_n$ is in the completion of $S_{\leq n+1}(U)$, such that 
	\[ \sum_{n\geq1}\frac{1}{n!}\Kos(\delta)_n(x,\ldots,x) = 0 \]
	(in particular, this implies that $x_0\in U^0$ is a Maurer-Cartan element of the underlying $L_\infty[1]$ algebra $(U,\delta_0)$). To distinguish it from others Maurer-Cartan sets we have introduced so far, we shall denote by $\operatorname{MC}_{IBL_\infty[1]}(U):=\operatorname{MC}_{L_\infty[1]}\Big(\widehat{IBL}(U)\Big)$ the set of Maurer-Cartan elements of the $IBL_\infty[1]$ algebra $(U,\delta)$.
	
	Given a continuous morphism of complete $IBL_\infty[1]$ algebras $f:S(U)[[t]]\to S(V)[[t]]$, according to Remark \ref{rem:SIBL} the cumulants $\kappa(f)_i$ are the Taylor coefficients of a continuous $L_\infty[1]$ morphism $\widehat{IBL}(U)\to\widehat{IBL}(V)$, hence there is an induced push-forward 
	\[ \operatorname{MC}(f):\operatorname{MC}_{IBL_\infty[1]}(U)\to\operatorname{MC}_{IBL_\infty[1]}(V):x\to\sum_{n\geq1}\frac{1}{n!}\kappa(f)_n(x,\ldots,x)\]
	between the corresponding Maurer-Cartan sets.
\end{definition}

In the hypotheses of the previous Theorem \ref{th:transferIBL}, we assume moreover that $U$ is a complete $IBL_\infty[1]$ algebra, and that $h,fg$ are continuous. Then it is not hard to check that $(V,\delta')$ is a complete $IBL_\infty[1]$ algebra with respect to the induced filtration (the only one making $f$ and $g$ continuous) and $F:(V,\delta')\to (U,\delta)$, $G:(U,\delta)\to (V,\delta')$ are continuous $IBL_\infty[1]$ morphisms. Moreover, according to Remark \ref{rem:HH} the contraction $(G,F,H)$ of $(S(U)[[t]],\delta)$ onto $(S(V)[[t]],\delta')$ restricts to a continuous contraction of $IBL(U)$ onto $IBL(V)$, hence induces a continuous contraction of $\widehat{IBL}(U)$ onto $\widehat{IBL}(V)$ which we continue to denote by $(G,F,H)$. As in the proof of the previous Theorem \ref{th:transferIBL}, the $L_\infty[1]$ algebra structure on $\widehat{IBL}(V)$ associated with the Koszul brackets $\Kos(\delta')_i$ and the $L_\infty[1]$ morphism $\widehat{IBL}(V)\to\widehat{IBL}(U)$ associated with the cumulants $\kappa(F)_i$ are induced via homotopy transfer along the contraction $(G,F,H)$ from the $L_\infty[1]$ algebra structure on $\widehat{IBL}(U)$ associated with the Koszul brackets $\Kos(\delta)_i$. Although we didn't show that the $L_\infty[1]$ morphism $\widehat{IBL}(U)\to\widehat{IBL}(V)$ associated with the cumulants $\kappa(G)_i$ is induced via homotopy transfer, this is not essential for the claim of Theorem \ref{th:kuranishi} to be true, and we only need to check that the $L_\infty[1]$ morphism $\kappa(G):\widehat{IBL}(U)\to\widehat{IBL}(V)$ is a left inverse to $\kappa(F):\widehat{IBL}(V)\to\widehat{IBL}(U)$, which follows from Lemma \ref{prop:cumulantsandcomposition} (cf. with [DELIGNE]). With these preparations, we can apply the formal Kuranishi Theorem \ref{th:kuranishi} to deduce the following result.

\begin{theorem}\label{th:IBLKur} In the above hypotheses, there are induced bijective correspondences 
		\[ \operatorname{MC}_{IBL_{\infty[1]}}(U) \cong \operatorname{MC}_{IBL_{\infty[1]}}(V) \times H\left(\widehat{IBL}(U)^0\right). \]
	\[ \operatorname{MC}_{IBL_{\infty[1]}}(V) \cong \operatorname{MC}_{IBL_{\infty[1]}}(U) \cap\operatorname{Ker}(H), \]
For the former, in one direction it sends $x\in\operatorname{MC}_{IBL_\infty[1]}(U)$ to the pair $(\operatorname{MC}(G)(x),H(x))\in \operatorname{MC}_{IBL_{\infty[1]}}(V) \times H\left(\widehat{IBL}(U)^0\right)$: the correspondence in the other direction might be defined recursively, as in the claim of Theorem \ref{th:kuranishi}. For the latter, in one direction it is given by $\operatorname{MC}(F):\operatorname{MC}_{IBL_{\infty[1]}}(V) \to \operatorname{MC}_{IBL_{\infty[1]}}(U) \cap\operatorname{Ker}(H)$, and in the other direction it sends a Maurer-Cartan element $x\in \operatorname{MC}_{IBL_{\infty[1]}}(U) \cap\operatorname{Ker}(H)$ to the one $G(x)=\operatorname{MC}(G)(x)\in\operatorname{MC}_{IBL_{\infty[1]}}(V)$.
\end{theorem}

\end{document}